\documentclass[%draft,
a4paper,reqno]{amsart}
\usepackage[T1]{fontenc}
\usepackage[english]{babel}
\usepackage{amssymb,bbm,enumerate}
\usepackage{color}
\usepackage[linktocpage=true,colorlinks=true, linkcolor=blue, citecolor=red, urlcolor=green]{hyperref}
\usepackage{graphicx}
\usepackage{float} %to make figures stay in place
\usepackage{bm}
\usepackage{xcolor}

\usepackage{tikz-cd}

\renewcommand{\ker}{\Ker}

\newcommand{\mc}[1]{\mathcal{#1}}
\newcommand{\mf}[1]{\mathfrak{#1}}
\newcommand{\mb}[1]{\mathbb{#1}}

\newcommand{\id}{\mathbbm{1}}

\DeclareMathOperator{\Hom}{Hom}
\DeclareMathOperator{\End}{End}

\DeclareMathOperator{\tr}{tr}

\DeclareMathOperator{\ad}{ad}

\DeclareMathOperator{\Ker}{Ker}
\DeclareMathOperator{\Span}{Span}

\DeclareMathOperator{\Res}{Res}

\DeclareMathOperator{\gr}{gr}

\theoremstyle{plain}
\newtheorem{theorem}{Theorem}[section]
\newtheorem{lemma}[theorem]{Lemma}
\newtheorem{proposition}[theorem]{Proposition}
\newtheorem{corollary}[theorem]{Corollary}

\theoremstyle{definition}
\newtheorem{definition}[theorem]{Definition}
\newtheorem{example}[theorem]{Example}

\theoremstyle{remark}
\newtheorem{remark}[theorem]{Remark}

\numberwithin{equation}{section}

\definecolor{light}{gray}{.9}

\usepackage{tabularx}
\usepackage{colortbl}
\usepackage{tikz}

\begin{document}

\title{Generators of the quantum finite $W$-algebras in type $A$}

\author{Alberto De Sole}
\address{Dipartimento di Matematica, Sapienza Universit\`a di Roma,
P.le Aldo Moro 2, 00185 Rome, Italy}
\email{desole@mat.uniroma1.it}
\urladdr{www1.mat.uniroma1.it/\$$\sim$\$desole}
\author{Laura Fedele}
\address{Dipartimento di Matematica, Sapienza Universit\`a di Roma,
P.le Aldo Moro 2, 00185 Rome, Italy}
\email{fedele@mat.uniroma1.it}
%
%\author{Victor G. Kac}
%\address{Dept of Mathematics, MIT,
%77 Massachusetts Avenue, Cambridge, MA 02139, USA}
%\email{kac@math.mit.edu}
%
\author{Daniele Valeri}
\address{Yau Mathematical Sciences Center, Tsinghua University, 100084 Beijing, China}
\email{daniele@math.tsinghua.edu.cn}

%\subjclass{
%Primary 17B63; 
%Secondary 17B69, 17B80, 37K30, 17B08
%}

%%%%%%%%

\begin{abstract}
We prove a conjecture proposed in \cite{DSKV16}
describing the Lax type operator $L(z)$ for the quantum finite $W$-algebras
of $\mf{gl}_N$ in terms of a PBW generating system for the $W$-algebra.
In doing so, we extend this result to an arbitrary good grading
and an arbitrary isotropic subspace of $\mf g[\frac12]$.
\end{abstract}

\maketitle

\tableofcontents

\pagestyle{plain}

%%%%%%%%%%%%%%%%%%%%%%%%%%%%%%%%%%%%%%%
\section{Introduction}\label{sec:1}

Quantum finite $W$-algebras were first introduced at the end of the $'70$s 
by Kostant \cite{Ko78}. 
They appeared naturally, 
in the special case of a principal nilpotent element $f\in\mf g$,
within the theory of Whittaker vectors and Whittaker models.
Kostant's work was soon after generalized by his student Lynch \cite{Ly79} 
to the more general case of even nilpotent elements.
In the 90's there has been a great attention by theoretical physicists 
on quantum finite $W$-algebras
and their relation to quantum affine $W$-algebras and BRST cohomology 
(cf. e.g. \cite{dBT93}, the book \cite{BS95} and references therein).
It was however only with Premet \cite{Pre02} in 2002 that a definition 
of quantum finite $W$-algebras for an arbitrary nilpotent element appeared,
and was used to prove a famous conjecture of Kac and Weisfeiler 
on modular representations of reductive Lie algebras.
Since then a growing interest has been shown towards $W$-algebras 
and their representations by the mathematical community,
one of the main reasons being the close connection between 
their representation theory and the representation theory of $\mf g$, 
as for instance in the theory of primitive ideals 
(see e.g. \cite{Lo12,Pre07,Pre07b}).
Moreover, an equivalence of categories due to Skryabin \cite{Sk02} 
relates the category of modules over the $W$-algebra 
and the category of Whittaker modules over $\mf g$.
Another important result is due to Brundan and Kleshchev \cite{BK08b},
extending (in the case of $\mf g = \mf{gl}_N$) the well-known Schur-Weyl duality 
to a duality between the $W$-algebra and the affine degenerate Hecke algebra.

\medskip

Let us briefly review the definition of the quantum finite $W$-algebra.
Let $ \mf g $ be a finite-dimensional complex reductive Lie algebra with
a non-degenerate symmetric bilinear form $ (\cdot\vert\cdot) $.
Let $ f\in\mf g $ be a nilpotent element, included, by Jacobson-Morozov Theorem,
in an $\mf{sl}_2$-triple $(f,h,e)$,
and let $\mf g=\bigoplus_{j\in\frac12\mb Z}\mf g[j]$
be the corresponding $\frac12\ad h$-eigenspace decomposition (Dynkin grading).
The \emph{quantum finite W-algebra} associated to these data is defined as
(cf. Definition \ref{def:Walg})
\begin{equation}\label{20180530:eq4}
W(\mf g,f)
:=\big(
U(\mf g)/U(\mf g)\Span\{ b-(f|b) \}_{b\in\mf g[\geq1]}
\big)^{\ad \mf g[\geq\frac12]}
\,.
\end{equation}
Note that the space $\mc I=U(\mf g)\Span\{ b-(f|b) \}_{b\in\mf g[\geq1]}$
is a left ideal of $U(\mf g)$, hence $U(\mf g)/\mc I$ does not have an algebra structure.
One the other hand, it is not hard to prove that the subspace of invariants \eqref{20180530:eq4}
has a well defined associative product, induced by that of $U(\mf g)$.

The $W$-algebra \eqref{20180530:eq4} can be thought of as a quantum Hamiltonian reduction
of the associative algebra $U(\mf g)$,
and it is a quantization of the Poisson algebra of functions
on the Slodowy slice \cite{Slo80}.
In literature there have been several definitions of the 
quantum finite $W$-algebra, all proved equivalent \cite{DCSHK},
and various generalizations.
In particular the Dynkin grading can be replaced by an arbitrary good grading for $ f\in\mf g$,
and the nilpotent subalgebras $\mf g[\geq1]$ and $\mf g[\geq\frac12]\subset\mf g$ 
in \eqref{20180530:eq4} can be replaced by their extensions associated to a certain
subspace $\mf l\subset\mf g[\frac12]$
isotropic with respect to the bilinear form $\omega=(f|[\cdot\,,\,\cdot])$.
We shall then use the notation $W(\mf g,f)=W(\mf g,f,\Gamma,\mf l)$
to keep track of the dependence on the good grading $\Gamma$
and the isotropic subspace $\mf l\subset\mf g[\frac12]$.
On the other hand,
by the results of Gan and Ginzburg \cite{GG02}, 
and Brundan and Goodwin \cite{BG05} (described in detail in Section \ref{sec:6}), 
the $W$-algebras associated to the various choices of $\Gamma$ and $\mf l$
are all isomorphic,
and the $W$-algebra ultimately only depends on the nilpotent orbit of $f$ in $\mf g$.

In the special case when $f$ is the zero nilpotent 
the $W$-algebra \eqref{20180530:eq4} coincides with the enveloping algebra $U(\mf g)$.
On the other hand, by a result of Kostant-Kazhdan, 
the $W$-algebra for a principal nilpotent $f$ is isomorphic 
to the center of the enveloping algebra:
$W(\mf g,f^{\rm{pr}})\simeq\mc Z(U(\mf g))$.
The isomorphism is explicitly described as 
the restriction of the quotient map 
$U(\mf g)\twoheadrightarrow U(\mf g)/\mc I$.
In this case, we can therefore compute the generators of the $W$-algebra 
by first computing generators for $\mc Z(U(\mf g))$, 
using for instance the Capelli determinant,
and then computing their images in the quotient $U(\mf g)/\mc I$.

In-between these extreme nilpotent cases lie all other $ W$-algebras, 
whose structure is in general hard to describe explicitly.
In fact, one of the main problems in the theory of $W$-algebras
is to find an explicit formula for a collection of PBW generators
(and the commutation relations between them).
Indeed, by a Theorem of Premet \cite{Pre02}
(cf. Theorem \ref{thm:Premetconds} below),
there exist a finite collection of PBW generators for the $W$-algebra 
parametrized by a basis of the centralizer of $f$ in $\mf g$,
but an explicit formula for these generators is, in general, not known.
In the case of the Lie algebra $\mf{gl}_N$ and its arbitrary nilpotent element $f$
Brundan and Kleshchev in \cite{BK06,BK08}
found a recursive formula for an operator $W(z)$
containing all the generators of the $W$-algebra.
In the present paper
we give a different description of the same recursive formula for $W(z)$
and give an alternative proof of its lying in the $W$-algebra.

\medskip

Let us describe explicitly the above mentioned recursive formula 
for the generators of $W(\mf{gl}_N,f)$.
The nilpotent element $f$ is associated to a partition $p_1^{r_1}p_2^{r_2}\dots p_s^{r_s}$
of $N=\sum_ip_ir_i$ with $r=r_1+\dots+r_s$ parts.
A good grading $\Gamma$ for $f$ is associated to a pyramid $p$
of $N$ boxes disposed in rectangles of sizes $p_i\times r_i$, $i=1,\dots,s$
(cf. Section \ref{sec:3.1}).
One example is the right aligned pyramid, corresponding to the good grading $\Gamma_R$.
We number the boxes of the pyramid from $1$ to $N$
going first bottom to top and then right to left,
and we depict an elementary matrix $e_{ij}\in\mf{gl}_N$ as an arrow
going from box $j$ to box $i$.
As explained in Section \ref{sec:3.2},
a basis of the centralizer of $f$ in $\mf g$ is indexed by
triples $(i,j;\ell)$, with $i,j=1,\dots,r_1$ 
(numbering the rows of the pyramid bottom to top) 
and $\ell=0,\dots,\min\{q_i,q_j\}-1$,
where $q_i$ and $q_j$ are the lengths of the rows $i$ and $j$ respectively
(i.e., $q_1=\dots=q_{r_1}=p_1$, $q_{r_1+1}=\dots=q_{r_1+r_2}=p_2$, and so on).
Hence, according to Premet's Theorem,
we want to construct a collection
$$
\big\{w_{i,j;\ell}\,\big|\,i,j=1,\dots,r;\,\ell=0,\dots,\min\{q_i,q_j\}-1\big\}
\,,
$$
of PBW generators of $W(\mf g,f)$
satisfying certain ``Premet's conditions'' (see Definition \ref{def:premet}).
For the grading $\Gamma_R$
we shall construct, recursively, an $r\times r$ matrix 
$W(z)=\big(w_{i,j}(z)\big)_{i,j=1}^{r}$ encoding all the generators of the $W$-algebra 
$W(\mf{gl}_N,f)$, viewed as a subspace of $U(\mf g)/\mc I$,
as follows:
\begin{equation}\label{eq:intro1}
w_{i,j}(z)
=
-(-z)^{q_i}\delta_{ij}
+
\sum_{\ell=0}^{\min\{q_i,q_j\}-1}(-z)^\ell w_{i,j;\ell}
\,.
\end{equation}
The base step in the recursion is $p_1=1$, $r_1=r=N$,
when the pyramid consists of a single column.
In this case we set
$$
w_{i,j}(z)=z\delta_{ij}+e_{ji}
\,\,,\,\,\,\,i,j=1,\dots,N
\,,
$$
where $e_{ji}\in\mf{gl}_N$ is the usual elementary matrix.
For $p_1>1$, erasing the leftmost column from the pyramid $p$,
we obtain a smaller pyramid $p'$ of $N-r_1$ boxes,
which is associated to a good grading $\Gamma'$ for a nilpotent element $f'$
in the Lie subalgebra $\mf g'\simeq \mf{gl}_{N-r_1}$.
The recursive formula \eqref{eq:recW} for the matrix elements $w_{i,j}(z)$,
in the present notation, reads:
\begin{equation}\label{eq:intro2}
\begin{split}
w_{i,j}(z)
&=
w'_{i,j}(z) 
\,\,\text{ if }\,\, j>r_1\,,
\\
w_{i,j}(z)
&=
-\frac1{r_1}\sum_{h=1}^{r_1}
[w'_{i,h}(z),e_{N-r_1+j,N-r_1-r'_1+h}]
\\
&
-\sum_{h=1}^{r_1}
w'_{i,h}(z)(z\delta_{hj}+e_{N-r_1+j,N-r_1+h}-(N-r_1)\delta_{h,j})
\\
& +\sum_{h=r_1+1}^r
w'_{i,h}(z)w'_{h,j;p_h-1}
\,\,\text{ if }\,\, j\leq r_1
\,,
\end{split}
\end{equation}
where $W'(z)=\big(w'_{i,j}(z)\big)_{i,j=1}^r$
is the matrix of generators for the smaller pyramid $p'$.
In the first term in the RHS of \eqref{eq:intro2}
$r'_1$ denotes the height of the shortest column of the pyramid $p'$
(and it is equal to $r_1$ if $p_1>p_2+1$, while it is equal to $r_2$ if $p_1=p_2+1$).
Equation \eqref{eq:intro2} has to be interpreted as follows:
given the elements $w'_{i,j;\ell}\in W(\mf g',f')\subset U(\mf g')/\mc I'$,
we take their representatives in $U(\mf g')$
and the RHS of \eqref{eq:intro2}, viewed as an element of 
$W(\mf g,f)\subset U(\mf g)/\mc I$, is independent of the choice of representatives.

The above recursion \eqref{eq:intro2}
already appeared, with a different notation, in \cite[Lem.3.4]{BK08}.
They obtain that the coefficients $w_{i,j;\ell}$ lie in the $W$-algebra
through an involved construction based on the Gauss decomposition
and the relation between $W$-algebras and the Yangian.
In the present paper, in particular in Theorem \ref{thm:20180208},
we give an alternative,  more direct proof of this fact.

\medskip

Something very similar to the matrix $W(z)$
is the Lax type operator $L(z)$ constructed in \cite{DSKV16,DSKV17}.
This is an $r_1\times r_1$ matrix which is supposed to encode 
the whole structure of the $W$-algebra.
Indeed, the commutation relations among the coefficients of the operator $L(z)$
can be written in terms of a Yangian type identity
(cf. equation \eqref{eq:yang}):
\begin{equation}\label{20180530:eq1}
(z-w)[L_{ij}(z),L_{hk}(w)]
=
L_{hj}(w) L_{ik}(z)-L_{hj}(z) L_{ik}(w)
\,\,,\,\,\,\,
i,j,h,k=1,\dots,r_1\,.
\end{equation}
Equation \eqref{20180530:eq1} is called Yangian identity 
since it coincides (up to a sign) to the defining relation of the Yangian of $\mf{gl}_{r_1}$.

The relationship between $W$-algebras and Yangians (in the case of $\mf g=\mf{gl}_N$ 
and a rectangular nilpotent element) had been highlighted by many authors, 
as for instance Drinfeld \cite{Dr88} or Ragoucy and Sorba \cite{RS99}.
However, the most significant contribution comes from the work of Brundan and Kleshchev who, 
in \cite{BK05}-\cite{BK08}, deeply study the $W$-algebras $ W(\mf{gl}_N,f)$ and establish 
an isomorphism between a certain subquotient of a Yangian $Y_r$ 
(the truncated shifted Yangian) and $W(\mf{gl}_N,f)$.

\medskip

The main weakness of the construction in \cite{DSKV16,DSKV17}
is that it is not clear if, and in which sense, the matrix $L(z)$
encodes all the generators of the $W$-algebra.
In the simpler case of classical affine $W$-algebras,
analogous constructions were already developed by the same authors
\cite{DSKV13,DSKV16a,DSKV17a}.
In this case, the theory was better understood,
and the authors were able to prove that the pseudodifferential Lax operator $L(\partial)$
is obtained as a quasideterminant of an explicitly described operator $W(\partial)$
encoding all the generators of the classical affine $W$-algebras
(for $\mf{gl}_N$ in \cite{DSKV16a}, and for all classical Lie algebras in \cite{DSKV17a}).

The analogous result was only conjectured in \cite{DSKV16}
for the quantum finite $W$-algebras of $\mf{gl}_N$
(and it was only tested for some very special choice of nilpotent $f$).
The main result of the present paper is the proof of that conjecture.
In fact, the matrix $W(z)$ defined by the recursion \eqref{eq:intro2}
solves, in the case of the right aligned pyramid, 
the issue raised in \cite{DSKV16},
as we have
\begin{equation}\label{eq:intro3}
L(z)=\big|W(z)\big|
\,,
\end{equation}
where the quasideterminant on the RHS is with respect to the first $r_1$ rows and columns.
This is stated in Theorem \ref{thm:main} and it is proved 
in Sections \ref{sec:5a}--\ref{sec:5d}.
In section \ref{sec:6} the same result is then generalized to an arbitrary good grading.

\medskip

The paper is organized as follows.
In Section \ref{sec:2} we review the definition of a quantum finite $W$-algebra,
its Kazhdan filtration,
and we state Premet's Theorem \ref{thm:Premetconds}
on the generators of the $W$-algebra.

In Section \ref{sec:3} we introduce all the notation that will be used throughout the paper.
In particular, we study the centralizer of a nilpotent element $f$ in $\mf{gl}_N$,
and we discuss the combinatorics of partitions and pyramids associated to $f$
and its good gradings.

In Section \ref{sec:4} we recall the notion of generalized quasideterminant \eqref{eq:quasidet}.
Using this notion, we review the definition of the operator $L(z)$ defined in \cite{DSKV16},
and we state the main result, Theorem \ref{thm:main},
claiming that equation \eqref{eq:intro3} holds for some operator $W(z)$
encoding all Premet's generators for the $W$-algebra.

The proof of Theorem \ref{thm:main}
is then the content of the following 5 sections.
In particular, Sections \ref{sec:5a}--\ref{sec:5d} 
are devoted to the special case when the good grading $\Gamma$ 
is associated to a right aligned pyramid.
After introducing, in Section \ref{sec:5a}, some notation and useful lemmas,
we provide, in Section \ref{sec:5b} a recursive formula for the operator $L(z)$,
which will be used for the inductive proof of Theorem \ref{thm:main}.
In Section \ref{sec:5c} we then define, still for a right aligned pyramid, the operator $W(z)$
via the recursive formula \eqref{eq:recW},
we prove that its coefficients lie in the $W$-algebra (Theorem \ref{thm:20180208}),
and we describe its form in terms of Premet's generators (Proposition \ref{20180405:prop2}).
All of this will be needed to prove, in Section \ref{sec:5d}, Theorem \ref{thm:main}
for the right aligned pyramid.

In Section \ref{sec:6} we extend our results to an arbitrary good grading $\Gamma$.
In order to do so,
we build a chain of adjacent good gradings for $f$ connecting $\Gamma$ 
to the right aligned grading $\Gamma_R$.
We then use the isomorphisms of \cite{GG02} and \cite{BG05} 
to describe the operator $W(z)$ and deduce Theorem \ref{thm:main}
from the special case of $\Gamma_R$.

Finally, in Section \ref{sec:8} we discuss the example of a pyramid with two columns
and we compute, in this case, the operator $W(z)$.

Throughout the paper the base field $\mb F$ is a a field of characteristic $0$.

\medskip

\noindent
{\bf Acknowledgements:}
We wish to thank Victor Kac for useful discussions
and Alexander Kleshchev for comments on the Ph.D. thesis
of the second author, from which this work originated.
The research of this paper was partially conducted during the visits
of all three authors to IHES,
of the first two authors to MIT,
and of the third author to the University of Rome La Sapienza.
We are grateful to all these institutions for their kind hospitality.
The first author was partially supported by the national PRIN grant ``Moduli and Lie Theory''
and University grants, 
the second author was supported by the University grant ``Avvio alla ricerca'' 2016,
and the third author was supported by a Tshinghua University startup research grant.

%%%%%%%%%%%%%%%%%%%%%%%%%%%%%%%%%%%%%%%
\section{Quantum finite \texorpdfstring{$W$}{W}-algebras}
\label{sec:2}

%%%
\subsection{Definition of $W(\mf g,f)$}
\label{sec:2.1}

Let $ \mf{g}$ be a reductive finite dimensional Lie algebra,
endowed with a non-degenerate symmetric bilinear form $(\cdot\,|\,\cdot)$.
Let  $ f \in \mf{g} $ be a nilpotent element.
Recall that \cite{EK05} a \textit{good grading} for $ f $ 
is a $ \frac{1}{2} \mb Z$-grading of $ \mf{g} $, 
\begin{equation}\label{eq:gamma}
\Gamma:\,\, \mf{g} 
= \bigoplus_{j \in \frac{1}{2} \mb Z} \mf{g}^\Gamma[j]
= \bigoplus_{j \in \frac{1}{2} \mb Z} \mf{g}[j]
\end{equation}
such that
$f \in \mf{g}[-1]$,
$\text{ad }f: \mf{g}[j] \longrightarrow \mf{g}[j-1]$ 
is injective for $ j \geq \frac{1}{2} $,
and $ \text{ad }f: \mf{g}[j] \longrightarrow \mf{g}[j-1]$ 
is surjective for $ j \leq \frac{1}{2} $.
In particular, if $ \Gamma $ is a good grading for $ f $, then $\text{ad} f$
restricts to a bijection $\mf{g}[\frac{1}{2}] 
\stackrel{\sim}{\longrightarrow} \mf{g}[-\frac{1}{2}]$.
We thus have a non-degenerate skewsymmetric bilinear form on $\mf g[\frac12]$ 
given by $\omega(a,b) = (f|[a,b])$.
Furthermore, let $\mf l\subset\mf g[\frac12]$ be an isotropic subspace w.r.t. this skewsymmetric form,
and let $\mf l\subset \mf l^\perp$ be its orthocomplement.
We shall consider the following two nilpotent subalgebras of $\mf g$:
\begin{equation}\label{eq:mn}
\mf m=\mf l\oplus\mf g[\geq1]
\,\subset\,
\mf n=\mf l^\perp\oplus\mf g[\geq1]
\,\subset\,\mf g
\,.
\end{equation}
Here and further we denote $\mf g[\geq1]=\oplus_{j\geq1}\mf g[j]$.

In \eqref{eq:gamma} and further on, when confusion may not arise,
we omit the superscript $\Gamma$ from the graded spaces $\mf g^\Gamma[j]=\mf g[j]$.
However, in Section \ref{sec:6}, where we will need to use several gradings
at the same time, we will recover the full notation $\mf g^\Gamma[j]$.

Attached to the Lie algebra $\mf g$, its nilpotent element $f$, the good grading $\Gamma$ 
and the isotropic subspace $\mf l\subset\mf g[\frac12]$,
we have the corresponding $W$-algebra:
\begin{definition}\label{def:Walg}
The quantum finite $W$-algebra $W(\mf g,f)=W(\mf g,f,\Gamma,\mf l)$ is
\begin{equation}\label{eq:Walg}
W(\mf g,f)
:=\big(U(\mf g)/\mc I\big)^{\ad \mf n}
\,\,,\text{ where }\,\,
\mc I=U(\mf g)\Span\{b-(f|b)\}_{b\in\mf m}
\,.
\end{equation}
\end{definition}
The quotient $U(\mf g)/\mc I$ is a cyclic left $\mf g$-module
generated by the element $\bar 1$, the image of $1\in U(\mf g)$.
Note that, while $U(\mf g)/\mc I$ does not have any natural algebra structure,
the subspace of $\ad\mf n$-invariants does,
and this structure is induced by that of $U(\mf g)$.
Recall that, by the results of \cite{BG05} and \cite{GG02},
the isomorphism class of $W(\mf g,f)$
does not depend on the choice of the good grading $\Gamma$
and of the isotropic subspace $\mf l$,
and it only depends on the nilpotent orbit of $f$.

%%%
\subsection{Kazhdan filtration}
\label{sec:2.2}

Recall that, associated to the good grading $\Gamma$,
we have the \emph{Kazhdan filtration} 
of $U(\mf g)$, defined by
\begin{equation}\label{eq:kaz}
F^\Gamma_{\Delta}U(\mf g)
=
F_{\Delta}U(\mf g)
=
\sum_{s-j_1-\ldots-j_s\leq\Delta}\mf g[j_1]\dots\mf g[j_s]
\,\,,\,\,\,\,
\Delta\in\frac12\mb Z
\,.
\end{equation}
In \eqref{eq:kaz} and further on, when confusion may not arise,
we omit the superscript $\Gamma$ from the filtered spaces 
$F^\Gamma_\Delta=F_\Delta$.
However, in Section \ref{sec:6}, where we will need to use several gradings
at the same time, we will recover the full notation $F^\Gamma_\Delta$.

The filtration \eqref{eq:kaz} is clearly an increasing algebra filtration, 
and the corresponding associated graded is a Poisson algebra
which, as a commutative associative algebra, is
$\gr U(\mf g)=S(\mf g)$,
graded by the Kazhdan degree (or \emph{conformal weight}).

We have the induced Kazhdan filtration 
of the $W$-algebra $W(\mf g,f)$:
\begin{equation}\label{eq:kazW}
F_{\Delta}W(\mf g,f)
=
W(\mf g,f)\cap \big(F_\Delta U(\mf g)/ F_\Delta \mc I\big)
\,\,,\,\,\,\,
\Delta\in\frac12\mb Z_+
\,.
\end{equation}
Note that, since $\mf g[>1]\subset \mc I$, 
we have $F_{\Delta}W(\mf g,f)=0$ for $\Delta<0$.
The associated graded is isomorphic to the classical finite $W$-algebra:
\begin{equation}\label{eq:clfinW}
\gr W(\mf g,f)
\simeq
W^{cl}(\mf g,f)
:=
\big(S(\mf g)\big/S(\mf g)\Span\{b-(f|b)\}_{b\in\mf m}\big)^{\ad\mf n} 
\,.
\end{equation}

%%%
\subsection{Premet's generators}
\label{sec:2.2b}

Fix a subspace $ U \subset \mf{g}  $ complementary to $ [f, \mf{g}] $ 
and compatible with the $ \frac{1}{2} \mb Z $-grading. 
For example, we could take $ U = \mf{g}^e $, the Slodowy slice. 
By the non-degeneracy of  $ (\cdot \, | \, \cdot ) $, the orthocomplement 
to $ [f, \mf{g}] $ is $ \mf{g}^f $, hence 
\begin{equation}\label{eq:decompgl12}
\mf{g} 
= 
[f,\mf g]\oplus U
=
\mf{g}^f\oplus U^\perp
\,.
\end{equation}
Let $\pi^f:\,\mf g\twoheadrightarrow\mf g^f$
be the projection with kernel $U^\perp$,
and let $\eta^f:\,S(\mf g)\twoheadrightarrow S(\mf g^f)$ be the
algebra homomorphism defined by
\begin{equation}\label{eq:projectionpif}
\eta^f(a)=\pi^f(a)+(f|a)
\,\,,\,\,\,\,\text{ for }\, a\in\mf g
\,.
\end{equation}
By assumption the subspace $U$ is compatible with the $\Gamma$-grading.
Hence, since for a good grading $\mf g^f\subset\mf g[\leq0]$,
we have $\pi^f(\mf g[\geq\frac12])=0$.
It follows that $\eta^f(b-(f|b))=0$ for every $b\in\mf m$,
and therefore $\eta^f$ induces a surjective algebra homomorphism
\begin{equation}\label{diag:commut-pi_gf}
\eta^f\,:\,\,W^{cl}(\mf g,f)\to S(\mf g^f)
\,.
\end{equation}

In \cite{Pre02}, Premet constructs a PBW basis for $ W(\mf{g}, f)$
by ``inverting'' the map $\eta^f$.
The result was stated in the particular case of a Dynkin grading $\Gamma$
and a Lagrangian subspace $\mf l\subset\mf g[\frac12]$,
but the same proof should work in general.
\begin{theorem}\cite[Theorem 4.6]{Pre02}\label{thm:Premetconds}
There exists a (non-unique) linear map
\begin{equation}
w: \mf{g}^f \longrightarrow W(\mf{g}, f)
\end{equation}
such that
$w(a) \in F_{\Delta}W(\mf{g}, f)$
and
$\eta^f(\gr_\Delta(w(a)) = a$
for $a\in\mf{g}^f[1-\Delta]$.
Moreover, given such a map $w$,
if $ \{x_i\}_{i=1}^{\dim \mf{g}^f} $ is an ordered basis of $ \mf{g}^f$ 
homogeneous w.r.t. the $\Gamma$-grading, $x_i \in \mf{g}^f[1 - \Delta_i]$,
then $\{w(x_i)\}_{i=1}^{\dim \mf{g}^f}$ form a PBW generating set of the $W$-algebra,
in the sense that then the monomials
$$
w(x_{i_1})\dots w(x_{i_k}) 
\,\,,\,\,\,\,
k\geq0,\,
1 \leq i_1 \leq \cdots \leq i_k \leq \dim \mf{g}^f,\, 
\Delta_{i_1} + \ldots + \Delta_{i_k} \leq \Delta
\,,
$$
form a basis of $ F_\Delta W(\mf{g}, f)$.
\end{theorem}
\begin{definition}\label{def:premet}
A \emph{Premet map} is an injective linear map $w:\,\mf g^f\to W(\mf g,f)$
satisfying the conditions of Theorem \ref{thm:Premetconds}, 
namely such that, for $a\in\mf g^f[1-\Delta]$,
\begin{enumerate}[(i)]
\item
$w(a) \in F_{\Delta}W(\mf{g}, f)$,
\item
$\eta^f(\gr_\Delta(w(a))) = a$.
\end{enumerate}
\end{definition}

%%%%%%%%%%%%%%%%%%%%%%%%%%%%%%%%%%%%%%%
\section{Partitions, pyramids, bases and so on}\label{sec:3}

In the present paper we shall only consider $W$-algebras associated
to the Lie algebra $\mf{gl}_N$.

%%%
\subsection{Partitions and pyramids associated to a nilpotent orbit}
\label{sec:3.1}

We let $V$ be a vector space of dimension $N$ and we let $\mf g=\End V$,
endowed with the trace form $(A|B)=\tr(AB)$.
We also fix a nilpotent element $f\in\mf g$,
which is associated to a partition of $N$:
\begin{equation}\label{eq:p}
\lambda=p_1^{r_1}p_2^{r_2}\dots p_s^{r_s}
\,,
\end{equation}
with $p_1>\dots>p_s$, $r_1,\dots,r_s\geq1$,
and $N=r_1p_1+\dots+r_sp_s$.
We also let $r=r_1+\dots+r_s$ the number of parts of the partition.

As in Section \ref{sec:2.1},
we fix a good grading $\Gamma$ for $f$ \eqref{eq:gamma}.
Recall \cite{Wang11} that there exists a semisimple element $h_\Gamma$
such that the $\Gamma$-degrees coincide with the $\ad(h_\Gamma)$-eigenvalues.
Note also that, as a consequence of the Jacobson-Morozov Theorem,
we can construct an $\mf{sl}_2$-triple $(f,h_\Sigma,e_\Sigma)$ of $\mf g$
containing $f$, which is compatible with the grading $\Gamma$,
in the sense that $h_\Sigma\in\mf g[0]$ and $e_\Sigma\in\mf g[1]$.
We thus have the corresponding Dynkin grading by $\frac12\ad(h_\Sigma)$-eigenspaces:
\begin{equation}\label{eq:Dgrading}
\Sigma: \,\,\mf{g} = \bigoplus_{j \in \frac{1}{2} \mb Z} \mf{g}^\Sigma[j]
\end{equation}

The action of the semisimple elements $h_\Gamma$ and $\frac12 h_\Sigma$ on $V$
provides us with the corresponding $\Gamma$ and Dynkin-gradings of $V$:
\begin{equation}\label{eq:Vgrading}
V 
= \bigoplus_{j \in \frac{1}{2} \mb Z} V[j]
= \bigoplus_{j \in \frac{1}{2} \mb Z} V^\Sigma[j]
\,.
\end{equation}
By the theory of good gradings \cite{EK05},
good gradings for $\mf{gl}_N$ are 
associated to the pyramids of shape $\lambda$.
For example,
for the partition $\lambda=3^221$ of $9$,
we can have the pyramid justified to the right:
\begin{figure}[H]
\setlength{\unitlength}{0.15in}
% selecting unit length
\centering
\begin{picture}(14,11)

\put(7,9){\framebox(2,2){}}

\put(5,7){\framebox(2,2){}}
\put(7,7){\framebox(2,2){}}

\put(3,5){\framebox(2,2){}}
\put(5,5){\framebox(2,2){}}
\put(7,5){\framebox(2,2){}}

\put(3,3){\framebox(2,2){}}
\put(5,3){\framebox(2,2){}}
\put(7,3){\framebox(2,2){}}

\put(2,2){\vector(1,0){9}}
\put(10,1){$x$}

\put(4,1.6){\line(0,1){0.8}}
\put(5,1.8){\line(0,1){0.4}}
\put(6,1.6){\line(0,1){0.8}}
\put(7,1.8){\line(0,1){0.4}}
\put(8,1.6){\line(0,1){0.8}}

\put(5.7,0.6){0}
\put(3.5,0.6){-1}
\put(7.8,0.6){1}

\end{picture}
\caption{}\label{fig1}
\end{figure}
\noindent
Each box corresponds to a basis element of $V$,
and the $x$-coordinates of the center of each box is the corresponding $\Gamma$-degree.
The Dynkin grading for the same partition corresponds to the symmetric pyramid:
\begin{figure}[H]
\setlength{\unitlength}{0.15in}
% selecting unit length
\centering
\begin{picture}(14,11)

\put(5,9){\framebox(2,2){}}

\put(4,7){\framebox(2,2){}}
\put(6,7){\framebox(2,2){}}

\put(3,5){\framebox(2,2){}}
\put(5,5){\framebox(2,2){}}
\put(7,5){\framebox(2,2){}}

\put(3,3){\framebox(2,2){}}
\put(5,3){\framebox(2,2){}}
\put(7,3){\framebox(2,2){}}

\put(2,2){\vector(1,0){9}}
\put(10,1){$x$}

\put(4,1.6){\line(0,1){0.8}}
\put(5,1.8){\line(0,1){0.4}}
\put(6,1.6){\line(0,1){0.8}}
\put(7,1.8){\line(0,1){0.4}}
\put(8,1.6){\line(0,1){0.8}}

\put(5.7,0.6){0}
\put(3.5,0.6){-1}
\put(7.8,0.6){1}

\end{picture}
\caption{}\label{fig2}
\end{figure}

%%%
\subsection{Some maps and decompositions}
\label{sec:3.1b}

As described in the previous Section \ref{sec:3.1},
the vector space $V$ is depicted by a pyramid,
in the sense that the boxes of the pyramid correspond to basis elements of $V$.
With this pictorial description,
the endomorphism $F\in\End V$ 
(corresponding to the nilpotent element $f\in\mf g$,
see the notation described in Section \ref{sec:3.4} below)
acts as a shift to the left (by one).
We also denote by $F^t$ the ``transpose'' of $F$,
which acts as the shift to the right (again by one).
The span of the rightmost (resp. leftmost) boxes of the pyramid
is then 
\begin{equation}\label{eq:V+-}
V_+ = \ker (F^t)
\qquad
\Big(\text{ resp. }\,\, V_-=\ker(F) \,\,\Big)
\,.
\end{equation}
Note that $V_+$ and $V_-$ both have dimension equal to $r$.

Consider the following direct sum decomposition of the space $V$:
\begin{equation}\label{eq:Vdu}
V=V^d\oplus V^u\,,
\end{equation}
where $V^d$ is the subspace corresponding to the bottom rectangle of the pyramid
(of height $r_1$ and base $p_1$),
while $V^u$ is the subspace corresponding to the rest of the pyramid.
In the Dynkin grading, $V^d$ is the isotipic component with respect to the $\mf{sl}_2$
action of highest weight $p_1-1$.
For example, for the diagram of Figure \ref{fig2},
the subspaces $V^d$ and $V^u$ are depicted in gray and green respectively
in the figure below:
\begin{figure}[H]
\setlength{\unitlength}{0.15in}
\setlength\fboxsep{0pt}
% selecting unit length
\centering
\begin{picture}(14,11)

\put(5,9){\colorbox{teal}{\framebox(2,2){}}}

\put(4,7){\colorbox{teal}{\framebox(2,2){}}}
\put(6,7){\colorbox{teal}{\framebox(2,2){}}}

\put(3,5){\colorbox{lightgray}{\framebox(2,2){}}}
\put(5,5){\colorbox{lightgray}{\framebox(2,2){}}}
\put(7,5){\colorbox{lightgray}{\framebox(2,2){}}}

\put(3,3){\colorbox{lightgray}{\framebox(2,2){}}}
\put(5,3){\colorbox{lightgray}{\framebox(2,2){}}}
\put(7,3){\colorbox{lightgray}{\framebox(2,2){}}}

\put(2,2){\vector(1,0){9}}
\put(10,1){$x$}

\put(4,1.6){\line(0,1){0.8}}
\put(5,1.8){\line(0,1){0.4}}
\put(6,1.6){\line(0,1){0.8}}
\put(7,1.8){\line(0,1){0.4}}
\put(8,1.6){\line(0,1){0.8}}

\put(5.7,0.6){0}
\put(3.5,0.6){-1}
\put(7.8,0.6){1}

\put(11,4.5){$V^d$}
\put(11,8.5){$V^u$}

\end{picture}
\caption{}\label{fig2b}
\end{figure}

As it is apparent by the discussion and the pictures of Section \ref{sec:3.1},
the largest and lowest degrees for both the $\Gamma$ grading and the Dynkin grading of $V$
are $\pm\frac {p_1-1}2$,
while the largest and lowest degrees for the $\Gamma$ grading and the Dynkin grading of $\mf g$
are $\pm (p_1-1)$.
%In fact, $\dim(\mf g^\Sigma[\pm d])=r_1^2$
%while, for the pyramid justified as in figure \ref{fig1},
%we have $\dim(\mf g[\pm d])=r_1r$.
%
Moreover, we have 
\begin{equation}\label{20180215:eq3}
V^d_{\pm} := V^d\cap V_{\pm}
=V^\Sigma[\pm\frac {p_1-1}2]\subseteq V[\pm\frac {p_1-1}2]
\,,
\end{equation}
(at least one of the inclusion being an equality), and
$$
V^\Sigma[\neq\pm\frac {p_1-1}2]\supseteq V[\neq\pm\frac {p_1-1}2]
$$
We also denote $V^u_{\pm}:=V^u\cap V_{\pm}$.
Obviously, we have
\begin{equation}\label{eq:V_-ud}
V_{\pm}=V_{\pm}^d\oplus V_{\pm}^u
\,.
\end{equation}

We also have the following decompositions:
\begin{equation}\label{eq:V-k}
V_{-} := \bigoplus_{k=0}^{p_1-1}
\Big(V_-\cap F^kV_{+}\Big)
\,\,\text{ and }\,\,
V_{+} := \bigoplus_{k=0}^{p_1-1}
\Big(V_+\cap (F^t)^kV_{-}\Big)
\,.
\end{equation}
Pictorially, the meaning of the decomposition \eqref{eq:V-k} is clear:
$V_-\cap F^kV_+$ is the span of the leftmost boxes of a pyramid
in rows of length $k+1$.
Similarly,
$V_+\cap (F^t)^kV_-$ is the span of the rightmost boxes of a pyramid
in rows of length $k+1$.
In particular, we have 
\begin{equation}\label{20180507:eq4}
\begin{array}{l}
\displaystyle{
\vphantom{\Big(}
V_-^d=V_-\cap F^{p_1-1}V_+
\,\,,\,\,\,\,
V_-^u=\bigoplus_{k=0}^{p_1-2}V_-\cap F^{k}V_+
\,,} \\
\displaystyle{
\vphantom{\Big(}
V_+^d=V_+\cap(F^t)^{p_1-1}V_-
\,\,,\,\,\,\,
V_+^u=\bigoplus_{k=0}^{p_1-2}V_+\cap (F^t)^{k}V_-
\,.}
\end{array}
\end{equation}

%%%
\subsection{The ``identity'' notation}
\label{sec:3.1c}

Throughout the paper, given a subspace $U\subset V$,
together with a ``natural'' splitting $V=U\oplus W$
(usually associated with the grading of $V$),
we shall denote, with a slight abuse of notation,
by $\id_U$  
both the identity map $U\stackrel{\sim}{\longrightarrow}U$,
the inclusion map $U\hookrightarrow V$,
and the projection map (with kernel $W$) $V\twoheadrightarrow U$;
the correct meaning of $\id_U$ should then be clear from the context.
Likewise,
if we further have a subspace $U_1\subset U$ with a ``natural'' splitting $U=U_1\oplus W_1$,
the same symbol $\id_{U_1}$
can mean not only the three maps 
identity $U_1\stackrel{\sim}{\longrightarrow}U_1$,
inclusion $U_1\hookrightarrow V$,
and projection (with kernel $W_1\oplus W$) $V\twoheadrightarrow U_1$,
but also 
the inclusion $U_1\hookrightarrow U$,
and projection (with kernel $W_1$) $U\twoheadrightarrow U_1$;
again, the correct meaning of $\id_{U_1}$ should then be clear from the context.
For example, since $V_{\pm}\subset V$ come with the natural splittings
$V=V_+\oplus FV=V_-\oplus F^tV$,
we shall denote the identity on $V_{\pm}$,
the inclusion $V_{\pm}\hookrightarrow V$
and the projection $V\twoheadrightarrow V_{\pm}$
(with kernel $FV$ and $F^tV$ respectively),
all by the same symbol $\id_{V_{\pm}}$.

Using the above notation, we clearly have
\begin{equation}\label{20180219:eq3}
FF^t=\id_V-\id_{V_+}=\id_{FV}
\,\,\text{ and }\,\,
F^tF=\id_V-\id_{V_-}=\id_{F^tV}
\,.
\end{equation}

%%%
\subsection{The centralizer $\mf g^f$}
\label{sec:3.2}

Recalling the definition \eqref{eq:V+-} of $V_{\pm}$,
we can decompose 
\begin{equation}\label{eq:V+-2}
V
=\bigoplus_{j=0}^{p_1-1}F^j V_+
=\bigoplus_{j=0}^{p_1-1}(F^t)^j V_-
\,,
\end{equation}
and, consequently,
$$
\mf g=\End(V)=
\bigoplus_{i,j=0}^{p_1-1}\Hom(F^i V_+,(F^t)^j V_-)
\,.
$$

We have the corresponding decomposition of the centralizer of $f$ in $\mf g$ as
\begin{equation}\label{eq:central}
\mf g^f=\bigoplus_{k=0}^{p_1-1}\mf g^f_k
\,\,\text{ where }\,\,
\mf g^f_k\subset
\bigoplus_{i=0}^{k}\Hom(F^i V_+,(F^t)^{k-i} V_-)
\,.
\end{equation}
For example, we obviously have
\begin{equation}\label{eq:central0}
\mf g^f_0=\Hom(V_+,V_-)
\,.
\end{equation}
According to the decomposition \eqref{eq:V-k},
we have the decomposition
\begin{equation}\label{20180502:eq2}
\mf g^f_0
=
\bigoplus_{h,k=0}^{p_1-1}
\mf g^f_0(h,k)
\,\,\text{ where }\,\,
\mf g^f_0(h,k)
:=
\Hom(V_+\cap(F^t)^hV_-,V_-\cap F^kV_+)
\,.
\end{equation}
\begin{lemma}\label{lem:gfk}
For $\ell\in\mb Z_+$, consider the linear map $\phi_\ell:\,\mf g\to\mf g$
given by
\begin{equation}\label{eq:phik}
\phi_\ell(A)
=
\sum_{i=0}^\ell F^i(F^t)^\ell A(F^t)^\ell F^{\ell-i}
\,.
\end{equation}
\begin{enumerate}[(a)]
\item
For $0\leq h,k\leq p_1-1$
and $\ell>\min\{h,k\}$,
we have
$$
\phi_\ell(\mf g_0^f(h,k))=0
\,.
$$
\item
For $0\leq h,k\leq p_1-1$
and $\ell\leq\min\{h,k\}$, 
the restriction of $\phi_\ell$ to $\mf g_0^f(h,k)$
is injective, with values in $\mf g_\ell^f$:
$$
\phi_\ell\,:\,\,
\mf g_0^f(h,k)
\hookrightarrow
\mf g^f_\ell
\,.
$$
\item
We have the direct sum decomposition
$$
\mf g^f
=
\bigoplus_{h,k=0}^{p_1-1}\bigoplus_{\ell=0}^{\min\{h,k\}}
\phi_\ell\big(\mf g_0^f(h,k)\big)
\,.
$$
\end{enumerate}
In particular, if $\{u_i\}_{i\in\mc F(h,k)}$ is a basis of 
$\mf g_0^f(h,k)$
for every $0\leq h,k\leq p_1-1$,
then a basis of $\mf g^f$ is given by
\begin{equation}\label{20180502:eq1}
\big\{\phi_{\ell}(u_i)
\,\big|\,
0\leq h,k\leq p_1-1,\,
i\in\mc F(h,k),\,
0\leq \ell\leq \min\{h,k\}
\big\}
\,.
\end{equation}
\end{lemma}
\begin{proof}
Let $A\in\mf g^f_0(h,k)$,
i.e. $A\in\End V$ and 
$$
A=\id_{V_-\cap F^kV_+}A\id_{V_+\cap(F^t)^hV_-}
\,.
$$
If $\ell>k$, we have
$(F^t)^\ell\id_{F^kV_+}=0$, so that $\phi_\ell(A)=0$.
If $\ell>h$, we have
$\id_{(F^t)^hV_-}(F^t)^\ell=0$, so that $\phi_\ell(A)=0$.
This proves part (a).
Next,
let $A\in\mf g_0^f=\Hom(V_+,V_-)$,
i.e. $A\in\End V$ and $A=\id_{V_-}A\id_{V_+}$.
Clearly, for $i\leq\ell$, 
$F^i(F^t)^\ell\id_{V_-}$ has image in $(F^t)^{\ell-i}V_-$,
and $\id_{V_+}(F^t)^\ell F^{\ell-i}$ has domain $F^iV_+$.
Hence, 
$$
\phi_\ell(A)\in\,\,\bigoplus_{i=0}^\ell\Hom(F^iV_+,(F^t)^{\ell-i}V_-)
\,.
$$
Therefore, to show that $\phi_\ell(A)$ lies in $\mf g^f_\ell$, we only need to check
that $\phi_\ell (A)$ commutes with $F$.
We have
$$
\begin{array}{l}
\displaystyle{
\vphantom{\Big(}
[F,\phi_\ell(A)]
=
F\phi_\ell(A)-\phi_\ell(A)F
} \\
\displaystyle{
\vphantom{\Big(}
=
\sum_{i=0}^\ell F^{i+1}(F^t)^\ell A(F^t)^\ell F^{\ell-i}
-\sum_{i=0}^\ell F^{i}(F^t)^\ell A(F^t)^\ell F^{\ell-i+1}
} \\
\displaystyle{
\vphantom{\Big(}
=
F^{\ell+1}(F^t)^\ell A(F^t)^\ell-(F^t)^\ell A(F^t)^\ell F^{\ell+1}
=
0\,,}
\end{array}
$$
since $A=\id_{V_-}A\id_{V_+}$, and
$$
F^{\ell+1}(F^t)^\ell\id_{V_-}
=
0
\,\,,\,\,\,\,
\id_{V_+}(F^t)^\ell F^{\ell+1}
=
0
\,.
$$
To complete the proof of claim (b),
we are left to check that the restriction of $\phi_\ell$ to $\mf g_0^f(h,k)$
is injective for $\ell\leq\min\{h,k\}$.
Let $A\in\mf g_0^f(h,k)$ be non-zero,
and let
$0\neq v_+=(F^t)^hv_-\in V_+\cap(F^t)^hV_-$
be such that $0\neq Av_+\in V_-\cap F^kV_+$.
For $\ell\leq\min\{h,k\}$, 
we consider the non-zero vector $(F^t)^{h-\ell}v_-$,
and we shall prove that 
\begin{equation}\label{20180504:eq1}
\phi_\ell(A)((F^t)^{h-\ell}v_-)\neq0\,.
\end{equation}
Note that, for $j\geq0$,
$$
\id_{V_+}(F^t)^\ell F^{j}(F^t)^{h-\ell}v_-
=
\delta_{j,0}v_+
\,.
$$
Hence, since $A=A\id_{V_+}$, we have
$$
\phi_\ell(A)(F^t)^{h-\ell}v_-
=
\sum_{i=0}^\ell F^i(F^t)^\ell A(F^t)^\ell F^{\ell-i}(F^t)^{h-\ell}v_-
=
F^\ell(F^t)^\ell A v_+
\,,
$$
and this vector is non zero since $Av_+\in V_-\cap F^kV_+$, and $\ell\leq k$.
This proves \eqref{20180504:eq1} and completes the proof of part (b).
Finally we prove claim (c).
Given $\ell\geq0$, the sum of the spaces $\phi_\ell(\mf g_0^f(h,k))$ 
is a direct sum.
Indeed, the space $V$ is direct sum of the subspaces $V(p_i)$ given by 
the various rectangles of the pyramid of sizes $p_i\times r_i$, $i=1,\dots,s$,
(cf. \eqref{eq:p}): $V=\oplus_{i=1}^s V(p_i)$.
Then, 
by the definition \eqref{eq:phik} of $\phi_\ell$ 
and \eqref{20180502:eq2} of $\mf g_0^f(h,k)$,
we clearly have
$$
\phi_\ell(\mf g_0^f(h,k))\,\subset\,\Hom(V(h+1),V(k+1))
\,.
$$
It follows that the sum 
$$
\sum_{h,k=0}^{p_1-1}\sum_{\ell=0}^{\min\{h,k\}}\phi_\ell(\mf g_0^f(h,k))
$$
is a direct sum.
The fact that it coincides with $\mf g^f$
is obtained by dimension counting
using, for example, the basis of $\mf g^f$ in \cite{EK05}.
The last assertion of the Lemma is an obvious consequence of part (c).
\end{proof}
We let
\begin{equation}\label{eq:phiz}
\phi_z
=
\sum_{\ell\in\mb Z_+}(-z)^\ell\phi_\ell
\,\,:\,\mf g^f_0\to\mf g^f[z]
\,.
\end{equation}

%%%
\subsection{Choices of bases}
\label{sec:3.3}

As it is clear by Figures \ref{fig1}-\ref{fig2}, 
we may assume that
there exists a basis of $V$ compatible with all the $\Gamma$-gradings
that we shall consider.
Hence, there exists a basis $\{u_i\}_{i\in I}$ of $\mf g$ compatible with the $\Gamma$-grading:
$$
I=\sqcup_{j\in\frac12\mb Z}I_j%=\sqcup_{j\in\frac12\mb Z}I^\Sigma_j
\,,
$$
where 
$\{u_i\}_{i\in I_j}$ %(respectively $\{u_i\}_{i\in I^\Sigma_j}$)
is a basis of $\mf g[j]$. %(resp. $\mf g^\Sigma[j]$).
We also assume that $\{u_i\}_{i\in I_{\frac12}}$ extends a basis $\{u_i\}_{i\in I_{\mf l}}$ of 
the isotropic subspace $\mf l\subset\mf g[\frac12]$.
Hence, $I_{\mf m}:=I_{\mf l}\sqcup(\sqcup_{j\geq 1}I_j)$ is the index set of a basis of $\mf m$ in \eqref{eq:mn},
and $I_{\mf p}:=I\backslash I_{\mf m}$ is the index set of a basis of a complementary subspace $\mf p\subset\mf g$.
We shall also let $\{u^i\}_{i\in I}$ be the basis of $\mf g$ dual to $\{u_i\}_{i\in I}$:
$(u_i|u^j)=\delta_{i,j}$.
Recall that the following completeness identities hold
\begin{equation}\label{eq:compl}
a=\sum_{i\in I}(a|u^i)u_i=\sum_{i\in I}(a|u_i)u^i
\,.
\end{equation}

We will also assume, without loss of generality,
that $\{u_i\}_{i\in I}$ extends
\begin{equation}\label{eq:mcF}
\{u_i\}_{i\in\mc F}
\,,
\end{equation}
basis of $\mf g^f_0=\Hom(V_+,V_-)$
compatible with the decomposition \eqref{20180502:eq2}.
In particular, we have $I^{\Sigma}_{1-p_1}=I_{1-p_1}\subset\mc F$.
Clearly, the set $\mc F$ has cardinality $r^2$.
The dual basis $\{u^i\}_{i\in\mc F}$ is a basis of $\Hom(V_-,V_+)\subset\mf g$.

%%%
\subsection{Notational convention}
\label{sec:3.4}

Throughout the paper, we shall use the following notational convention:
we shall denote with lowercase letters the elements of the Lie algebra $\mf g$, viewed as elements of the universal enveloping algebra $U(\mf g)$,
and with the corresponding uppercase letters the same elements of $\mf g$, viewed as elements of $\End V$.
For instance, $F$ denotes the nilpotent endomorphism of $V$ corresponding to $f\in\mf g$.
Moreover, $\{U_i\}_{i\in I}$ is the basis of $\End V$ corresponding to $\{u_i\}_{i\in I}$,
and $\{U^i\}_{i\in I}$ is the dual basis w.r.t. the trace form of $\End V$.

With the above notational convention, we let 
\begin{equation}\label{eq:E}
E=\sum_{i\in I}u_iU^i
\,\,\text{ and }\,\,
E_{\mf p}=\sum_{i\in I_{\mf p}}u_iU^i
\,\,\in U(\mf g)\otimes\End V
\,,
\end{equation}
where we drop the tensor product sign for elements of $U(\mf g)\otimes\End V$.
At times we shall also denote, for $j\in\frac12\mb Z$,
\begin{equation}\label{eq:Ej}
E_j=\sum_{i\in I_j}u_iU^i
\,\,\in U(\mf g)\otimes\End V
\,.
\end{equation}
Similarly for $E_{\leq j}, E_{<j}$, etc.
For example, if $\mf l=0$, we have $E_{\mf p}=E_{\leq\frac12}$.
Clearly, we have $E=\sum_{j\in\frac12\mb Z}E_j$.
By the completeness relations \eqref{eq:compl}, for $a\in\mf g[i]$
(and denoting, as explained above, by $A$ the same element of $\End V$),
we have
\begin{equation}\label{20180215:eq1}
(a|E_j)=\delta_{i+j,0}A
\,.
\end{equation}

Recalling the definition of the quotient space $U(\mf g)/\mc I$ in \eqref{eq:Walg},
we have 
$$
E\bar1=(F+E_{\mf p})\bar1
\in (U(\mf g)/\mc I)\otimes\End V
\,.
$$

We also let
\begin{equation}\label{eq:D}
D_{\mf m}=-\sum_{i\in I_{\mf m}}U^iU_i
\,\in\End V
\,.
\end{equation}

%%%
\subsection{Notation for brackets}
\label{sec:3.4b}

Throughout the paper, we shall also use the following notation for brackets.
For elements $a,b$ in an associative algebra $R$, the bracket $[a,b]$
stands, as usual, for the commutator $ab-ba$.
Likewise, the bracket of elements of $R\otimes\End V$
is their commutator:
$$
[a_1A_2,b_1B_2]
=
a_1b_1A_2 B_2-b_1a_1B_2A_2
=
[a_1,b_1]A_2B_2+b_1a_1[A_2,B_2]
\,,
$$
for $a_1,b_1\in R$ and $A_2,B_2\in\End V$.
Here, as usual, we drop the tensor product sign for elements of $R\otimes\End V$.
Instead, we denote by $[\cdot\,,\,\cdot]^1$ the bracket in the first factor of 
$R\otimes\End V$ composed with the associative product in the second factor:
\begin{equation}\label{eq:bracket2}
[a_1A_2,b_1B_2]^1=[a_1,b_1]A_2B_2
\,\in R\otimes\End V
\,.
\end{equation}
We shall at times also encounter the bracket of an element of $R$
and an element of $R\otimes\End V$.
By this we mean the commutator in $R$:
\begin{equation}\label{eq:bracket3}
[a,b_1B_2]=[a,b_1]B_2
\,\Big(=[a\id_V,b_1B_2]\Big)
\,\in R\otimes\End V
\,.
\end{equation}

%%%
\subsection{The matrices $Z(z)$ and $W(z)$}
\label{sec:3.5}

Consider the following polynomial
\begin{equation}\label{20180215:eq2}
\id_{V_+}(1+zF^t)^{-1}\id_{V_-}
=
\sum_{k\in\mb Z_+}
(-1)^kz^k
\id_{V_+}(F^t)^k\id_{V_-}
\,\in\mb C[z]\otimes\Hom(V_-,V_+)
\,.
\end{equation}
Here the identity maps are used with the meaning specified in Section \ref{sec:3.1c}.
In particular, $\id_{V_-}$ denotes the inclusion $V_-\hookrightarrow V$,
while $\id_{V_+}$ denotes the projection $V\twoheadrightarrow V_+$
with kernel $FV$.
The operator \eqref{20180215:eq2}
maps each leftmost box of a pyramid to the corresponding rightmost box
obtained by a shift by $k$ to the right, multiplied by $(-z)^k$.

Recalling the map $\phi_z$ in \eqref{eq:phiz}, 
the basis \eqref{eq:mcF} of $\mf g^f_0=\Hom(V_+,V_-)$,
and using the notation introduced in Sections \ref{sec:3.1c} and \ref{sec:3.4},
we define the following operator
\begin{equation}\label{eq:Z}
Z(z)
=
z\id_{V_+}(1+zF^t)^{-1}\id_{V_-}
+
\sum_{i\in\mc F}\phi_z(u_i) U^i
\,\in U(\mf g^f)[z]\otimes\Hom(V_-,V_+)
\,.
\end{equation}
By Lemma \ref{lem:gfk},
the matrix entries of $Z(z)$ are polynomials in $z$
whose coefficients form a basis of $\mf g^f$.
More precisely, if, as in Lemma \ref{lem:gfk}, $\{u_i\}_{i\in\mc F(h,k)}$
is a basis of $\mf g_0^f(h,k)$, for every $0\leq h,k\leq p_1-1$,
we have
\begin{equation}\label{20180502:eq3}
\begin{array}{l}
\displaystyle{
\vphantom{\Big(}
\id_{V_+\cap(F^t)^hV_-}Z(z)\id_{V_-\cap F^k V_+}
} \\
\displaystyle{
\vphantom{\Big(}
=
-\delta_{h,k}(-z)^{k+1}
%\id_{V_+\cap(F^t)^kV_-}
(F^t)^k\id_{V_-\cap F^k V_+}
+
\sum_{\ell=0}^{\min\{h,k\}} \sum_{i\in\mc F(h,k)}
(-z)^\ell \phi_\ell(u_i) U^i
\,.}
\end{array}
\end{equation}

Let $w:\,\mf g^f\to W(\mf g,f)$ be a Premet map (cf. Definition \ref{def:premet}).
By Theorem \ref{thm:Premetconds} the elements $w(\phi_k(u_i))$,
$i\in\mc F$, $k\in\mb Z_+$, form a set of PBW generators for the $W$-algebra $W(\mf g,f)$.
We denote by $W(z)$ the image via $w$ of the matrix \eqref{eq:Z}:
\begin{equation}\label{eq:W}
W(z)
=
\!\!
z\id_{V_+}(1+zF^t)^{-1}\id_{V_-}
+
\sum_{i\in\mc F}w(\phi_z(u_i)) U^i
\,\in W(\mf g,f)[z]\otimes\Hom(V_-,V_+)
\,.
\end{equation}

%%%%%%%%%%%%%%%%%%%%%%%%%%%%%%%%%%%%%%%
\section{The matrix \texorpdfstring{$L(z)$}{L(z)} and generators of the \texorpdfstring{$W$}{W}-algebra}\label{sec:4}

%%%
\subsection{Generalized quasideterminants}
\label{sec:4.1}

Recall that \cite{GGRW05,DSKV17}, given $A\in R\otimes \Hom(V,\widetilde V)$
($R$ being a unital associative algebra and $V,\widetilde V$ 
being finite dimensional vector spaces),
and decompositions $V=U\oplus U^\perp$
and $\widetilde V=W\oplus W^\perp$,
the corresponding $(U,W)$ \emph{generalized quasideterminant} of $A$
is defined as
\begin{equation}\label{eq:quasidet}
|A|_{U,W}=(\id_{U} A^{-1}\id_{W})^{-1}
\,\in R\otimes\Hom(U,W)
\,,
\end{equation}
provided that it exists, i.e. provided that $A\in R\otimes \Hom(V,\widetilde V)$
is invertible and $\id_U A^{-1}\id_{W}\in R\otimes\Hom(W,U)$ is invertible.

The generalized quasideterminant has the following hereditary property:
\begin{equation}\label{eq:hered}
\big||A|_{U,W}\big|_{U_1,W_1}
=
|A|_{U_1,W_1}
\,,
\end{equation}
for splittings $V=U\oplus U^\perp$, $U=U_1\oplus U_1^\perp$,
$\widetilde V=W\oplus W^\perp$ and $W=W_1\oplus W_1^\perp$,
provided that all quasideterminants exist.
Obviously, the RHS is associated to the splittings 
$V=U_1\oplus (U^\perp\oplus U_1^\perp)$
and 
$\widetilde V=W_1\oplus (W^\perp\oplus W_1^\perp)$.
\begin{proposition}\cite[Prop.2.4]{DSKV17}\label{prop:quasidet}
Suppose that $A \in R\otimes \Hom(V,\widetilde V)$ is invertible,
and that we have decompositions $V=U\oplus U^\perp$
and $\widetilde V=W\oplus W^\perp$.
Then the $(U,W)$ generalized quasideterminant $|A|_{U,W}$ exists
if and only if $\id_{W^{\perp}}A\id_{U^{\perp}}\in R\otimes\Hom(U^{\perp},W^{\perp})$
is invertible,
and in this case we have
\begin{equation}\label{eq:quasidet2}
|A|_{U,W}
=
\id_W
\Big(
A-
A\id_{U^{\perp}}
\big(\id_{W^{\perp}}A\id_{U^{\perp}}\big)^{-1}
\id_{W^{\perp}}A
\Big)
\id_{U}
\,.
\end{equation}
\end{proposition}

%%%
\subsection{The matrix $L(z)$ for the $W$-algebras in type $A$}
\label{sec:4.2}
Recall from equations \eqref{eq:Vgrading} and \eqref{20180215:eq3} that we have the splittings
$V=V_{\pm}^d\oplus V^{\Sigma}[\neq\pm\frac {p_1-1}2]$.
The Yangian type operator for the quantum finite $W$-algebra $W(\mf g,f)$
defined in \cite{DSKV16}
is constructed as the following $(V_-^d,V_+^d)$ generalized quasideterminant
\begin{equation}\label{eq:L}
L(z)
=
|z\id_V+F+E_{\mf p}+D_{\mf m}|_{V_-^d,V_+^d}\bar1
\,,
\end{equation}
where $E_{\mf p}$ is as in \eqref{eq:E} and $D_{\mf m}$ is as in \eqref{eq:D}.
\begin{example}\label{ex:f=0}
In the special case when $f=0$,
the corresponding pyramid consists of a single column, with $p_1=1$ and $r_1=r=N$,
the good grading is concentrated at degree $0$,
the matrix $D_{\mf m}$ in \eqref{eq:D} vanishes,
and the subspaces $V_{\pm}$ and $V^d_{\pm}$ are all equal to $V$.
Moreover, for $f=0$ we have $W(\mf g,f)=U(\mf g)$.
Hence, in this case the operator $L(z)$ in \eqref{eq:L} becomes
\begin{equation}\label{eq:L1}
L(z)
=
z\id_V+E
\,.
\end{equation}
\end{example}

The following result is proved in \cite{DSKV16,DSKV17} in the case when
$\mf l=0$ and the grading $\Gamma$ coincides with the Dynkin grading $\Sigma$.
The proof in the general case is essentially the same \cite{Fed18}.
\begin{theorem}\label{thm:Lexists}
The quasideterminant defining $L(z)$ exists, and
$$
L(z)\,\in\,
W(\mf g,f)((z^{-1}))\otimes\Hom(V^d_-,V^d_+)
\,.
$$
Moreover, $L(z)$ is an operator of Yangian type for the quantum finite $W$-algebra $W(\mf g,f)$,
i.e.
\begin{equation}\label{eq:yang}
\begin{array}{c}
\displaystyle{
\vphantom{\Big(}
(z-w)\big(
L(z)\otimes L(w)
-\big(\id_{V^d_+}\otimes L(w)\big)
\big(L(z)\otimes\id_{V^d_-}\big)
\big)
} \\
\displaystyle{
\vphantom{\Big(}
=
\Omega
\big(L(w)\otimes L(z)-L(z)\otimes L(w)\big)
\,,}
\end{array}
\end{equation}
in
$W(\mf g,f)[[z^{-1},w^{-1}]][z,w]\otimes\Hom(V^d_-,V^d_+)^{\otimes2}$.
\end{theorem}
The operator 
$\Omega\in\End{(V^d_+)}^{\otimes2}\simeq\End({V^d_+}^{\otimes 2})$
corresponds to switching the two factors.
\begin{remark}
When written in matrix form,
equation \eqref{eq:yang} coincides with the defining equation
of the Yangian of $\mf{gl}_{r_1}$ \cite{Mol07}.
\end{remark}

%%%
\subsection{The matrix $L(z)$ in terms of Premet's generators of $W(\mf g,f)$}
\label{sec:4.3}

The following result, relating the operator $L(z)$ in \eqref{eq:L}
and $W(z)$ in \eqref{eq:W}, was conjectured in \cite{DSKV16}
for the special case of Dynkin grading $\Gamma$ and zero isotropic subspace $\mf l$,
and it is the main result of the present paper.
\begin{theorem}\label{thm:main}
There exists a Premet map $w:\,\mf g^f\to W(\mf g,f)$ (cf. Definition \ref{def:premet})
such that
\begin{equation}\label{eq:main}
L(z)
=
|W(z)|_{V_-^d,V_+^d}
\,,
\end{equation}
where $L(z)$ and $W(z)$ are given by \eqref{eq:L}
and \eqref{eq:W} respectively.
\end{theorem}
\begin{remark}\label{20180522:rem}
It can be useful to expand the quasideterminant $|W(z)|_{V_-^d,V_+^d}$ in powers of $z$.
By \eqref{eq:W} we get
\begin{align*}
& |W(z)|_{V_-^d,V_+^d}
=
-(-z)^{p_1}(F^t)^{p_1-1}
+
\sum_{\ell=0}^{p_1-1}\sum_{i\in\mc F(p_1-1,p_1-1)}
(-z)^\ell w_{i,\ell}U^i
\\
& +
\sum_{s=0}^\infty 
\sum_{h_1,\dots,h_{s+1}=0}^{p_1-2}
\sum_{\ell_0=0}^{h_1}
\sum_{\ell_1=0}^{\min\{h_1,h_2\}}
\dots
\sum_{\ell_s=0}^{\min\{h_s,h_{s+1}\}}
\sum_{\ell_{s+1}=0}^{h_{s+1}}
\\
&\times
\sum_{i_0\in\mc F(p_1-1,h_1)}
\sum_{i_1\in\mc F(h_1,h_2)}
\dots
\sum_{i_s\in\mc F(h_s,h_{s+1})}
\sum_{i_{s+1}\in\mc F(h_{s+1},p_1-1)}
\\
&\times
\vphantom{\Big(}
(-z)^{L-H-s-1} 
w_{i_0,\ell_0}w_{i_1,\ell_1}\dots w_{i_{s+1},\ell_{s+1}}
U^{i_0}F^{h_1}U^{i_1}\dots 
%F^{h_s}U^{i_s}
F^{h_{s+1}}U^{i_{s+1}}
\,,
\end{align*}
where $\mc F(h,k)$ is the index set introduced in \eqref{20180502:eq1},
$L=\ell_0+\ell_1+\dots+\ell_{s+1}$, $H=h_1+\dots+h_{s+1}$,
and $w_{i,\ell}=w(\phi_\ell(u_i))$.
\end{remark}
\begin{remark}\label{20180522:rem2}
We believe that, fixed a subspace $U\subset\mf g$
complementary to $[f,\mf g]$, 
the corresponding Premet map $w:\,\mf g^f\to W(\mf g,f)$
for which equation \eqref{eq:main} holds is unique.
\end{remark}

%%%%%%%%%%%%%%%%%%%%%%%%%%%%%%%%%%%%%%%
\section{The case of aligned pyramid: notation and preliminary lemmas}\label{sec:5a}

We shall first prove Theorem \ref{thm:main}
in the special case when the good grading $\Gamma$ corresponds to
a pyramid $p$ which is aligned to the right.
This is the content of the present section.
In Section \ref{sec:6} we shall then extend the proof to the case of an arbitrary
good grading for $f\in\mf g$.

%%%
\subsection{Some pictures}\label{sec:5.1}

Throughout this section,
we let $\Gamma:\,\mf g=\bigoplus_{j=-p_1+1}^{p_1-1}\mf g[j]$
be the good grading for $f$ associated to the pyramid aligned to the right.
For example, for the partition $(4,4,3,1)$ of $12$,
we have the pyramid

\begin{figure}[H]
\setlength{\unitlength}{0.15in}
\begin{picture}(12,11)
\setlength\fboxsep{0pt}

\put(7,9){\framebox(2,2){}}

\put(3,7){\framebox(2,2){}}
\put(5,7){\framebox(2,2){}}
\put(7,7){\framebox(2,2){}}

\put(1,5){\framebox(2,2){}}
\put(3,5){\framebox(2,2){}}
\put(5,5){\framebox(2,2){}}
\put(7,5){\framebox(2,2){}}

\put(1,3){\framebox(2,2){}}
\put(3,3){\framebox(2,2){}}
\put(5,3){\framebox(2,2){}}
\put(7,3){\framebox(2,2){}}

\put(0,2){\vector(1,0){11}}
\put(10,1){$x$}

\put(2,1.8){\line(0,1){0.4}}
\put(3,1.6){\line(0,1){0.8}}
\put(4,1.8){\line(0,1){0.4}}
\put(5,1.6){\line(0,1){0.8}}
\put(6,1.8){\line(0,1){0.4}}
\put(7,1.6){\line(0,1){0.8}}
\put(8,1.8){\line(0,1){0.4}}

\put(4.7,0.6){0}
\put(2.5,0.6){-1}
\put(6.8,0.6){1}
\end{picture}
\caption{ }
\label{fig3}
\end{figure}

As explained in Section \ref{sec:3.1},
each box of the pyramid corresponds to a basis element of $V=\mb F^N$,
and its $\Gamma$-degree coincides with the $x$-coordinate of the center of the box.
The nilpotent element $F\in\End V$ is the shift to the left
and $F^t$ is the shift to the right.

Recall the definitions \eqref{eq:V+-}, \eqref{eq:Vdu} and \eqref{eq:V_-ud} 
of the subspaces $V_{\pm},V^d,V^u,V_{\pm}^d,V_{\pm}^u\subset V$.
For example, for the pyramid of Figure \ref{fig3}, the subspaces $V_-$ and $V_+$ 
correspond, respectively, 
to the boxes colored in gray and blue below:

%V_+ e V_-
\begin{figure}[H]
\setlength{\unitlength}{0.15in}
\begin{picture}(12,11)
\setlength\fboxsep{0pt}

\put(7,9){\framebox(2,2){}}

\put(3,7){\colorbox{gray}{\framebox(2,2){}}}
\put(5,7){\framebox(2,2){}}
\put(7,7){\colorbox{cyan}{\framebox(2,2){}}}

\put(1,5){\colorbox{gray}{\framebox(2,2){}}}
\put(3,5){\framebox(2,2){}}
\put(5,5){\framebox(2,2){}}
\put(7,5){\colorbox{cyan}{\framebox(2,2){}}}

\put(1,3){\colorbox{gray}{\framebox(2,2){}}}
\put(3,3){\framebox(2,2){}}
\put(5,3){\framebox(2,2){}}
\put(7,3){\colorbox{cyan}{\framebox(2,2){}}}

\put(0,2){\vector(1,0){11}}
\put(10,1){$x$}

\put(2,1.8){\line(0,1){0.4}}
\put(3,1.6){\line(0,1){0.8}}
\put(4,1.8){\line(0,1){0.4}}
\put(5,1.6){\line(0,1){0.8}}
\put(6,1.8){\line(0,1){0.4}}
\put(7,1.6){\line(0,1){0.8}}
\put(8,1.8){\line(0,1){0.4}}

\put(4.7,0.6){0}
\put(2.5,0.6){-1}
\put(6.8,0.6){1}

\put(-1.5,6.5){$V_-$}
\put(11,6.5){$V_+$}

\begin{tikzpicture}[remember picture,overlay]
\path[draw,fill=gray](2.66,4.2) -- (3.43,4.2)--(2.66,3.44)-- cycle;
\path[draw,fill=cyan](3.43,4.19)--(2.66,3.44)--(3.43,3.44) -- cycle;
\end{tikzpicture}

\end{picture}
\caption{ }
\label{fig4}
\end{figure}
\noindent
The subspaces $V_-^d$ and $V_-^u$ correspond, respectively,
to the boxes colored in light grey and green below:
%V_-up V_-down
\begin{figure}[H]
\setlength{\unitlength}{0.15in}
\begin{picture}(12,11)
\setlength\fboxsep{0pt}

\put(7,9){\colorbox{teal}{\framebox(2,2){}}}

\put(3,7){\colorbox{teal}{\framebox(2,2){}}}
\put(5,7){\framebox(2,2){}}
\put(7,7){\framebox(2,2){}}

\put(1,5){\colorbox{lightgray}{\framebox(2,2){}}}
\put(3,5){\framebox(2,2){}}
\put(5,5){\framebox(2,2){}}
\put(7,5){\framebox(2,2){}}

\put(1,3){\colorbox{lightgray}{\framebox(2,2){}}}
\put(3,3){\framebox(2,2){}}
\put(5,3){\framebox(2,2){}}
\put(7,3){\framebox(2,2){}}

\put(0,2){\vector(1,0){11}}
\put(10,1){$x$}

\put(2,1.8){\line(0,1){0.4}}
\put(3,1.6){\line(0,1){0.8}}
\put(4,1.8){\line(0,1){0.4}}
\put(5,1.6){\line(0,1){0.8}}
\put(6,1.8){\line(0,1){0.4}}
\put(7,1.6){\line(0,1){0.8}}
\put(8,1.8){\line(0,1){0.4}}

\put(4.7,0.6){0}
\put(2.5,0.6){-1}
\put(6.8,0.6){1}

\put(-1.5,5){$V_-^d$}
\put(-1.5,9){$V_-^u$}

\end{picture}
\caption{ }
\label{fig5}
\end{figure}
\noindent
Moreover, the subspaces $V_+^d$ and $V_+^u$ correspond, respectively,
to the boxes colored in light gray and green below:
%V_+up V_+down
\begin{figure}[H]
\setlength{\unitlength}{0.15in}
\begin{picture}(12,11)
\setlength\fboxsep{0pt}

\put(7,9){\colorbox{teal}{\framebox(2,2){}}}

\put(3,7){\framebox(2,2){}}
\put(5,7){\framebox(2,2){}}
\put(7,7){\colorbox{teal}{\framebox(2,2){}}}

\put(1,5){\framebox(2,2){}}
\put(3,5){\framebox(2,2){}}
\put(5,5){\framebox(2,2){}}
\put(7,5){\colorbox{lightgray}{\framebox(2,2){}}}

\put(1,3){\framebox(2,2){}}
\put(3,3){\framebox(2,2){}}
\put(5,3){\framebox(2,2){}}
\put(7,3){\colorbox{lightgray}{\framebox(2,2){}}}

\put(0,2){\vector(1,0){11}}
\put(10,1){$x$}

\put(2,1.8){\line(0,1){0.4}}
\put(3,1.6){\line(0,1){0.8}}
\put(4,1.8){\line(0,1){0.4}}
\put(5,1.6){\line(0,1){0.8}}
\put(6,1.8){\line(0,1){0.4}}
\put(7,1.6){\line(0,1){0.8}}
\put(8,1.8){\line(0,1){0.4}}

\put(4.7,0.6){0}
\put(2.5,0.6){-1}
\put(6.8,0.6){1}

\put(11,5){$V_+^d$}
\put(11,9){$V_+^u$}

\end{picture}
\caption{ }
\label{fig5a}
\end{figure}

For the induction arguments which will be used throughout 
Sections \ref{sec:5a}--\ref{sec:5d},
we shall need to consider the pyramid 
$p'$ obtained from $p$ by removing the leftmost column.

It corresponds to the subspace $V'=V[>-\frac {p_1-1}2]\cong \mb F^{N-r_1}$. 
In the example of Figure \ref{fig3},
the subspace $V'$ and the pyramid $p'$ are depicted in the following picture.

%Vprime
\begin{figure}[H]
\setlength{\unitlength}{0.15in}
\begin{picture}(12,11)
\setlength\fboxsep{0pt}

\put(7,9){\framebox(2,2){}}

\put(3,7){\framebox(2,2){}}
\put(5,7){\framebox(2,2){}}
\put(7,7){\framebox(2,2){}}

\put(1,5){\dashbox{0.2}(2,2){}}
\put(3,5){\framebox(2,2){}}
\put(5,5){\framebox(2,2){}}
\put(7,5){\framebox(2,2){}}

\put(1,3){\dashbox{0.2}(2,2){}}
\put(3,3){\framebox(2,2){}}
\put(5,3){\framebox(2,2){}}
\put(7,3){\framebox(2,2){}}

\put(0,2){\vector(1,0){11}}
\put(10,1){$x$}

\put(2,1.8){\line(0,1){0.4}}
\put(3,1.6){\line(0,1){0.8}}
\put(4,1.8){\line(0,1){0.4}}
\put(5,1.6){\line(0,1){0.8}}
\put(6,1.8){\line(0,1){0.4}}
\put(7,1.6){\line(0,1){0.8}}
\put(8,1.8){\line(0,1){0.4}}

\put(4.7,0.6){0}
\put(2.5,0.6){-1}
\put(6.8,0.6){1}

\put(11,7){$V'$}

\end{picture}
\caption{ }
\label{fig6}
\end{figure}

\noindent
The associated (to $p'$) nilpotent endomorphism is denoted by $F'$,
which is the shift to the left,
while $(F')^t$ is the shift to the right.
Similarly, we denote by $p''$, $V''$, etc. the pyramid, the space, etc. obtained
by removing the two leftmost columns.

For the pyramid $p'$
we can consider the subspaces $V'_+$ and $V'_-$ as in \eqref{eq:V+-}
with $F$ replaced by $F'$.
Clearly, $V'_+=V_+$, while $V'_-$ has the decomposition (cf. \eqref{eq:V_-ud})
\begin{equation}\label{eq:V''_-ud}
V'_-=F^tV_-^{d}\oplus {V}_-^{u}\,%,\,
%\quad\text{where}\quad
%{V'}_-^{u}=V_-^{u}
\,.
\end{equation}
For the example of Figure \ref{fig6},
$F^tV_-^{d}$ and ${V}_-^u$ correspond to the 
boxes colored in light grey and green respectively
in the figure below.

%Vprime
\begin{figure}[H]
\setlength{\unitlength}{0.15in}
\begin{picture}(12,11)
\setlength\fboxsep{0pt}

\put(7,9){\colorbox{teal}{\framebox(2,2){}}}

\put(3,7){\colorbox{teal}{\framebox(2,2){}}}
\put(5,7){\framebox(2,2){}}
\put(7,7){\framebox(2,2){}}

\put(1,5){\dashbox{0.2}(2,2){}}
\put(3,5){\colorbox{lightgray}{\framebox(2,2){}}}
\put(5,5){\framebox(2,2){}}
\put(7,5){\framebox(2,2){}}

\put(1,3){\dashbox{0.2}(2,2){}}
\put(3,3){\colorbox{lightgray}{\framebox(2,2){}}}
\put(5,3){\framebox(2,2){}}
\put(7,3){\framebox(2,2){}}

\put(0,2){\vector(1,0){11}}
\put(10,1){$x$}

\put(2,1.8){\line(0,1){0.4}}
\put(3,1.6){\line(0,1){0.8}}
\put(4,1.8){\line(0,1){0.4}}
\put(5,1.6){\line(0,1){0.8}}
\put(6,1.8){\line(0,1){0.4}}
\put(7,1.6){\line(0,1){0.8}}
\put(8,1.8){\line(0,1){0.4}}

\put(4.7,0.6){0}
\put(2.5,0.6){-1}
\put(6.8,0.6){1}

\put(-2.5,5){$F^tV_-^d$}
\put(-1.5,9){$V_-^u$}

\end{picture}
\caption{ }
\label{fig7}
\end{figure}

We can also decompose $V'_-$ as in \eqref{eq:V_-ud}:
\begin{equation}\label{eq:V'_-ud}
V'_-={V'}_-^{d}\oplus {V'}_-^{u}\,,
\end{equation}
where
${V'}_-^{d}%={V'}[-\frac {p_1-1}2 + 1]
\supset F^t V_-^d$,
and ${V'}_-^{u}%=V'_-\cap {V'}[>-\frac {p_1-1}2 +1]
\subset V_-^u$.
For the example of Figure \ref{fig6},
${V'}_-^{d}$ and ${V'}_-^u$ correspond to the 
boxes colored in light gray and green respectively
in the figure below.

%Vprime
\begin{figure}[H]
\setlength{\unitlength}{0.15in}
\begin{picture}(12,11)
\setlength\fboxsep{0pt}

\put(7,9){\colorbox{teal}{\framebox(2,2){}}}

\put(3,7){\colorbox{lightgray}{\framebox(2,2){}}}
\put(5,7){\framebox(2,2){}}
\put(7,7){\framebox(2,2){}}

\put(1,5){\dashbox{0.2}(2,2){}}
\put(3,5){\colorbox{lightgray}{\framebox(2,2){}}}
\put(5,5){\framebox(2,2){}}
\put(7,5){\framebox(2,2){}}

\put(1,3){\dashbox{0.2}(2,2){}}
\put(3,3){\colorbox{lightgray}{\framebox(2,2){}}}
\put(5,3){\framebox(2,2){}}
\put(7,3){\framebox(2,2){}}

\put(0,2){\vector(1,0){11}}
\put(10,1){$x$}

\put(2,1.8){\line(0,1){0.4}}
\put(3,1.6){\line(0,1){0.8}}
\put(4,1.8){\line(0,1){0.4}}
\put(5,1.6){\line(0,1){0.8}}
\put(6,1.8){\line(0,1){0.4}}
\put(7,1.6){\line(0,1){0.8}}
\put(8,1.8){\line(0,1){0.4}}

\put(4.7,0.6){0}
\put(2.5,0.6){-1}
\put(6.8,0.6){1}

\put(-1.5,5.5){${V'}_-^d$}
\put(-1.5,10){${V'}_-^u$}

\end{picture}
\caption{ }
\label{fig7b}
\end{figure}
Note that, with the notation introduced above,
$V_-^d,\,{V'}_-^d,\,{V''}_-^d,\,\dots$
correspond to the leftmost, second leftmost, third leftmost, $\dots$, columns of the pyramid.
They are of heights $r_1,\,r'_1,\,r''_1,\,\dots$ respectively.

Recall the decompositions \eqref{eq:V-k} of $V_{\pm}$.
After removing the leftmost column, we have the corresponding decompositions of ${V'}_{\pm}$.
We have
\begin{equation}\label{20180502:eq4}
V_-\cap F^kV_+
\,\,\left\{\begin{array}{l}
=V'_-\cap F^kV_+ \,\,\text{ if }\,\, k<p_1-2\,, \\
\subset V'_-\cap F^kV_+ \,\,\text{ if }\,\, k=p_1-2\,, \\
\not\subset {V'} \,\,\text{ if }\,\, k=p_1-1\,,
\end{array}\right.
\end{equation}
and similarly
\begin{equation}\label{20180502:eq5}
V_+\cap (F^t)^hV_-
\,\,\left\{\begin{array}{l}
={V}_+\cap (F^t)^h(V')_- \,\,\text{ if }\,\, h<p_1-2\,, \\
\subset V_+\cap (F^t)^{p_1-2}V'_- \,\,\text{ if }\,\, h\geq p_1-2 \,.
\end{array}\right.
\end{equation}
In particular,
\begin{equation}\label{20180504:eq4}
F^tV_-^d\oplus(V_-\cap F^{p_1-2}V_+)={V'}_-^d
\,\,\text{ and }\,\,
V_+^d\oplus(V_+\cap (F^t)^{p_1-2}V_-)={V'}_+^d
\,.
\end{equation}

%%%
\subsection{Removing a column}\label{sec:5.1c}

Recall the operator $E\in U(\mf g)\otimes\End V$ in \eqref{eq:E}.
For a right aligned pyramid $p$,
the $\Gamma$-grading is even, hence the operator $E_{\mf p}$ in \eqref{eq:E} 
coincides with
\begin{equation}\label{eq:E<=0}
E_{\leq0}=\sum_{i\in I_{\leq0}}u_iU^i
\,.
\end{equation}
Also, in the case of an aligned pyramid, $\mf m=\mf g[\geq1]$,
and we simply write $D$ in place of the matrix $D_{\mf m}$ in \eqref{eq:D}:
\begin{equation}\label{eq:D0}
D
=
-\sum_{i\in I_{\geq1}}U^iU_i
\,.
\end{equation}
In this case, the matrix $D$ is diagonal with respect to the $\Gamma$-grading
with the following eigenvalues (cf. \cite{DSKV17}):
\begin{equation}\label{20180219:eq2}
D
=
-\sum_{k\in\frac12\mb Z}
\dim(V[\geq k+1])\id_{V[k]}
\,.
\end{equation}

Recall from Section \ref{sec:5.1} 
the notation $p'$ for the pyramid obtained by removing 
the leftmost column from $p$,
and $V'\subset V$ for the corresponding subspace.
As $V$ is the standard representation of $\mf g\simeq\End V$,
$V'$ is the standard representation of $\mf g'=\mf{gl}_{N-r_1}$,
where $r_1$ coincides with the height of the removed column.
We already introduced in Section \ref{sec:5.1} the nilpotent endomorphisms
$F'$ and ${F'}^t$.
Likewise, we denote by $E',E'_{\leq0},D'$
the same operators as in \eqref{eq:E}, \eqref{eq:E<=0} and \eqref{eq:D0}
respectively, for $V'$ and $\mf g'$ in places of $V$ and $\mf g$.
\begin{lemma}\label{lem:restr}
\begin{enumerate}[(a)]
\item
$E_{\leq0}|_{V'}=E'_{\leq0}$
\item
$D|_{V'}=D'$
\item
$F|_{F^tV'}=F'|_{F^tV'}$
\item
$F^t|_{V'}={F'}^t$
\end{enumerate}
\end{lemma}
\begin{proof}
Note that the operators $U^i$ appearing in the expression \eqref{eq:E<=0} of $E_{\leq0}$
have non-negative degrees, hence they preserve the subspace $V'\subset V$.
Claim (a) follows.
The same argument works for claim (b).
Claims (c) and (d) are obvious.
\end{proof}

%%%
\subsection{Preliminary Lemmas}\label{sec:5.1b}

\begin{lemma}\label{20180426:lem}
Let $U,W\subset V$ be subspaces, and let $\{q_i\}_{i\in J}$ be a basis of $\Hom(U,W)\subset\mf g$,
and $\{Q^i\}_{i\in J}$ be the dual basis of $\Hom(W,U)\subset\End V$.
We have
$$
\sum_{i\in I}U_i \otimes \id_UU^i\id_W
=
\sum_{i\in I}\id_WU_i\id_U\otimes U^i
\,,
$$
or, equivalently,
$$
\id_UE\id_W
=
\sum_{i\in I}u_i \id_UU^i\id_W
=
\sum_{i\in J}q_i Q^i
\,.
$$
\end{lemma}
\begin{proof}
Obvious.
\end{proof}
\begin{lemma}\label{20180216:lem1}
\begin{enumerate}[(a)]
\item
For $A\in\End V$, we have
$\sum_{i\in I}U_iAU^i=\tr(A)\id_V$.
\item
If $U\subset V$ is a subspace,
$\sum_{i\in I}U_i\id_UU^i=\dim(U)\id_V$.
\end{enumerate}
\end{lemma}
\begin{proof}
Clearly, the statement is independent of the choice of the dual bases 
$\{U_i\}$, $\{U^i\}$ of $\End V$.
Claim (a) is then easily checked directly, 
choosing the dual bases $\{E_{ij}\}$, $\{E_{ji}\}$
of elementary matrices.
Claim (b) is a special case of (a).
\end{proof}
Recall the definition \eqref{eq:Ej} of the operators $E_j\in U(\mf g)\otimes\End V$,
$j\in\frac12\mb Z$.
\begin{lemma}\label{lem:E}
Let $i,j,k\in\frac12\mb Z$.
\begin{enumerate}[(a)]
\item
Let $X\subset (\End V)[k]$ be an endomorphism.
We have:
\begin{equation}\label{20180301:eq0a}
[E_i,XE_j]^1
=
\delta_{k,i+j}
\sum_{\ell\in\frac12\mb Z}
\tr(E_{i+j}X\id_{V[\ell]})\id_{V[\ell+j]}
-
\delta_{k,0}
\sum_{\ell\in\frac12\mb Z}
\tr(X\id_{V[\ell]})
E_{i+j}\id_{V[\ell+j]}
\,,
\end{equation}
where $[\cdot\,,\,\cdot]^1$ is the bracket defined in \eqref{eq:bracket2}.
\item
Let $U\subset V[k]$ be a subspace. We have:
\begin{equation}\label{20180301:eq0b}
[E_i,\id_UE_j]^1
=
\delta_{i+j,0}
\tr\big(E_{0}\id_U\big)\id_{V[k+j]}
- \dim(U) E_{i+j}\id_{V[k+j]}
\,.
\end{equation}
\item
Let $U\subset V[k]$ be a subspace. We have:
\begin{equation}\label{20180301:eq0c}
[E_i,\id_UF^tE_j]^1
=
\delta_{1,i+j}
\tr\big(\id_UF^tE_{1}\big)\id_{V[k+j-1]}
\equiv
\delta_{1,i+j}
\dim(U\cap F^tV)\id_{V[k+j-1]}
\,.
\end{equation}
\end{enumerate}
\end{lemma}
\begin{proof}
We have:
$$
\begin{array}{l}
\vphantom{\Big(}
\displaystyle{
[E_i,XE_j]^1
=
\sum_{a\in I_i,b\in I_j}[u_a,u_b]U^aXU^b
=
\sum_{a\in I_i,b\in I_j,c\in I}u_c([u_a,u_b]|u^c)U^aX U^b
} \\
\vphantom{\Big(}
\displaystyle{
=
\sum_{a\in I,b\in I_j,c\in I_{i+j}}u_c(u_a|[u_b,u^c])U^aX U^b
=
\sum_{b\in I_j,c\in I_{i+j}}u_c[U_b,U^c]X U^b
} \\
\vphantom{\Big(}
\displaystyle{
=
\sum_{\ell\in\frac12\mb Z}
\sum_{b\in I,c\in I_{i+j}}u_c
(U_bU^c-U^cU_b)X\id_{V[\ell]}U^b\id_{V[\ell+j]}
} \\
\vphantom{\Big(}
\displaystyle{
=
\sum_{\ell\in\frac12\mb Z}
\sum_{b\in I}
(U_bE_{i+j}X \id_{V[\ell]}U^b
-E_{i+j}U_bX \id_{V[\ell]}U^b)
\id_{V[\ell+j]}
} \\
\vphantom{\Big(}
\displaystyle{
=
\delta_{k,i+j}
\sum_{\ell\in\frac12\mb Z}
\tr(E_{i+j}X\id_{V[\ell]})\id_{V[\ell+j]}
-
\delta_{k,0}
\sum_{\ell\in\frac12\mb Z}
\tr(X\id_{V[\ell]})
E_{i+j}\id_{V[\ell+j]}
\,.}
\end{array}
$$
In the first equality we used the definition \eqref{eq:Ej} of $E_i$,
in the second equality we used the completeness identity \eqref{eq:compl},
in the third equality we used the invariance of the trace form of $\mf g$,
in the fourth equality we used again \eqref{eq:compl},
in the fifth equality we used 
the decomposition $\id_V=\sum_\ell\id_{V[\ell]}$,
finally in the last equality we used Lemma \ref{20180216:lem1}.
This proves part (a).
Parts (b) and (c) are obtained as the special cases when $X=\id_U$ and $X=\id_UF^t$
respectively.
\end{proof}
\begin{lemma}\label{lem:E2}
For $q\in\mf h:=\Hom(F^tV_-^d,V_-^d)\subset\mf g$ 
and $w\in F^tV'\oplus V_-^u$, we have
(cf. notation \eqref{eq:bracket3})
\begin{equation}\label{eq:ciao}
[q,E_{\leq0}](w)
=0
\,\,\text{ in }\,\,
\mf g\otimes V
\end{equation}
\end{lemma}
\begin{proof}
Let us pair the first factor in the LHS of \eqref{eq:ciao} with an arbitrary element $a\in\mf g$.
We get, by the invariance of the trace form and the completeness identity \eqref{eq:compl}
$$
\sum_{j\in I_{\leq0}}(a|[q,u_j])U^j(w)
=
[A,Q]_{\geq0}(w)
\,,
$$
where $A$ (resp. $Q$)$\in\End V$ is the endomorphism 
corresponding to $a$ (resp. $q$)$\in\mf g$,
and the index $\geq0$ denotes the component in $\End V$
of non-negative $\Gamma$-degree.
Obviously, $Q(w)=0$, 
since $Q \in \Hom(F^tV_-^d, V_-^d) $ 
and $w \in F^tV' \oplus V_-^u \subset \Ker Q$.
Moreover, $QA(w)\in V_-^d=V[-\frac d2]$,
while $w\in F^tV'\oplus V_-^u\subset V[>-\frac d2]$.
Hence, $(QA)_{\geq0}(w)=0$.
\end{proof}
\begin{lemma}\label{20180216:lem2}
For $ a \in \mf g[k] $, and for any $ i \in \frac12 \mb Z $, the following holds:
\begin{equation}\label{20180219:eq4}
[a, E_{i}] = [E_{i+k},A]
\,, 
\end{equation}
where $A \in\End V$ is the endomorphism 
corresponding to $a \in\mf g$.
In particular, for $i+k\geq1$, we have
\begin{equation}\label{20180219:eq5}
[a, E_{i}] \equiv \delta_{i+k,1}[F,A]
\,\mod \mc I\,, 
\end{equation}
where $\mc I\subset U(\mf g)$ is the left ideal defined in \eqref{eq:Walg}
(for $\mf l=0$).
\end{lemma}
\begin{proof}
We have,
$$
\begin{array}{l}
\displaystyle{
\vphantom{\Big(}
[a,E_{i}] =
\sum_{j \in I_i} [a,u_j]U^j = 
\sum_{j \in I_i, h \in I} u_h ([a,u_j] | u^h)U^j =
\sum_{j \in I_i, h \in I_{i+k}} u_h ([u^h,a] | u_j)U^j 
}\\
\displaystyle{
\vphantom{\Big(}
= \sum_{j \in I, h \in I_{i+k}} u_h ([u^h,a] | u_j)U^j =
\sum_{h \in I_{i+k}} u_h [U^h,A]  =
[E_{i+k},A]
\,.}
\end{array}
$$
Here, in the second equality we have used the completeness identity \eqref{eq:compl}, in the third equality we have used the invariance of the trace form of $\mf g$ and the fact that $ ([a,u_j] | u^h) = 0 $ for any $ j \in I_i $ when $ h \not\in  I_{i+k}$, in the fourth equality we have used the fact that $ ([u^h,a] | u_j) = 0$ for any $ h \in I_h $ when $ j \not\in I_i $, and in the fifth equality we have again used the completeness identity.
Equation \eqref{20180219:eq5} follows from \eqref{20180219:eq4}
and the definition of the ideal $\mc I$.
\end{proof}

%%%%%%%%%%%%%%%%%%%%%%%%%%%%%%%%%%%%%%%
\section{Recursive formula for the operator \texorpdfstring{$L(z)$}{L(z)} for an aligned pyramid}\label{sec:5b}

%%%

\subsection{The operator $T(z)$}\label{sec:5.2}

Recall from Section \ref{sec:3.1c} that $V_\pm\subset V$ come with the natural splittings
$V = V_+\oplus FV=V_-\oplus F^tV$.
Recall also the definition \eqref{eq:L} of the operator $L(z)$.
\begin{proposition}\label{prop:T}
The following generalized quasideterminant exists:
\begin{equation}\label{eq:def:matrixT}
T(z) =  |z\id_V + F + E_{\mf p} + D_{\mf{m}}  |_{V_-,V_+}
\,\in U(\mf g)[z]\otimes \Hom(V_-,V_+)
\,.
\end{equation}
Moreover, we have
\begin{equation}\label{eq:LT}
L(z)
=
| T(z) |_{V_-^d,V_+^d} \bar1
\,,
\end{equation}
where the RHS is associated to the splittings \eqref{eq:V_-ud} of $V_-$ and $V_+$.
\end{proposition}
\begin{proof}
Clearly, $z\id_V + F + E_{\mf p} + D_{\mf{m}}$
is invertible in $U(\mf g)((z^{-1}))\otimes\End V$ by geometric series expansion.
Hence, by Proposition \ref{prop:quasidet},
in order to prove that the generalized quasideterminant defining $T(z)$ exists,
it suffices to prove that
$$
\id_{FV}(z\id_V + F + E_{\mf p} + D_{\mf{m}})\id_{F^tV}
\,\in U(\mf g)[z]\otimes\Hom(F^t V,FV)
$$
is invertible.
This is the case since $\id_{FV}F\id_{F^tV}$ is obviously invertible,
and 
$$
(\id_{FV}F\id_{F^tV})^{-1}\circ\id_{FV}(z\id_V + E_{\mf p} + D_{\mf{m}})\id_{F^tV}
$$
is nilpotent (having positive degree).
As a consequence, the corresponding geometric series expansion for
$(\id_{FV}(z\id_V + F + E_{\mf p} + D_{\mf{m}})\id_{F^tV})^{-1}$
is finite, so it is polynomial in $z$.
The first claim follows by Proposition \ref{prop:quasidet}.
Equation \eqref{eq:LT} follows
by the decompositions $V=V_-^d\oplus(V_-^u\oplus F^tV)=V_+^d\oplus(V_+^u\oplus FV)$ and the hereditary property \eqref{eq:hered}
of the generalized quasideterminants.
\end{proof}
\begin{remark}\label{rem:BK1}
The operator $T(z)$, in a slightly different form, first appeared in \cite{BK06}.
The precise relation of the operator $T^{BK}(z)$ in \cite{BK06}
and the operator $T(z)$ defined above is 
$$
T^{BK}(z)=z^{-1-X}T(z)z^{X}\,,
$$
where $z^{X}$ is the semisimple endomorphism of $V$
with value $z^k$ on $V[k]$.
\end{remark}

%%%
\subsection{Recursive formula for $T(z)$}\label{sec:5.3}

With the notation of Section \ref{sec:5.1c}
we denote by $T'(z)$ the operator \eqref{eq:def:matrixT}
relative to the pyramid $p'$.
\begin{proposition}\label{prop:recursionT}
The following recursive formula holds for the operator $T(z)$:
\begin{equation}\label{eq:recT}
\begin{array}{l}
\vphantom{\Big(}
T(z)
=
T'(z)\id_{V_-^u}
-\frac1{r_1}[T'(z),\id_{F^tV_-^d}E_{-1}]^1
-T'(z)F^t(z\id_V+E_0+D)\id_{V_-^d}
\,,
\end{array}
\end{equation}
where the bracket $[\cdot\,,\,\cdot]^1$ is on the first factor of $U(\mf g)\otimes\End V$
(cf. Section \ref{sec:3.4b}).
\end{proposition}
\begin{proof}
Let $v\in V_-=V_-^u\oplus V_-^d$.
By the definition of the generalized quasideterminant \eqref{eq:def:matrixT}, 
$T(z)v$ is the unique element of $U(\mf g)[z]\otimes V_+$
of the form
\begin{equation}\label{eq:rec1}
T(z)v
=
(z\id_V+F+E_{\leq0}+D)(v+w)
\,,
\end{equation}
for some
$w\in U(\mf g)((z^{-1}))\otimes F^tV$.

Consider first the case when $v\in V_-^u$.
In this case, $T'(z)v$ is the unique element of $U(\mf g')[z]\otimes V_+$
of the form
\begin{equation}\label{eq:rec2}
T'(z)v
=
(z\id_{V'}+F'+E'_{\leq0}+D')(v+w')
\,,
\end{equation}
for some $w'\in U(\mf g')((z^{-1}))\otimes F^tV'$.
By Lemma \ref{lem:restr},
the RHS of \eqref{eq:rec2} is unchanged if we replace $\id_{V'}$, $F'$, $E'_{\leq0}$ and $D'$
by $\id_{V}$, $F$, $E_{\leq0}$ and $D$ respectively.
Since, 
obviously, $U(\mf g')\subset U(\mf g)$ and $F^tV'\subset F^tV$,
the RHS of \eqref{eq:rec2} is of the form \eqref{eq:rec1}.
Namely, $T'(z)v$ solves the problem defining $T(z)v$,
and therefore they coincide.
Note that the second and third term of the RHS of \eqref{eq:recT}
vanish when applied to $v\in V_-^u$.
Hence,
the restrictions to $V_-^u$ of both sides of equation \eqref{eq:recT}
coincide.

Next, let $v\in V_-^d$.
Note that we have the decomposition (cf. Figures \ref{fig5} and \ref{fig7})
$F^tV=F^tV_-^{d}\oplus F^tV'$.
Hence, the formula \eqref{eq:rec1} defining $T(z)v\in U(\mf g)[z]\otimes V_+$
can be rewritten as
\begin{equation}\label{eq:rec3}
T(z)v
=
(z\id_V+F+E_{\leq0}+D)(v+w_1+w_2)
\,,
\end{equation}
for some
$w_1\in U(\mf g)((z^{-1}))\otimes F^tV_-^{d}$
and $w_2\in U(\mf g)((z^{-1}))\otimes F^t{V'}$.
Applying the projection onto $V_-^d$ (with kernel $V_-^u\oplus F^tV$)
to both sides of equation \eqref{eq:rec3},
and recalling that $V_+\subset V_-^u\oplus F^tV$, we get
$$
0
=
(z\id_V+E_{0}+D)(v)+F(w_1)
\,.
$$
Hence, equation \eqref{eq:rec3} becomes
\begin{equation}\label{eq:rec3b}
T(z)v
=
(z\id_V+F+E_{\leq0}+D)(v
-F^t(z\id_V+E_{0}+D)(v)+w_2)
\,.
\end{equation}
Since $F\circ F^t$ acts as the identity on $V_-^d$, we have
$$
\begin{array}{l}
\vphantom{\Big(}
(z\id_V+F+E_{\leq0}+D)(
F^t(z\id_V+E_{0}+D)(v))
\\
\vphantom{\Big(}
=
(z\id_V+E_{0}+D)(v)
+
(z\id_V+E_{\leq0}+D)
(F^t(z\id_V+E_{0}+D)(v))
\,.
\end{array}
$$
Hence, equation \eqref{eq:rec3b}
can be rewritten as
\begin{equation}\label{eq:rec4}
T(z)v
=
E_{<0}(v)
-
(z\id_V+E_{\leq0}+D)
(F^t(z\id_V+E_{0}+D)(v))
+
(z\id_V+F+E_{\leq0}+D)(w_2)
\,.
\end{equation}
Note that $w_2$ lies in $U(\mf g)((z^{-1}))\otimes F^tV'$.
Moreover, $v\in V_-^d$ and therefore $F^t(z\id_V+E_{0}+D)(v)$ lies in 
$U(\mf g)((z^{-1}))\otimes F^tV_-^{d}$.
Recalling that $F'\id_{F^tV_-^{d}}=0$
and applying Lemmas \ref{lem:restr} and \ref{lem:E}(b), we thus get, from \eqref{eq:rec4}
that $T(z)v$ is the unique element of $U(\mf g)((z^{-1}))\otimes V_+$ of the form
\begin{equation}\label{eq:rec5}
\begin{array}{l}
\displaystyle{
\vphantom{\Big(}
T(z)v
=
-
\frac1{r_1}[E_{\leq0},\id_{F^tV_-^d}E_{-1}]^1(v)
} \\
\displaystyle{
\vphantom{\big(}
-
(z\id_{V'}\!+\!F'\!+\!E'_{\leq0}\!+\!D')
(F^t(z\id_V\!+\!E_{0}\!+\!D)(v))
+
(z\id_{V'}\!+\!F'\!+\!E'_{\leq0}\!+\!D')(w_2)
\,,}
\end{array}
\end{equation}
for some $w_2\in U(\mf g)((z^{-1}))\otimes F^t{V'}$.

Next, let us compute separately the second and third term in the RHS of \eqref{eq:recT}
applied to $v\in V_-^d$.
The second term in the RHS of \eqref{eq:recT} is
(cf. the notation of Section \ref{sec:3.4b})
\begin{equation}\label{eq:rec6}
-\frac1{r_1}[T'(z),\id_{F^tV_-^d}E_{-1}]^1(v)
=
\frac1{r_1}\sum_{i\in I}[u_i,T'(z)\id_{F^tV_-^d}U^i\id_{V_-^d}v]
\,.
\end{equation}
By Lemma \ref{20180426:lem}, we can replace
\begin{equation}\label{20180216:eq3}
\sum_{i\in I}u_i \id_{F^tV_-^d}U^i\id_{V_-^d}
=
\sum_{i\in J}q_iQ^i
\,\in U(\mf g)\otimes\End V\,,
\end{equation}
where $\{q_i\}_{i\in J}$
is a basis of $\mf h:=\Hom(F^t V_-^d,V_-^d)\subset\mf g$,
$\{q^i\}_{i\in J}$
is the dual basis of $\Hom(V_-^d,F^t V_-^d)\subset\mf g$,
and, as usual, $\{Q_i\}_{i\in J}$, $\{Q^i\}_{i\in J}$
denote the same elements, viewed as endomorphisms in $\End V$.
Hence, equation \eqref{eq:rec6} becomes
\begin{equation}\label{eq:rec6b}
-\frac1{r_1}[T'(z),\id_{F^tV_-^d}E_{-1}]^1(v)
=
\frac1{r_1}\sum_{i\in J}[q_i,T'(z)Q^iv]
\,.
\end{equation}
Note that $Q^iv\in F^tV_-^{d}$ for every $i\in J$.
By the definition \eqref{eq:rec1} of $T'(z)$, 
$T'(z)Q^iv$ is the unique element of $U(\mf g')((z^{-1}))\otimes F^tV$
of the form
\begin{equation}\label{eq:rec7}
T'(z)Q^iv
=
(z\id_{V'}+F'+E'_{\leq0}+D')(Q^iv+w_i)
\,,
\end{equation}
for some $w_i\in U(\mf g')((z^{-1}))\otimes F^tV'$.
Combining equations \eqref{eq:rec6b} and \eqref{eq:rec7},
we get
\begin{equation}\label{eq:rec8}
\begin{array}{l}
\displaystyle{
\vphantom{\Big(}
-\frac1{r_1}[T'(z),\id_{F^tV_-^d}E_{-1}]^1(v)
=
\frac1{r_1}\sum_{i\in J}
[q_i,(z\id_{V'}+F'+E'_{\leq0}+D')(Q^iv+w_i)]
} \\
\displaystyle{
\vphantom{\Big(}
=
\frac1{r_1}\sum_{i\in J}
\Big(
[q_i,E_{\leq0}(Q^iv)]
+
[q_i,(z\id_{V'}+F'+E'_{\leq0}+D')(w_i)]
\Big)
} \\
\displaystyle{
\vphantom{\Big(}
=
-\frac1{r_1}
[E_{\leq0},\id_{F^tV_-^d}E_{-1}](v)
+
(z\id_{V'}+F'+E'_{\leq0}+D')
\frac1{r_1}
\sum_{i\in J}
[q_i,w_i]
} \\
\displaystyle{
\vphantom{\Big(}
\,\,\,\,\,\,\,\,\,\,\,\,\,\,\,\,\,\,\,\,\,\,\,\,\,\,\,\,\,\,\,\,\,\,\,\,
\,\,\,\,\,\,\,\,\,\,\,\,\,\,\,\,\,\,\,\,\,\,\,\,\,\,\,\,\,\,\,\,\,\,\,\,
\,\,\,\,\,\,\,\,\,\,\,\,\,\,\,\,\,\,\,\,\,\,\,\,\,\,\,\,\,\,\,\,\,\,\,\,
\,\,\,\,\,\,\,\,\,\,\,\,\,\,\,\,\,\,\,\,\,\,\,\,\,\,\,\,\,\,\,\,\,\,\,\,
\in U(\mf g)((z^{-1}))\otimes V_+
\,.
}
\end{array}
\end{equation}
For the last term of the RHS of \eqref{eq:rec8}
we used the fact $[q_i,E'_{\leq 0}](w_i)=0$ 
for every $i\in J$,
due to Lemma \ref{lem:E2}.

By the definition \eqref{eq:rec1} of $T'(z)$,
the third term in the RHS of \eqref{eq:recT} applied to $v\in V_-^d$ 
is the unique element 
in $U(\mf g')((z^{-1}))\otimes V_+$ of the form
\begin{equation}\label{eq:rec9}
-(z\id_{V'}+F'+E'_{\leq0}+D')
\big(F^t(z\id_V+E_0+D)(v)+\widetilde w\big)
\,,
\end{equation}
for some $\widetilde w\in U(\mf g')((z^{-1}))\otimes F^tV'$.

To conclude, we observe that the sum of \eqref{eq:rec8} and \eqref{eq:rec9}
is an element of $U(\mf g)((z^{-1}))\otimes V_+$
and it has the same form as the RHS of \eqref{eq:rec5}
with 
$$
w_2
=
\frac1{r_1} \sum_{i\in J}[q_i,w_i]-\widetilde w\,\in U(\mf g)((z^{-1}))\otimes F^tV'
\,.
$$
This completes the proof.
\end{proof}

%%%
\subsection{Recursive formula for $L(z)$}\label{sec:5.4}

Recall from \eqref{eq:LT} that  $ L(z) = \widetilde{L}(z)\bar1$, where
\begin{equation}\label{eq:Ltilde}
\widetilde{L}(z) = | T(z) |_{V_-^d,V_+^d}  \in U(\mf g)((z^{-1}))\otimes \Hom(V_-^d,V_+^d)
\,.
\end{equation}
With the notation of Sections \ref{sec:5.1} and \ref{sec:5.1c}
we denote by 
\begin{equation}\label{eq:L'tilde}
\widetilde{L'}(z) = | T'(z) |_{{V'}_-^d,{V'}_+^d}  \in U(\mf g')((z^{-1}))\otimes \Hom({V'}_-^d,{V'}_+^d)
\,,
\end{equation}
the operator \eqref{eq:Ltilde}
relative to the pyramid $p'$.

Also recall from Section \ref{sec:5.1} that we have the following decompositions
\begin{equation}\label{20180215:eq4}
{V'}_-^d=F^tV_-^d\oplus(V_-^u\cap{V'}_-^d)
\quad\text{and}\quad
{V'}_+^d=V_+^d\oplus(V_+^u\cap {V'}_+^d)
\,.
\end{equation}
\begin{proposition}\label{prop:recursionL}
The following recursive formula holds for the operator $\widetilde{L}(z)$:
\begin{equation}\label{eq:recL}
\widetilde{L}(z) = 
-\frac1{r_1}[| \widetilde{L'}(z) |_{F^tV_-^d,V_+^d},\id_{F^tV_-^d}E_{-1}]^1
-| \widetilde{L'}(z) |_{F^tV_-^d,V_+^d}F^t(z\id_V+E_0+D)\id_{V_-^d}
\,,
\end{equation}
where the quasideterminant $| \widetilde{L'}(z) |_{F^tV_-^d,V_+^d}$ is associated to the splittings
in \eqref{20180215:eq4}.
\end{proposition}
\begin{proof}
First, we observe that the quasideterminant $ | \widetilde{L'}(z) |_{F^tV_-^d,V_+^d}$ exists.
Indeed, using the decompositions $V'={V'}_{\pm}^d\oplus{V'}^\Sigma[\neq \pm\frac {d-1}2]$
and \eqref{20180215:eq4},
and the hereditary property \eqref{eq:hered} of quasideterminants we have
$$
| \widetilde{L'}(z) |_{F^tV_-^d,V_+^d}=| z\id_{V'} + F' + E'_{\leq 0} + D' |_{F^tV_-^d, V_+^d}
\,.
$$
The existence of the latter quasideterminant is proved in the same way as for 
$L(z) $ in Theorem \ref{thm:Lexists} (cf. \cite{DSKV17}).

Let $ v \in  V^d_- $. Then, by \eqref{eq:Ltilde} and the definition of quasideterminant, 
$ \widetilde{L}(z)v $ is the unique element in $ U(\mf g) ((z^{-1})) \otimes V_+^d $ of the form
\begin{equation}\label{eq:recL1}
\widetilde{L}(z)v = T(z) (v + w)
\,,
\end{equation}
for some $ w \in U(\mf g )((z^{-1}))\otimes V_-^u$ 
(here we are using the splitting \eqref{eq:V_-ud}).
By Proposition \ref{prop:recursionT} we can rewrite \eqref{eq:recL1} as
\begin{equation}\label{eq:recL2}
\widetilde{L}(z)v = T'(z)w - \frac{1}{r_1}[T'(z), \id_{F^tV_-^d}E_{-1}]^1v - T'(z)F^t(z\id_V + E_0 + D)v
\,.
\end{equation}

Let us now consider the RHS of \eqref{eq:recL}. 
By \eqref{eq:L'tilde}, and using the splittings \eqref{eq:V_-ud} 
(both for $V $ and $V'$), \eqref{eq:V''_-ud}, 
\eqref{eq:V'_-ud} and \eqref{20180215:eq4}, by the hereditary property of generalized quasideterminants we have
\begin{equation}\label{eq:recL3}
| \widetilde{L}'(z) |_{F^t{V}_-^d,{V}_+^d}
= | | T'(z) |_{{V'}_-^d,{V'}_+^d}  |_{F^t{V}_-^d,{V}_+^d}
= | T'(z) |_{F^t{V}_-^d,{V}_+^d}
\,.
\end{equation}
Next, we apply the second summand in the RHS of \eqref{eq:recL} to $ v \in V^d_-$.
By \eqref{eq:recL3} we get that
$ | \widetilde{L}'(z) |_{F^tV_-^d,V_+^d} F^t(z\id_V+E_0+D) v$
is the unique element of $U(\mf g) \otimes V_+^d $ of the form
\begin{equation}\label{eq:recL4}
T'(z) \big( F^t(z\id_V+E_0+D)v + w^{(1)} \big)
 \,,
 \end{equation}
for some $ w^{(1)} \in U(\mf g)((z^{-1}))\otimes V^u_- $
(here we are using the decomposition \eqref{eq:V''_-ud}). 
Let us apply the first summand in the RHS of \eqref{eq:recL} to $ v \in V^d_-$.
Again by \eqref{eq:recL3} and by equation \eqref{20180216:eq3}
(cf. Lemma \ref{20180426:lem}),
we have
\begin{equation}\label{eq:recL5}
\begin{array}{l}
\displaystyle{
\vphantom{\Big(}
[ | \widetilde{L}'(z) |_{F^tV_{-}^d,V_+^d}, \id_{F^tV_-^d}E_-^d]^1 v 
= 
\sum_{i \in J}[ | T'(z) |_{F^tV_{-}^d,V_+^d}, q_iQ^i]^1v 
} \\
\displaystyle{
\vphantom{\Big(}
=
\sum_{i \in J}[ | T'(z) |_{F^tV_{-}^d,V_+^d} Q^iv , q_i ]
\,.}
\end{array}
\end{equation}
By the definition of quasideterminant, for every $ i \in J $, $ | T'(z) |_{F^tV_{-}^d,V_+^d}Q^iv $ 
is the unique element of $ U(\mf g')((z^{-1})) \otimes V_+^d $ 
of the form $ T'(z)(Q^iv + w^{(2)}_i)$, for certain 
$ w^{(2)}_i \in U(\mf g)((z^{-1}))\otimes V^u_- $. 
Therefore, the RHS of \eqref{eq:recL5} becomes 
\begin{equation}\label{eq:recL6}
\begin{array}{l}
\displaystyle{
\vphantom{\Big(}
\sum_{i \in J}[ T'(z) (Q^iv + w^{(2)}_i), q_i]
= } \\
\displaystyle{
\vphantom{\Big(}
=\sum_{i \in J}[ T'(z) , q_i](Q^iv) + \sum_{i \in J}[ T'(z)w^{(2)}_i, q_i] 
= } \\
\displaystyle{
\vphantom{\Big(}
= [ T'(z) , \id_{F^tV_-^d}E_{-1}]^1v + \sum_{i \in J}[ T'(z)w^{(2)}_i, q_i] 
\,.}
\end{array}
\end{equation}
Moreover, for each $ i \in J $, $ T'(z)w^{(2)}_i$ is the unique element 
of $ U(\mf g')[z]\otimes V_+ $ of the form
\begin{equation}
(z\id_{V'} + F' + E'_{\leq 0} + D')(w^{(2)}_i + x^{(2)}_i)\,,
\end{equation}
for some $ x^{(2)}_i \in U(\mf g')((z^{-1})) \otimes F^tV' $. Thus,
\begin{equation}\label{eq:recL8}
\begin{array}{l}
\displaystyle{
\vphantom{\Big(}
\sum_{i \in J}[ T'(z)w^{(2)}_i, q_i]  = \sum_{i \in J}
[ (z\id_{V'} + F' + E'_{\leq 0} + D')(w^{(2)}_i + x^{(2)}_i), q_i] 
= } \\
\displaystyle{
\vphantom{\Big(}
= \sum_{i \in J}[ E'_{\leq 0}, q_i] (w^{(2)}_i + x^{(2)}_i)  
+ (z\id_{V'} + F' + E'_{\leq 0} + D') \sum_{i \in J}[ w^{(2)}_i + x^{(2)}_i, q_i] 
= } \\
\displaystyle{
\vphantom{\Big(}
= (z\id_{V'} + F' + E'_{\leq 0} + D') \sum_{i \in J}[ w^{(2)}_i + x^{(2)}_i, q_i] 
%\in (z\id_{V'} + F' + E'_{\leq 0} + D') \big( U(\mf g)((z^{-1})) \otimes \big( V^u_- + F^tV'\big) \big)
\,.}
\end{array}
\end{equation}
Indeed, by Lemma \ref{lem:E2} we have that $ [ E'_{\leq 0}, q_i] (w^{(2)}_i + x^{(2)}_i) = 0 $ for every 
$ i \in J$.
Moreover, 
\begin{equation}\label{eq:recL8b}
(z\id_{V'} + F' + E'_{\leq 0} + D') \frac1{r_1}\sum_{i \in J}[ w^{(2)}_i + x^{(2)}_i, q_i] 
= T'(z) \frac1{r_1} \sum_{i \in J}[ w^{(2)}_i, q_i] 
%\in (z\id_{V'} + F' + E'_{\leq 0} + D') \big( U(\mf g)((z^{-1})) \otimes \big( V^u_- + F^tV'\big) \big)
\,.
\end{equation}
Indeed, by applying the RHS of \eqref{eq:recL} to $ v \in V_-^d $ and combining it with \eqref{eq:recL4}, \eqref{eq:recL6} and \eqref{eq:recL8}
we can conclude that
$$
(z\id_{V'} + F' + E'_{\leq 0} + D') \frac1{r_1} \sum_{i \in J}[ w^{(2)}_i + x^{(2)}_i, q_i]  
\in U(\mf g)((z^{-1}))\otimes V_+
\,.
$$
On the other hand, $ T'(z) \frac1{r_1}\sum_{i \in J}[ w^{(2)}_i, q_i]  $
is the unique element of $ U(\mf g)((z^{-1})) \otimes V_+ $ of the form
$$
(z\id_{V'} + F' + E'_{\leq 0} + D') \big(  \frac1{r_1}\sum_{i \in J}[ w^{(2)}_i , q_i] + \widetilde{x}_i \big)
\,,
$$
for some $ \widetilde{x}_i \in U(\mf g)((z^{-1}))\otimes F^tV' $.
Hence, \eqref{eq:recL8b} holds since $ \frac1{r_1}\sum_{i \in J}[x^{(2)}_i, q_i] 
\in U(\mf g)((z^{-1}))\otimes F^tV'$ as well.

Combining \eqref{eq:recL4}, \eqref{eq:recL6}, \eqref{eq:recL8} and \eqref{eq:recL8b}
we get that the RHS of equation \eqref{eq:recL}  applied to $ v \in V^d_- $
is equal to
\begin{equation}\label{eq:recL9}
\begin{array}{l}
\displaystyle{
\vphantom{\Big(}
-\frac1{r_1}[ T'(z), \id_{F^tV_-^d}E_{-1}]^1v - T'(z)F^t(z\id_V+E_0+D)v 
\,- } \\
\displaystyle{
\vphantom{\Big(}
- T'(z)w^{(1)} - T'(z) \frac1{r_1}\sum_{i \in J}[ w^{(2)}_i, q_i] 
\in U(\mf g)((z^{-1}))\otimes V_+^d
\,.}
\end{array}
\end{equation}
To conclude, we observe that \eqref{eq:recL9} has the same form as the RHS of \eqref{eq:recL2}
with 
$$
w
=
\frac1{r_1}\sum_{i \in J}[ q_i, w^{(2)}_i] - w^{(1)}  \,\in U(\mf g)((z^{-1}))\otimes V_-^u
\,.
$$
This completes the proof.
\end{proof}

%%%%%%%%%%%%%%%%%%%%%%%%%%%%%%%%%%%%%%%
\section{The operator \texorpdfstring{$W(z)$}{W(z)} for an aligned pyramid}\label{sec:5c}

%%%
\subsection{Recursive definition for the operator $W(z)$}\label{sec:5.5}

In this section we construct recursively,
for a right aligned pyramid, a matrix $W(z)$ as in \eqref{eq:W}.
We shall prove in Theorem \ref{thm:20180208} below
that coefficients of the matrix $W(z)$ indeed lie in the $W$-algebra $W(\mf g,f)$,
while in the next Section \ref{sec:5.6}
we shall prove that the matrix $W(z)$ has Premet's form \eqref{eq:W}.

Consider first the case when $f=0$.
In this case the partition $\lambda$ has all parts equal to $1$,
$p_1=1$, and the corresponding pyramid consists of a single column.
Moreover, we have $W(\mf g,0)=U(\mf g)$, and $V=V[0]=V_-=V_+$.
We let, in this case,
\begin{equation}\label{20180208:eq1}
W(z)
=
z\id_V+E
\,\in U(\mf g)[z]\otimes\End V
\,,
\end{equation}
where $E$ is the matrix \eqref{eq:E}.

Next, let $p$ be a right aligned pyramid with $p_1>1$.
We let $W(z)=\widetilde W(z)\bar1$,
where $\widetilde W(z) \in U(\mf g)[z]\otimes\Hom(V_-,V_+)$,
and $\bar 1$ is the image of $1$ in the quotient $U(\mf g)/\mc I$ (cf. \eqref{eq:Walg}).
Denote, as in Section \ref{sec:5.1}, 
by $p'$ the pyramid obtained from $p$ by removing the leftmost column.
Assume by induction that the matrix 
$\widetilde W'(z) \in U(\mf g')[z]\otimes\Hom(V'_-,V_+)$
has been defined.
Then, we define $\widetilde W(z)$
via the following recursive formula (cf. \eqref{eq:recT}):
\begin{equation}\label{eq:recW}
\begin{array}{l}
\vphantom{\Big(}
\widetilde W(z)
=
\widetilde W'(z)\id_{V_-^u}
-
\frac1{r_1}[\widetilde W'(z),\id_{F^tV_-^d}E_{-1}]^1-\widetilde W'(z)F^t(z\id_V+E_0+D)\id_{V_-^d}
\\
\vphantom{\Big(}
\qquad\qquad
+\Res_x x^{-1}\widetilde W'(z)
\id_{V_-^u}
(1+x^{-1}F)^{-1}\widetilde W'(x)F^t\id_{V_-^d}
\,,
\end{array}
\end{equation}
where $(1+x^{-1}F)^{-1}$ is expanded as a geometric series.
The recursive definition \eqref{eq:recW} of $\widetilde{W}(z)$
first appeared, in matrix components, in \cite{BK08}.
Recall from \eqref{eq:mcF} 
that $\{U^i\}_{i\in\mc F}$ is a basis of $\Hom(V_-,V_+)\subset\End V$.
\begin{lemma}\label{20180222:lem1}
Let, as in Lemma \ref{lem:gfk}, $\{u_i\}_{i\in\mc F(h,k)}$ be a basis of $\mf g_0^f(h,k)$
for every $0\leq h,k\leq p_1-1$,
and let $\mc F=\sqcup_{h,k}\mc F(h,k)$.
Let $\{U^i\}_{i\in\mc F(h,k)}$ be the dual (w.r.t. the trace form) basis of
$\Hom(V_-\cap F^kV_+,V_+\cap(F^t)^hV_-)$.
If we decompose the matrix $ \widetilde W(z) $ in the following form
\begin{equation}\label{20180222:eq1}
\widetilde W(z) = z \id_{V_+}(1 + zF^t)^{-1}\id_{V_-} 
+ \sum_{i \in \mc F} \widetilde w_i(z) U^i
\,,
\end{equation}
where $ \widetilde w_i(z) \in U(\mf g)[z] $, 
then the degree of the polynomials $\widetilde w_i(z)$
are subject to the following restrictions:
\begin{equation}\label{20180222:eq2}
\deg\big(
\widetilde w_i(z)
\big)
\leq \min\{h,k\}
\,\text{ for every }\,
i\in\mc F(h,k)
\,.
\end{equation}
Equivalently (cf. equation \eqref{20180502:eq3}),
\begin{equation}\label{20180502:eq3b}
\begin{array}{l}
\displaystyle{
\vphantom{\Big(}
\id_{V_+\cap(F^t)^hV_-}\widetilde{W}(z)\id_{V_-\cap F^k V_+}
} \\
\displaystyle{
\vphantom{\Big(}
=
-\delta_{h,k}(-z)^{k+1}
%\id_{V_+\cap(F^t)^kV_-}
(F^t)^k\id_{V_-\cap F^k V_+}
+
\sum_{\ell=0}^{\min\{h,k\}} \sum_{i\in\mc F(h,k)}
(-z)^\ell \widetilde{w}_{i,\ell} U^i
\,,}
\end{array}
\end{equation}
for some elements $\widetilde{w}_{i.\ell}\in U(\mf g)$.
\end{lemma}
\begin{proof}
The equivalence of \eqref{20180222:eq1}-\eqref{20180222:eq2}
and \eqref{20180502:eq3b} is obvious.
We shall prove \eqref{20180502:eq3b} by induction on the number $p_1$ 
of columns of the pyramid $p$.
For $p_1=1$ we have $V_+=V_-=V$ and $\widetilde{W}(z)=z\id_V+E$
by \eqref{20180208:eq1}.
In this case the claim \eqref{20180502:eq3b} is obvious.
For $p_1>1$, we shall consider separately the four cases
(1) $h,k<p_1-1$,
(2) $h=p_1-1,\,k<p_1-1$, 
(3) $h<p_1-1,\,k=p_1-1$,
and (4) $h=k=p_1-1$.
Recall that $V_-\cap F^kV_+\subset V_-^u$ for $k<p_1-1$,
and $V_-\cap F^{p_1-1}V_+=V_-^d$.
If $h,k<p_1-1$,
by the inclusions \eqref{20180502:eq4} and \eqref{20180502:eq5}
we have
\begin{equation}\label{20180507:eq2}
\id_{V_+\cap(F^t)^hV_-}
=
\id_{V_+\cap(F^t)^hV_-}
\id_{V_+\cap(F^t)^hV'_-}
\,\,\text{ and }\,\,
\id_{V_-\cap F^k V_+}
=
\id_{V'_-\cap F^k V_+}
\id_{V_-\cap F^k V_+}
\,.
\end{equation}
Hence,
by the recursive definition \eqref{eq:recW} of $\widetilde{W}(z)$
and the inductive assumption,
we have
\begin{equation}\label{20180504:eq2}
\begin{array}{l}
\displaystyle{
\vphantom{\Big(}
\id_{V_+\cap(F^t)^hV_-}\widetilde{W}(z)\id_{V_-\cap F^k V_+}
} \\
\displaystyle{
\vphantom{\Big(}
=
\id_{V_+\cap(F^t)^hV_-}
\id_{V_+\cap(F^t)^hV'_-}
\widetilde{W'}(z)
\id_{V'_-\cap F^k V_+}
\id_{V_-\cap F^k V_+}
} \\
\displaystyle{
\vphantom{\Big(}
=
-\delta_{h,k}(-z)^{k+1}
%\id_{V_+\cap(F^t)^hV_-}
(F^t)^k
\id_{V_-\cap F^k V_+}
\!\!\!
+
\!\!\!\!\!\!
\sum_{\ell=0}^{\min\{h,k\}} 
\!\!\!\!\!\!\!
\sum_{\ i\in\mc F'(h,k)}
\!\!\!\!\!\!
(-z)^\ell \widetilde{w}'_{i,\ell} 
\id_{V_+\cap(F^t)^hV_-}{U'}^i\id_{V_-\cap F^k V_+}
\,.}
\end{array}
\end{equation}
The second term in the RHS of \eqref{20180504:eq2}
is a polynomial in $z$ of degree less than or equal to $\min\{h,k\}$.
The coefficient of $(-z)^\ell$ in such term
lies in $U(\mf g)\otimes\Hom(V_-\cap F^k V_+,V_+\cap(F^t)^hV_-)$,
hence it can always be written as 
$\sum_{i\in\mc F(h,k)} \widetilde{w}_{i,\ell} U^i$ for some $\widetilde{w}_{i,\ell}\in U(\mf g)$,
proving \eqref{20180502:eq3b}.
Next, 
consider the case $h=p_1-1$, $k<p_1-1$.
In this case the argument is similar.
Recall that $V_+\cap (F^t)^{p_1-1}V_-=V_+^d$
and hence, by the inclusions \eqref{20180504:eq4}
we have
\begin{equation}\label{20180507:eq7}
\id_{V_+\cap(F^t)^{p_1-1}V_-}
=
\id_{V_+^d}
=
\id_{V_+^d}
\id_{V_+\cap(F^t)^{p_1-2}V'_-}
\,.
\end{equation}
Hence,
by the recursive definition \eqref{eq:recW} of $\widetilde{W}(z)$
and the inductive assumption,
we have
\begin{equation}\label{20180504:eq3}
\begin{array}{l}
\displaystyle{
\vphantom{\Big(}
\id_{V_+^d}\widetilde{W}(z)\id_{V_-\cap F^k V_+}
} \\
\displaystyle{
\vphantom{\Big(}
=
\id_{V_+^d}
\id_{V_+\cap(F^t)^{p_1-2}V'_-}
\widetilde{W'}(z)
\id_{V'_-\cap F^k V_+}
\id_{V_-\cap F^k V_+}
} \\
\displaystyle{
\vphantom{\Big(}
=
-\delta_{k,p_1-2}(-z)^{p_1-1}
%\id_{V_+\cap(F^t)^hV_-}
\id_{V_+^d}
(F^t)^{p_1-1}
\id_{V_-\cap F^{p_1-2} V_+}
} \\
\displaystyle{
\vphantom{\Big(}
+
\sum_{\ell=0}^k
\sum_{\ i\in\mc F'(p_1-2,k)}
(-z)^\ell \widetilde{w}'_{i,\ell} 
\id_{V_+^d}{U'}^i\id_{V_-\cap F^k V_+}
\,.}
\end{array}
\end{equation}
Note that the first term in the RHS of \eqref{20180504:eq3}
vanishes since, loosely speaking, 
a power of $F^t$ cannot map a row of length $p_1-1$
to a row of length $p_1$. 
For the second term in the RHS of \eqref{20180504:eq3}
the argument is the same as for the second term in the RHS of \eqref{20180504:eq2}:
it is a polynomial in $z$ of degree bounded by $k=\min\{h,k\}$,
with coefficients in $U(\mf g)\otimes\Hom(V_-\cap F^k V_+,V_+^d)$.
Next,
consider the case $h<p_1-1$, $k=p_1-1$.
Recall that $V_d\cap F^{p_1-1}V_+=V_-^d$, and by \eqref{20180504:eq4}, 
we have
\begin{equation}\label{20180507:eq3}
\id_{F^tV_-^d}
=
\id_{{V'}_-^d}
\id_{F^tV_-^d}
\,.
\end{equation}
By the recursive equation \eqref{20180222:eq1} and the inductive assumption, we have
\begin{equation}\label{20180507:eq1}
\begin{array}{l}
\displaystyle{
\vphantom{\Big(}
\id_{V_+\cap(F^t)^hV_-}\widetilde{W}(z)\id_{V_-^d}
} \\
\displaystyle{
\vphantom{\Big(}
=
-
\frac1{r_1}
\id_{V_+\cap(F^t)^hV_-}
[\id_{V_+\cap(F^t)^hV'_-}\widetilde W'(z)\id_{{V'}_-^d},
\id_{F^tV_-^d}
E_{-1}]^1
\id_{V_-^d}
} \\
\displaystyle{
\vphantom{\Big(}
-
\id_{V_+\cap(F^t)^hV_-}
\id_{V_+\cap(F^t)^hV'_-}\widetilde W'(z)\id_{{V'}_-^d}
F^t(z\id_V+E_0+D)
\id_{V_-^d}
} \\
\displaystyle{
\vphantom{\Big(}
+
\sum_{j=0}^{p_1-2}
\Res_x x^{-1}
\id_{V_+\cap(F^t)^hV_-}
\id_{V_+\cap(F^t)^hV'_-}\widetilde W'(z)\id_{V'_-\cap F^jV_+}
} \\
\displaystyle{
\vphantom{\Big(}
\qquad\qquad\qquad\qquad
\times
\id_{V_-\cap F^jV_+}
(1+x^{-1}F)^{-1}\widetilde W'(x)F^t
\id_{V_-^d}
} \\
\displaystyle{
\vphantom{\Big(}
=
-
\frac1{r_1}
\sum_{\ell=0}^{h} \sum_{i\in\mc F'(h,p_1-2)}
(-z)^\ell
\id_{V_+\cap(F^t)^hV_-}
[
{\widetilde{w}'}_{i,\ell} {U'}^i
,
\id_{F^tV_-^d}
E_{-1}]^1
\id_{V_-^d}
} \\
\displaystyle{
\vphantom{\Big(}
-
\id_{V_+\cap(F^t)^hV_-}
\Big(
-\delta_{h,p_1-2}(-z)^{p_1-1}
(F^t)^{p_1-2}\id_{V'_-\cap F^{p_1-2} V_+}
} \\
\displaystyle{
\vphantom{\Big(}
+
\sum_{\ell=0}^{h} \sum_{i\in\mc F'(h,p_1-2)}
(-z)^\ell {\widetilde{w}'}_{i,\ell} {U'}^i
\Big)
F^t(z\id_V+E_0+D)
\id_{V_-^d}
} \\
\displaystyle{
\vphantom{\Big(}
+
\sum_{j=0}^{p_1-2}
\Res_x x^{-1}
\id_{V_+\cap(F^t)^hV_-}
\Big(
-\delta_{h,j}(-z)^{j+1}
(F^t)^j\id_{V'_-\cap F^j V_+}
} \\
\displaystyle{
\vphantom{\Big(}
+
\sum_{\ell=0}^{\min\{h,j\}} \sum_{i\in\mc F'(h,j)}
(-z)^\ell {\widetilde{w}'}_{i,\ell} {U'}^i
\Big)
\id_{V_-\cap F^jV_+}
(1+x^{-1}F)^{-1}\widetilde W'(x)F^t
\id_{V_-^d}
\,,}
\end{array}
\end{equation}
For the first equality of \eqref{20180507:eq1}
we used the recursive definition \eqref{eq:recW} of $\widetilde{W}(z)$,
and the identities \eqref{20180507:eq4}, \eqref{20180507:eq2} and \eqref{20180507:eq3}.
For the second equality of \eqref{20180507:eq1}
we used the inductive assumption \eqref{20180502:eq3b}.
Having to prove \eqref{20180502:eq3b},
we are only interested in the terms in the RHS of \eqref{20180507:eq1}
with powers of $z$ greater than $h=\min\{h,k\}$.
We wish to prove that the sum of all such terms vanishes.
They are
\begin{equation}\label{20180507:eq5}
\begin{array}{l}
\displaystyle{
\vphantom{\Big(}
\delta_{h,p_1-2}
(-z)^{p_1-1}
\id_{V_+\cap(F^t)^hV_-}
(F^t)^{p_1-1}
(z\id_V+E_0+D)
\id_{V_-^d}
} \\
\displaystyle{
\vphantom{\Big(}
+
(-z)^{h+1}
\sum_{i\in\mc F'(h,p_1-2)}
\id_{V_+\cap(F^t)^hV_-}
 {\widetilde{w}'}_{i,h} {U'}^i
F^t \id_{V_-^d}
} \\
\displaystyle{
\vphantom{\Big(}
-
(-z)^{h+1}
\Res_x x^{-1}
\id_{V_+\cap(F^t)^hV_-}
(F^t)^h
\id_{V_-\cap F^hV_+}
(1+x^{-1}F)^{-1}\widetilde W'(x)F^t
\id_{V_-^d}
\,.}
\end{array}
\end{equation}
The first term in \eqref{20180507:eq5} is identically zero:
for $h=p_1-2$, we have 
$$
\id_{V_+\cap(F^t)^hV_-}(F^t)^{p_1-1}=0
\,.
$$
We can compute the third term in \eqref{20180507:eq5}
by computing explicitly the residue in $x$.
Since $\widetilde W'(x)=\id_{V_+}\widetilde W'(x)$, we have
\begin{equation}\label{20180507:eq6}
\begin{array}{l}
\displaystyle{
\vphantom{\Big(}
-
(-z)^{h+1}
\Res_x x^{-1}
\id_{V_+\cap(F^t)^hV_-}
(F^t)^h
\id_{V_-\cap F^hV_+}
(1+x^{-1}F)^{-1}\widetilde W'(x)F^t
\id_{V_-^d}
} \\
\displaystyle{
\vphantom{\Big(}
=
(-z)^{h+1}
\Res_x (-x)^{-h-1}
\id_{V_+\cap(F^t)^hV_-}
(F^t)^h
\id_{V_-\cap F^hV_+}
F^h
\widetilde W'(x)
F^t
\id_{V_-^d}
} \\
\displaystyle{
\vphantom{\Big(}
=
(-z)^{h+1}
\Res_x (-x)^{-h-1}
\id_{V_+\cap(F^t)^hV_-}
\widetilde W'(x)
F^t
\id_{V_-^d}
} \\
\displaystyle{
\vphantom{\Big(}
=
-(-z)^{h+1}
\sum_{i\in\mc F'(h,p_1-2)}
\id_{V_+\cap(F^t)^hV_-}
\widetilde{w}'_{i,h}{U'}^i
F^t
\id_{V_-^d}
\,.}
\end{array}
\end{equation}
For the last equality we used the inductive assumption \eqref{20180502:eq3b}
on $\widetilde W'(x)$ and we computed the residue in $x$.
Combining \eqref{20180507:eq5} and \eqref{20180507:eq6},
we get that \eqref{20180507:eq5} vanishes, as claimed.
Finally, we consider the case $h=k=p_1-1$.
The computation is similar to the previous case.
Using \eqref{20180507:eq7} in place of \eqref{20180507:eq2},
the same computations leading to \eqref{20180507:eq1}
lead to
\begin{equation}\label{20180507:eq8}
\begin{array}{l}
\displaystyle{
\vphantom{\Big(}
\id_{V_+^d}\widetilde{W}(z)\id_{V_-^d}
} \\
%\displaystyle{
%\vphantom{\Big(}
%=
%-
%\frac1{r_1}
%\id_{V_+^d}
%[\id_{V_+\cap(F^t)^{p_1-2}V'_-}\widetilde W'(z)\id_{{V'}_-^d},
%\id_{F^tV_-^d}
%E_{-1}]^1
%\id_{V_-^d}
%} \\
%\displaystyle{
%\vphantom{\Big(}
%-
%\id_{V_+^d}
%\id_{V_+\cap(F^t)^{p_1-2}V'_-}\widetilde W'(z)\id_{{V'}_-^d}
%F^t(z\id_V+E_0+D)
%\id_{V_-^d}
%} \\
%\displaystyle{
%\vphantom{\Big(}
%+
%\sum_{j=0}^{p_1-2}
%\Res_x x^{-1}
%\id_{V_+^d}
%\id_{V_+\cap(F^t)^{p_1-2}V'_-}\widetilde W'(z)\id_{V'_-\cap F^jV_+}
%} \\
%\displaystyle{
%\vphantom{\Big(}
%\qquad\qquad\qquad\qquad
%\times
%\id_{V_-\cap F^jV_+}
%(1+x^{-1}F)^{-1}\widetilde W'(x)F^t
%\id_{V_-^d}
%} \\
\displaystyle{
\vphantom{\Big(}
=
-
\frac1{r_1}
\sum_{\ell=0}^{p_1-2} \sum_{i\in\mc F'(p_1-2,p_1-2)}
(-z)^\ell
\id_{V_+^d}
[
{\widetilde{w}'}_{i,\ell} {U'}^i
,
\id_{F^tV_-^d}
E_{-1}]^1
\id_{V_-^d}
} \\
\displaystyle{
\vphantom{\Big(}
-
\id_{V_+^d}
\Big(
-(-z)^{p_1-1}
(F^t)^{p_1-2}\id_{V'_-\cap F^{p_1-2} V_+}
} \\
\displaystyle{
\vphantom{\Big(}
+
\sum_{\ell=0}^{p_1-2} \sum_{i\in\mc F'(p_1-2,p_1-2)}
(-z)^\ell {\widetilde{w}'}_{i,\ell} {U'}^i
\Big)
F^t(z\id_V+E_0+D)
\id_{V_-^d}
} \\
\displaystyle{
\vphantom{\Big(}
+
\sum_{j=0}^{p_1-2}
\Res_x x^{-1}
\id_{V_+^d}
\Big(
-\delta_{j,p_1-2}(-z)^{j+1}
(F^t)^j\id_{V'_-\cap F^j V_+}
} \\
\displaystyle{
\vphantom{\Big(}
+
\sum_{\ell=0}^{j} \sum_{i\in\mc F'(p_1-2,j)}
(-z)^\ell {\widetilde{w}'}_{i,\ell} {U'}^i
\Big)
\id_{V_-\cap F^jV_+}
(1+x^{-1}F)^{-1}\widetilde W'(x)F^t
\id_{V_-^d}
\,.}
\end{array}
\end{equation}
As before, 
in order to prove \eqref{20180502:eq3b}
we are only interested in the terms in the RHS of \eqref{20180507:eq8}
with powers of $z$ greater than $p_1-1=\min\{h,k\}$.
There is only one:
\begin{equation}\label{20180507:eq9}
-(-z)^{p_1}
\id_{V_+^d}
(F^t)^{p_1-2}\id_{V'_-\cap F^{p_1-2} V_+}
F^t
\id_{V_-^d}
=
-(-z)^{p_1}
\id_{V_+^d}
(F^t)^{p_1-1}
\id_{V_-^d}
\,.
\end{equation}
Since \eqref{20180507:eq9}
agrees with the first term in the RHS of \eqref{20180502:eq3b},
the claim is proved.
\end{proof}
\begin{theorem}\label{thm:20180208}
For every $a\in\mf g[\geq1]$ we have
\begin{equation}\label{20180208:eq2}
[a,\widetilde W(z)]
\,\in\,
\mc I[z]\otimes\Hom(V_-,V_+)
\,.
\end{equation}
Hence, $W(z)$ has coefficients with entries in the $W$-algebra $W(\mf g,f)$:
\begin{equation}\label{20180208:eq2b}
W(z)
\,\in\,
W(\mf g,f)[z]\otimes\Hom(V_-,V_+)
\,.
\end{equation}
\end{theorem}
\begin{proof}
We shall prove condition \eqref{20180208:eq2} with a two step induction.
First, we prove the two base cases $p_1=1$ and $p_1=2$.
When $p_1=1$, we have $\mf g=\mf g[0]$, and so condition \eqref{20180208:eq2}
is empty.

Let us consider the case $p_1=2$.
In this case the $\Gamma$-grading of $V$ is as in the following picture:

%Vprime
\begin{figure}[H]
\setlength{\unitlength}{0.15in}
\begin{picture}(12,11)
\setlength\fboxsep{0pt}

\put(3,3){\framebox(2,5){$\vdots$}}
\put(5,3){\framebox(2,8){$\vdots$}}

\put(2,2){\vector(1,0){7}}
\put(10,1){$x$}

\put(4,1.8){\line(0,1){0.4}}
\put(5,1.6){\line(0,1){0.8}}
\put(6,1.8){\line(0,1){0.4}}

\put(4.7,0.6){0}
\put(3,0.6){$-\frac12$}
\put(5.7,0.6){$\frac12$}

\put(-3,5.5){${V}_-^d=V[-\frac12]$}
\put(2,9.5){${V}_-^u$}

\put(8,7.5){${V}_+=V[\frac12]=V'$}
\end{picture}
\caption{ }
\label{fig11}
\end{figure}
\noindent
where $\dim(V_-^d)=r_1$ and $\dim(V_+)=r$.
The $\Gamma$-grading of $\mf g$ is
$\mf g=\mf g[-1]\oplus\mf g[0]\oplus\mf g[1]$,
where
\begin{equation}\label{20180219:eq1}
\begin{array}{l}
\displaystyle{
\vphantom{\Big(}
\mf g[-1]=\Hom(V_+,V_-^d)
\,,\,\,
\mf g[1]=\Hom(V_-^d,V_+)
\,,} \\
\displaystyle{
\vphantom{\Big(}
\mf g[0]=\mf g'\oplus\End(V_-^d)
\,,\,\,
\mf g'=\End(V_+)
\,.}
\end{array}
\end{equation}
From equation \eqref{20180208:eq1} we have that
\begin{equation}\label{20180219:eq6}
\widetilde{W}'(z)=z\id_{V'}+E'\in U(\mf g')[z]\otimes\End V'
\,.
\end{equation}
Let $a\in\mf g[1]$.
By the recursive formula \eqref{eq:recW}, $\widetilde W(z)\id_{V_-^u}=\widetilde W'(z)\id_{V_-^u}$.
Hence, 
by \eqref{20180219:eq5} and \eqref{20180219:eq6}, we have
$$
[a, \widetilde W(z)\id_{V_-^u}] 
= 
[a, E'\id_{V_-^u}] 
= 
[a, E_0\id_{V_-^u}] 
\equiv
[F,A]\id_{V_-^u}
=0
\,\mod\mc I
\,,
$$
since $V_-^u \subset \Ker F\cap\ker A$.
On the other hand, from equations \eqref{eq:recW} and \eqref{20180219:eq6}, 
we have
\begin{equation}\label{20180217:eq1}
\begin{array}{l}
\vphantom{\Big(}
[a, \widetilde W(z)\id_{V_-^d}]
=
-
\frac1{r_1} [a, [E',\id_{F^tV_-^d}E_{-1}]^1]
\\
\vphantom{\Big(}
-
[a,E']F^t(z\id_{V} + E_0 + D)\id_{V_-^d}
-
(z\id_{V'} + E')F^t[a, E_0]\id_{V_-^d} 
\\
\vphantom{\Big(}
+
\Res_x x^{-1}[a,E']\id_{V_-^u}(1 + x^{-1}F)^{-1}(x\id_{V'} \!+\! E')F^t\id_{V_-^d}
\\
\vphantom{\Big(}
+
(z\id_{V'} + E')
\id_{V_-^u}
[a, E']F^t\id_{V_-^d} 
\,.
\end{array}
\end{equation}
Note that $E'=\id_{V_+}E_0$.
Hence, by Lemma \ref{lem:E}(b), we have
\begin{equation}\label{20180601:eq1}
[E',\id_{F^tV_-^d}E_{-1}]^1
=
-r_1E_{-1}
\,,
\end{equation}
and applying equation \eqref{20180219:eq4} we get
\begin{equation}\label{20180219:eq7}
-\frac1{r_1}[a,[E',\id_{F^tV_-^d}E_{-1}]^1]
= 
[E_0,A]
\,.
\end{equation}
By \eqref{20180219:eq4}, we also have
\begin{equation}\label{20180219:eq8}
[a, E_0]\equiv [F,A]\mod\mc I
\,\,,\,\,\,\,
[a, E']
\equiv 
[F,A]\id_{V_+}
=
-AF
\mod\mc I
\,.
\end{equation}
Moreover, 
again by \eqref{20180219:eq4}, we have
\begin{equation}\label{20180219:eq9}
\begin{array}{l}
\displaystyle{
\vphantom{\Big(}
[a,E']F^t(z\id_{V} + E_0 + D)\id_{V_-^d}
=
\sum_{i\in I_0}
[a,u_i]U^iF^t(z\id_{V} + E_0 + D)\id_{V_-^d}
} \\
\displaystyle{
\vphantom{\Big(}
\equiv
-AFF^t(z\id_{V} + E_0 + D)\id_{V_-^d}
+\sum_{i\in I_0}
U^iF^t[[a,u_i],E_0]\id_{V_-^d}
} \\
\displaystyle{
\vphantom{\Big(}
\equiv
-A(z\id_{V} + E_0 + D)\id_{V_-^d}
+\sum_{i\in I_0}
U^iF^t[F,[A,U_i]]\id_{V_-^d}
} \\
\displaystyle{
\vphantom{\Big(}
=
-A(z\id_{V} + E_0 + D)\id_{V_-^d}
+\sum_{i\in I}
U^i\id_{F^tV}[A,U_i]\id_{V_-^d}
} \\
\displaystyle{
\vphantom{\Big(}
=
-A(z\id_{V} + E_0 + D)
-r_1
A
\mod\mc I
\,.}
\end{array}
\end{equation}
For the last equality we used Lemma \ref{20180216:lem1}.
Similarly, we have
\begin{equation}\label{20180219:eq10}
\begin{array}{l}
\displaystyle{
\vphantom{\Big(}
\Res_x x^{-1}[a,E']\id_{V_-^u}(1 + x^{-1}F)^{-1}(x\id_{V'} \!+\! E')F^t\id_{V_-^d}
} \\
\displaystyle{
\vphantom{\Big(}
=
\sum_{i\in I_0}
\Res_x x^{-1}[a,u_i]U^i\id_{V_-^u}(1 + x^{-1}F)^{-1}(x\id_{V'} \!+\! E')F^t\id_{V_-^d}
} \\
\displaystyle{
\vphantom{\Big(}
\equiv
-
\Res_x x^{-1}AF\id_{V_-^u}(1 + x^{-1}F)^{-1}(x\id_{V'} \!+\! E')F^t\id_{V_-^d}
} \\
\displaystyle{
\vphantom{\Big(}
+
\sum_{i\in I_0}
\Res_x x^{-1}U^i\id_{V_-^u}(1 + x^{-1}F)^{-1}[[a,u_i],E']F^t\id_{V_-^d}
} \\
\displaystyle{
\vphantom{\Big(}
=
\sum_{i\in I_0}
U^i\id_{V_-^u}[[a,u_i],E']F^t\id_{V_-^d}
} \\
\displaystyle{
\vphantom{\Big(}
\equiv
-\sum_{i\in I_0}
U^i\id_{V_-^u}[A,U_i]FF^t\id_{V_-^d}
} \\
\displaystyle{
\vphantom{\Big(}
=
-\sum_{i\in I}
U^i\id_{V_-^u}[A,U_i]\id_{V_-^d}
} \\
\displaystyle{
\vphantom{\Big(}
=
(r-r_1)A
\mod\mc I
\,.}
\end{array}
\end{equation}
Again, 
for the last equality we used Lemma \ref{20180216:lem1}.
Combining equations \eqref{20180217:eq1}, 
\eqref{20180219:eq7}, \eqref{20180219:eq8}, \eqref{20180219:eq9} and \eqref{20180219:eq10},
we get
\begin{equation}\label{20180219:eq11}
\begin{array}{l}
\displaystyle{
\vphantom{\Big(}
[a, \widetilde W(z)\id_{V_-^d}]
\equiv
[E_0,A]
+
A(z\id_{V} + E_0 + D)
+
r_1A
-
(z\id_{V'} + E')F^tFA
} \\
\displaystyle{
\vphantom{\Big(}
+
(r-r_1)A
-
(z\id_{V'} + E')
\id_{V_-^u}
A
\mod\mc I
\,.}
\end{array}
\end{equation}
It is now easy to check that the RHS of \eqref{20180219:eq11}
vanishes, using the facts that $D|_{V_-^d}=-r$,
$F^tF+\id_{V_-^u}=\id_{V_+}$ and $E_0\id_{V_+}=E'$.

Next, we prove the claim for $p_1>2$,
assuming that \eqref{20180208:eq2} holds for $ p' $.
Note that we can decompose 
$$
\mf g[\geq 1]  =  \mf{g'}[\geq 1] \oplus \Hom(V_-^d,V')
\,.
$$
We shall prove Claim \eqref{20180208:eq2}
separately for $a\in \mf g'[\geq1]$
and for $a\in \Hom(V_-^d,V')$.
For $a \in \mf{g'}[\geq 1]$, by the induction hypothesis we have
\begin{equation}\label{20180219:eq12}
[a, \widetilde{W}(z) \id_{V_-^u}]
= 
[a, \widetilde{W}'(z) \id_{V_-^u}] \in \mc I[z]
\,,
\end{equation}
since $\mc I' =U(\mf g')\Span\{b - (f' | b)\}_{b \in \mf{g'}[\geq 1]} \subset \mc I$. 
The last inclusion holds because, 
as it is easily checked, $(b|f')=(b|f)$ for every $b\in\mf{g'}[\geq 1]$.
Next, using equation \eqref{eq:recW} we have
\begin{equation}\label{20180217:eq3}
\begin{array}{l}
\vphantom{\Big(}
[a, \widetilde W(z)\id_{V_-^d}]
=
-
\frac1{r_1}[a, [\widetilde W'(z),\id_{F^tV_-^d}E_{-1}]^1]
-[a,\widetilde W'(z)]F^t(z\id_V+E_0+D)\id_{V_-^d}
\\
\vphantom{\Big(}
%\qquad\qquad
\quad - \widetilde W'(z)F^t[a,E_0]\id_{V_-^d}
+\Res_x x^{-1}[a,\widetilde W'(z)]
\id_{V_-^u}
(1+x^{-1}F)^{-1}\widetilde W'(x)F^t\id_{V_-^d}
\\
\vphantom{\Big(}
%\qquad\qquad
\quad
+\Res_x x^{-1}\widetilde W'(z)
\id_{V_-^u}
(1+x^{-1}F)^{-1}[a,\widetilde W'(x)]F^t\id_{V_-^d}
\,.
\end{array}
\end{equation}
By the induction hypothesis, $ [a, \widetilde{W}'(z)] \equiv 0\mod \mc I'$. 
In particular, the last term in the RHS of \eqref{20180217:eq3} vanishes modulo $\mc I$.
The second last term in the RHS of \eqref{20180217:eq3} vanishes as well (modulo $\mc I$, 
in fact modulo $\mc I'$),
since by inductive assumption $\widetilde W'(x)$ has coefficients 
in $\widetilde W' : = \{y \in U(\mf g') \mid [a, y] \in \mc I', \forall a \in \mf g'[\geq 1] \}$,
and $\mc I' \subset \widetilde W' $ is a bilateral ideal.
On the other hand, by \eqref{20180219:eq5},
we have 
\begin{equation}\label{20180219:eq14}
[a,E_0]\id_{V_-^d}
\equiv
\id_{V_-^d}
[F,A]
\id_{V_-^d}
=0
\mod\mc I
\,,
\end{equation}
since $F\id_{V_-^d}=A\id_{V_-^d}=0$
(recall that $A\in\End V'$).
We are left to consider the first two terms in the RHS of \eqref{20180217:eq3}.
For the second term, 
by the induction hypothesis $[a,\widetilde W'(z)]\equiv 0\mod\mc I'$. 
In other words, it is a linear combination of the form
\begin{equation}\label{20180219:eq15}
[a,\widetilde W'(z)]
=
\sum u(b-(f|b))C
\,,
\end{equation}
with $u\in U(\mf g')$ and $b\in\mf g'[\geq1]$ and $C\in\Hom(V'_-,V'_+)$.
Hence, the second term of the RHS of \eqref{20180217:eq3}
has the form
\begin{equation}\label{20180219:eq13}
\begin{array}{l}
\displaystyle{
\vphantom{\Big(}
-\sum u(b-(b|f))CF^t(z\id_V+E_0+D)\id_{V_-^d}
} \\
\displaystyle{
\vphantom{\Big(}
=
-\sum uCF^t(z\id_V+E_0+D)\id_{V_-^d}(b-(b|f))
-\sum uCF^t[b,E_0]\id_{V_-^d}
\,.}
\end{array}
\end{equation}
Both terms in the RHS of \eqref{20180219:eq13}
vanish modulo $\mc I$,
the first by definition and the second by \eqref{20180219:eq14}.
Next, let us consider 
the first term in the RHS of equation \eqref{20180217:eq3}.
Using the notation introduced in \eqref{20180216:eq3}
(cf. Lemma \ref{20180426:lem}), 
we have
$$
\begin{array}{l}
\displaystyle{
\vphantom{\Big(}
[a, [\widetilde W'(z),\id_{F^tV_-^d}E_{-1}]^1]
=
\sum_{i \in J}[a, [\widetilde W'(z), q_i]]\id_{F^tV_-^d}Q^i
}\\
\displaystyle{
\vphantom{\Big(}
=
\sum_{i \in J} \big( 
[[a, \widetilde W'(z)],q_i]
+
[\widetilde W'(z),[a, q_i]]
\big)
\id_{F^tV_-^d}Q^i
\,.}
\end{array}
$$
By construction, 
$q_i \in \mf h:=\Hom(F^tV_-^d, V_-^d)$ 
and $a\in\mf g'[\geq 1] \subset \Hom(V', V[\geq-\frac{p_1-1}2+2])$.
Since $F^tV_-^d\cap V[\geq-\frac{p_1-1}2+2]=0$ and $V_-^d\cap V'=0$,
$[a,q_i]=0$ for all $i\in J$.
On the other hand, 
by \eqref{20180219:eq15} we have
$$
[[a, \widetilde W'(z)],q_i]=\sum \big([u,q_i](b-(f|b))+u[b,q_i]\big)C
\equiv0\mod\mc I\,.
$$
Indeed, for the same argument as above (with $b$ in place of $a$), we have $[b,q_i]=0$.
Hence, \eqref{20180208:eq2} holds for $a\in\mf g'[\geq1]$.

We are left to prove Claim \eqref{20180208:eq2} for $a\in\Hom(V_-^d,V')$.
Note that
$$
\Hom(V_-^d,V')
=
\Hom(V_-^d,{V'}_-^d)
+
[\Hom(V_-^d,{V'}_-^d),\mf g'[\geq1]]
\,.
$$
Hence, by the Jacobi identity,
it suffices to prove Claim \eqref{20180208:eq2} for $a\in\Hom(V_-^d,{V'}_-^d)$.
We have the following decomposition
\begin{equation}\label{20180219:eq16}
V_-
=
{V'}_-^u
\oplus
({V'}_-^d\cap V_-^u)
\oplus
V_-^d
\,.
\end{equation}
We shall consider separately the restriction of equation \eqref{20180208:eq2}
to each of the subspaces \eqref{20180219:eq16}.
When we restrict to ${V'}_-^u\subset V_-^u$, we have,
by double recursion,
$$
[a, \widetilde W(z)\id_{{V'}_-^u}] = 
[a, \widetilde W'(z)\id_{{V'}_-^u}] =
[a, \widetilde W''(z)\id_{{V'}_-^u}] = 0
\,,
$$
since 
\begin{equation}\label{20180219:eq17}
[\Hom(V_-^d,{V'}_-^d),\mf g'']=0
\,.
\end{equation}
Next, when we restrict \eqref{20180208:eq2}
to ${V'}_-^d\cap V_-^u$, we have,
by double recursion,
$$
\begin{array}{l}
%\displaystyle{
\vphantom{\Big(}
\widetilde W(z)\id_{{V'}_-^d \cap V_-^u} = 
\widetilde W'(z)\id_{{V'}_-^d \cap V_-^u} 
\\
\vphantom{\Big(}
=
-
\frac1{r_1 + r'_1}[\widetilde W''(z),\id_{{F'}^t{V'}_-^d}E_{-1}]^1\id_{{V'}_-^d \cap V_-^u}
\\
\vphantom{\Big(}
-
\widetilde W''(z)F^t(z\id_{V'}+E_0+D)\id_{{V'}_-^d\cap V_-^u}
\\
\vphantom{\Big(}
+ \Res_x x^{-1}\widetilde W''(z)
\id_{{V'}_-^u}
(1+x^{-1}F)^{-1}\widetilde W''(x)F^t\id_{{V'}_-^d\cap V_-^u}
\,,
\end{array}
$$
where, as before, $r'_1$ denotes the height of the leftmost column of the pyramid $p'$
(which is equal to $r_1$ or $r_2$ depending whether $p_1>p_2+1$ or $p_1=p_2+1$
respectively).
Here we used Lemma \ref{lem:restr}.
By \eqref{20180219:eq17}, $[a,\widetilde W''(z)]=0$.
Hence,
\begin{equation}\label{20180218:eq1}
\begin{array}{l}
\displaystyle{
\vphantom{\Big(}
[a,\widetilde W(z)\id_{{V'}_-^d \cap V_-^u} ]
} \\
\displaystyle{
\vphantom{\Big(}
=
-
\frac1{r_1 + r'_1}[\widetilde W''(z),\id_{F^t{V'}_-^d}[a,E_{-1}]]^1\id_{{V'}_-^d \cap V_-^u} 
-\widetilde W''(z)F^t[a,E_0]\id_{{V'}_-^d \cap V_-^u}
\,.}
\end{array}
\end{equation}
By equation \eqref{20180219:eq4}, we have
$$
\id_{F^t{V'}_-^d}[a,E_{-1}]\id_{{V'}_-^d \cap V_-^u}
=
\id_{F^t{V'}_-^d}[E_0,A]\id_{{V'}_-^d \cap V_-^u} 
=
0\,,
$$
since $A\id_{V'}=0$ and $E_0$ preserves $V'$.
Hence,
the first term in the RHS of \eqref{20180218:eq1} vanishes.
Moreover, by equation \eqref{20180219:eq5}, 
the second term in the RHS of \eqref{20180218:eq1} is
$$
-\widetilde W''(z)F^t[a,E_0]\id_{{V'}_-^d \cap V_-^u}
\equiv
-\widetilde W''(z)F^t[F,A]\id_{{V'}_-^d \cap V_-^u}
=
0
\,,
$$
since $A({V'}_-^d)=0$ and $F(V_-^u)=0$.

We are now left consider the most difficult case:
the restriction of \eqref{20180208:eq2} to $V_-^d$.
By double recursion \eqref{eq:recW} we have
\begin{equation}\label{20180221:eq1}
\begin{array}{l}
\displaystyle{
\vphantom{\Big(}
\widetilde W(z)\id_{V_-^d}
}\\
\displaystyle{
\vphantom{\Big(}
=
-\frac1{r_1}
\Big[
-\frac1{r_1 + r'_1}[\widetilde W''(z),\id_{{F'}^t{V'}_-^d}E'_{-1}]^1
-\widetilde W''(z){F'}^t(z\id_{V'}+\textcolor{red}{E'_0}+D')\id_{{V'}_-^d}
}\\
\displaystyle{
\vphantom{\Big(}
\quad
+\Res_y y^{-1}\widetilde W''(z)
\id_{{V'}_-^u}
(1+y^{-1}{F'})^{-1}\widetilde W''(y){F'}^t\id_{{V'}_-^d}
\,,
\id_{F^tV_-^d}\textcolor{red}{E_{-1}}
\Big]^1
}\\
\displaystyle{
\vphantom{\Big(}
%second
-
\bigg(
-\frac1{r_1 + r'_1}[\widetilde W''(z),\id_{{F'}^t{V'}_-^d}E'_{-1}]^1
-\widetilde W''(z){F'}^t(z\id_{V'}+\textcolor{red}{E'_0}+D')\id_{{V'}_-^d}
}\\
\displaystyle{
\vphantom{\Big(}
\quad
+\Res_y y^{-1}\widetilde W''(z)
\id_{{V'}_-^u}
(1+y^{-1}{F'})^{-1}\widetilde W''(y){F'}^t\id_{{V'}_-^d}
\bigg)
F^t(z\id_V+\textcolor{red}{E_0}+D)\id_{V_-^d}
}\\
\displaystyle{
\vphantom{\Big(}
%fourth
+\Res_x x^{-1}\widetilde W''(z)
\id_{{V'}_-^u}
(1+x^{-1}F)^{-1}
\bigg(
-\frac1{r_1 + r'_1}[\widetilde W''(x),\id_{{F'}^t{V'}_-^d}E'_{-1}]^1
}\\
\displaystyle{
\vphantom{\Big(}
\quad
- \widetilde W''(x){F'}^t(x\id_{V'}+\textcolor{red}{E'_0}+D')\id_{{V'}_-^d}
}\\
\displaystyle{
\vphantom{\Big(}
\quad
+\Res_y y^{-1}\widetilde W''(y)
\id_{{V'}_-^u}
(1+y^{-1}{F'})^{-1}\widetilde W''(y){F'}^t\id_{{V'}_-^d}
\bigg)
F^t\id_{V_-^d}
}\\
\displaystyle{
\vphantom{\Big(}
%fifth
+\Res_x x^{-1}
\bigg(
\!-\frac1{r_1 + r'_1}[\widetilde W''(z),\id_{{F'}^t{V'}_-^d}E'_{-1}]^1
-\widetilde W''(z){F'}^t(z\id_{V'}+\textcolor{red}{E'_0}+D')\id_{{V'}_-^d}
}\\
\displaystyle{
\vphantom{\Big(}
\quad
+\Res_y y^{-1}\widetilde W''(z)
\id_{{V'}_-^u}
(1+y^{-1}{F'})^{-1}\widetilde W''(y){F'}^t\id_{{V'}_-^d}
\bigg)
\id_{V_-^u \cap {V'}_-^d}
(1+x^{-1}F)^{-1}
}\\
\displaystyle{
\vphantom{\Big(}
\quad
\times
\bigg(
-\frac1{r_1 + r'_1}[\widetilde W''(x),\id_{{F'}^t{V'}_-^d}E'_{-1}]^1
-\widetilde W''(x){F'}^t(x\id_{V'}+\textcolor{red}{E'_0}+D')\id_{{V'}_-^d}
}\\
\displaystyle{
\vphantom{\Big(}
\quad
+\Res_{\bar{y}} \bar{y}^{-1}\widetilde W''(\bar{y})
\id_{{V'}_-^u}
(1+\bar{y}^{-1}{F'})^{-1}\widetilde W''(\bar{y}){F'}^t\id_{{V'}_-^d}
\bigg)
F^t\id_{V_-^d}
\,.}
\end{array}
\end{equation}
For the last two terms we used 
the decomposition $ V_-^u = {V'}_-^u \oplus (V_-^u \cap {V'}_-^d) $.
By assumption, $a\in\Hom(V_-^d,{V'}_-^d)\subset\mf g$
(and $A\in\Hom(V_-^d,{V'}_-^d)\subset\End V$).
Hence, 
$[a,\widetilde W''(z)]=0$, and, by Lemma \ref{20180216:lem2},
$$
[a, \id_{{F'}^t{V'}_-^d}E'_{-1}]
=  \id_{{F'}^t{V'}_-^d} [a,E_{-1}]\id_{{V'}_-^d}
=  \id_{{F'}^t{V'}_-^d} [E_0,A]\id_{{V'}_-^d}
= 0
\,.
$$
Hence, if we use \eqref{20180221:eq1}
to compute $[a, \widetilde W(z)\id_{V_-^d}]$,
we only need to consider the terms in which $\ad a$ acts 
on $E'_0$, $E_0$ and $\id_{F^tV_-^d}E_{-1}$.
Which leaves us precisely with seven terms
(corresponding to the seven highlighted terms in the RHS of \eqref{20180221:eq1}).
By Lemma \ref{20180216:lem2}, we have
\begin{equation}\label{20180223:eq3}
\begin{array}{l}
\displaystyle{
\vphantom{\Big(}
[a,E_0]
=
[E_1,A]
\,\equiv 
[F,A]\mod\mc I
\,,} \\
\displaystyle{
\vphantom{\Big(}
[a,E'_0]
=
\id_{V'}[E_1,A]\id_{V'}
=
-AE_1
\,\equiv
-AF\mod\mc I
\,,} \\
\displaystyle{
\vphantom{\Big(}
[a,E_{-1}]
=
[E_0,A]
\,.}
\end{array}
\end{equation}
Note that 
$\id_{{V'}_-^d}E_{-1}$,
$\id_{V_-^d}E_0$,
$\id_{{V'}_-^d}E_0$,
$\id_{{V'}_-^d}AE_0$,
are all matrices with coefficients in 
$\End(V_-^d\oplus {V'}_-^d)\subset\mf g$,
while $\widetilde W''(z)$ is a matrix with coefficients in $\mf g''\subset\mf g$.
Since, obviously, $[\mf g'',\End(V_-^d\oplus {V'}_-^d)]=0$,
we have
\begin{equation}\label{20180301:eq1}
\big[\widetilde W''(z), \id_{{V'}_-^d}E_{-1}\big]^1 
=
\big[\widetilde W''(z), \id_{V_-^d\oplus {V'}_-^d}E_{0}\big]^1 
=
\big[\widetilde W''(z), \id_{{V'}_-^d}AE_{0}\big]^1 
= 0
\,.
\end{equation}
Moreover,
$E_{0}\id_{{V'}_-^d}$,
$E_{-1}\id_{{V'}_-^d}$,
and $\id_{{V'}_-^d}AE_0$
are matrices with coefficients in 
$\End({V'}_-^d)$,
$\Hom({V''}_-^d,{V'}_-^d)$
and $\End(V_-^d)\subset\mf g$ respectively.
Since, obviously, 
$[\End({V'}_-^d\oplus{V''}_-^d),\End(V_-^d)]=0$,
we have
\begin{equation}\label{20180301:eq1b}
\big[
E_{0},
\id_{{V'}_-^d}AE_0
\big]^1 
=
\big[
E_{-1},
\id_{{V'}_-^d}AE_0
\big]^1 
= 0
\,.
\end{equation}
Let us now compute the first two contributions to $[a,\widetilde W(z)]$
(coming from the first two highlighted terms in the RHS of \eqref{20180221:eq1}).
They are
\begin{equation}\label{20180301:eq2}
\begin{array}{l}
\displaystyle{
\vphantom{\Big(}
\Big[a,
-\frac1{r_1}
\Big[
-\frac1{r_1 + r'_1}[\widetilde W''(z),\id_{{F'}^t{V'}_-^d}E'_{-1}]^1
-\widetilde W''(z){F'}^t(z\id_{V'}+\textcolor{red}{E'_0}+D')\id_{{V'}_-^d}
}\\
\displaystyle{
\vphantom{\Big(}
\quad
+\Res_y y^{-1}\widetilde W''(z)
\id_{{V'}_-^u}
(1+y^{-1}{F'})^{-1}\widetilde W''(y){F'}^t\id_{{V'}_-^d}
\,,
\id_{F^tV_-^d}\textcolor{red}{E_{-1}}
\Big]^1
\Big]
}\\
\displaystyle{
\vphantom{\Big(}
=
\frac1{r_1(r_1+r'_1)}
\Big[
[\widetilde W''(z),\id_{{F'}^t{V'}_-^d}E'_{-1}]^1
\,,
\id_{F^tV_-^d}[E_0,A]
\Big]^1
}\\
\displaystyle{
\vphantom{\Big(}
-
\frac1{r_1}
\Big[
\widetilde W''(z){F'}^tAE_1\id_{{V'}_-^d}
\,,
\id_{F^tV_-^d}E_{-1}
\Big]^1
}\\
\displaystyle{
\vphantom{\Big(}
+
\frac1{r_1}
\Big[
\widetilde W''(z){F'}^t(z\id_{V'}+E'_0+D')\id_{{V'}_-^d}
\,,
\id_{F^tV_-^d}[E_0,A]
\Big]^1
}\\
\displaystyle{
\vphantom{\Big(}
=
\frac1{r_1(r_1+r'_1)}
\Big[
\widetilde W''(z),
[\id_{{F'}^t{V'}_-^d}E'_{-1}
\,,
\id_{F^tV_-^d}[E_0,A]]^1
\Big]^1
}\\
\displaystyle{
\vphantom{\Big(}
-
\frac1{r_1}
\widetilde W''(z)
{F'}^tA
\Big[
E_1
\,,
\id_{F^tV_-^d}E_{-1}
\Big]^1
+
\frac1{r_1}
\widetilde W''(z)
{F'}^t
\Big[
E_0
\,,
\id_{F^tV_-^d}[E_0,A]
\Big]^1
}\\
\displaystyle{
\vphantom{\Big(}
=
\frac1{r_1(r_1+r'_1)}
\Big[
\widetilde W''(z),
\id_{{F'}^t{V'}_-^d}[E_{-1}
\,,
\id_{F^tV_-^d}E_0]^1
\Big]^1A\id_{V_-^d}
}\\
\displaystyle{
\vphantom{\Big(}
-
\frac1{r_1}
\widetilde W''(z)
{F'}^tA
\Big[
E_1
\,,
\id_{F^tV_-^d}E_{-1}
\Big]^1
+
\frac1{r_1}
\widetilde W''(z)
{F'}^t
\Big[
E_0
\,,
\id_{F^tV_-^d}E_0
\Big]^1A\id_{V_-^d}
}\\
\displaystyle{
\vphantom{\Big(}
=
-
\frac1{r_1+r'_1}
\Big[
\widetilde W''(z),
\id_{{F'}^t{V'}_-^d}E_{-1}
\Big]^1A\id_{V_-^d}
+
\widetilde W''(z)
{F'}^t[A,E_0]\id_{V_-^d}
\,.}
\end{array}
\end{equation}
For the first equality of \eqref{20180301:eq2} we used \eqref{20180223:eq3} and 
\eqref{20180301:eq1},
for the second equality we used \eqref{20180301:eq1} again, 
for the third equality we have used  \eqref{20180301:eq1b},
and for the last equality we used Lemma \ref{lem:E}(b).
Next, let us consider the third and fourth contributions 
to $[a,\widetilde W(z)\id_{V_-^d}]$
(coming from the third and fourth highlighted term in \eqref{20180221:eq1}).
We have
\begin{equation}\label{20180221:eq3b}
\begin{array}{l}
\displaystyle{
\vphantom{\Big(}
-\Bigg[a,
\bigg(
-\frac1{r_1 + r'_1}[\widetilde W''(z),\id_{{F'}^t{V'}_-^d}E'_{-1}]^1
-\widetilde W''(z){F'}^t(z\id_{V'}+\textcolor{red}{E'_0}+D')\id_{{V'}_-^d}
}\\
\displaystyle{
\vphantom{\Big(}
+\Res_y y^{-1}\widetilde W''(z)
\id_{{V'}_-^u}
(1+y^{-1}{F'})^{-1}\widetilde W''(y){F'}^t\id_{{V'}_-^d}
\bigg)
F^t(z\id_V+\textcolor{red}{E_0}+D)\id_{V_-^d}
\Bigg]
}\\
\displaystyle{
\vphantom{\Big(}
\equiv
\widetilde W''(z){F'}^t[E_1,A]\id_{F^tV_-^d}
F^t(z\id_V+E_0+D)\id_{V_-^d}
}\\
\displaystyle{
\vphantom{\Big(}
-
\bigg(
-\frac1{r_1 + r'_1}[\widetilde W''(z),\id_{{F'}^t{V'}_-^d}E'_{-1}]^1
-\widetilde W''(z){F'}^t(z\id_{V'}+E'_0+D')\id_{{V'}_-^d}
}\\
\displaystyle{
\vphantom{\Big(}
+\Res_y y^{-1}\widetilde W''(z)
\id_{{V'}_-^u}
(1+y^{-1}{F'})^{-1}\widetilde W''(y){F'}^t\id_{{V'}_-^d}
\bigg)
F^t[F,A]\id_{V_-^d}
}\\
\displaystyle{
\vphantom{\Big(}
\equiv
\widetilde W''\!(z){F'}^t[F,A]\id_{F^tV_-^d}
F^t(z\id_V\!\!+\!\!E_0\!\!+\!\!D)\id_{V_-^d}\!
+
\!\widetilde W''\!(z){F'}^t[[E_1,A],\!\id_{F^tV_-^d}
F^tE_0]^1\id_{V_-^d}
}\\
\displaystyle{
\vphantom{\Big(}
-
\bigg(
-\frac1{r_1 + r'_1}[\widetilde W''(z),\id_{{F'}^t{V'}_-^d}E'_{-1}]^1
-\widetilde W''(z){F'}^t(z\id_{V'}+E'_0+D')\id_{{V'}_-^d}
}\\
\displaystyle{
\vphantom{\Big(}
+\Res_y y^{-1}\widetilde W''(z)
\id_{{V'}_-^u}
(1+y^{-1}{F'})^{-1}\widetilde W''(y){F'}^t\id_{{V'}_-^d}
\bigg)
\id_{F^tV_-^d}A\id_{V_-^d}
}\\
\displaystyle{
\vphantom{\Big(}
\equiv
-
\widetilde W''(z){F'}^tA((z-N+r_1)\id_V+E_0)\id_{V_-^d}
-
r_1 \widetilde W''(z){F'}^tA\id_{V_-^d}
}\\
\displaystyle{
\vphantom{\Big(}
+
\frac1{r_1 + r'_1}
[\widetilde W''(z),\id_{{F'}^t{V'}_-^d}E'_{-1}]^1
\id_{F^tV_-^d}A\id_{V_-^d}
}\\
\displaystyle{
\vphantom{\Big(}
+
\widetilde W''(z){F'}^t((z-N+2r_1+r'_1)\id_{V'}+E'_0)
\id_{F^tV_-^d}A\id_{V_-^d}
}\\
\displaystyle{
\vphantom{\Big(}
-
\Res_y y^{-1}\widetilde W''(z)
\id_{{V'}_-^u}
(1+y^{-1}{F'})^{-1}\widetilde W''(y){F'}^t
\id_{F^tV_-^d}A\id_{V_-^d}
\,.}
\end{array}
\end{equation}
For the first equality we used \eqref{20180223:eq3},
for the second equality we used \eqref{20180223:eq3}
and the identities $F\id_{V_-^d}=0$ and $F^tF=\id_{F^tV_-^d}$,
while for the third equality
we used the identities $A\id_{F^tV_-^d}=0$,
$F\id_{F^tV_-^d}F^t=\id_{V_-^d}$,
$D\id_{V_-^d}=-(N-r_1)\id_{V_-^d}$,
and $D'\id_{F^tV_-^d}=-(N-2r_1-r'_1)\id_{F^tV_-^d}$
(cf. \eqref{20180219:eq2}),
and the fact that, by Lemma \ref{lem:E}(c), we have
$$
[E_1,\id_{F^tV_-^d}F^tE_0]^1
=
r_1\id_{V_-^d}
\,.
$$
Next, the fifth contribution to 
$[a,\widetilde W(z)\id_{V_-^d}]$
coming from \eqref{20180221:eq1})
is, by \eqref{20180223:eq3}
\begin{equation}\label{20180301:eq3}
\begin{array}{l}
\displaystyle{
\vphantom{\Big(}
\Bigg[a,
\Res_x x^{-1}\widetilde W''(z)
\id_{{V'}_-^u}
(1+x^{-1}F)^{-1}
\bigg(
-\frac1{r_1 + r'_1}[\widetilde W''(x),\id_{{F'}^t{V'}_-^d}E'_{-1}]^1
}\\
\displaystyle{
\vphantom{\Big(}
- \widetilde W''(x){F'}^t(x\id_{V'}+\textcolor{red}{E'_0}+D')\id_{{V'}_-^d}
}\\
\displaystyle{
\vphantom{\Big(}
+\Res_y y^{-1}\widetilde W''(y)
\id_{{V'}_-^u}
(1+y^{-1}{F'})^{-1}\widetilde W''(y){F'}^t\id_{{V'}_-^d}
\bigg)
F^t\id_{V_-^d}
\Bigg]
}\\
\displaystyle{
\vphantom{\Big(}
\equiv
\Res_x x^{-1}\widetilde W''(z)
\id_{{V'}_-^u}
(1+x^{-1}F)^{-1}
\widetilde W''(x){F'}^tAF\id_{{V'}_-^d}
F^t\id_{V_-^d}
}\\
\displaystyle{
\vphantom{\Big(}
=
\Res_x x^{-1}\widetilde W''(z)
\id_{{V'}_-^u}
(1+x^{-1}F)^{-1}
\widetilde W''(x){F'}^tA\id_{V_-^d}
\,.}
\end{array}
\end{equation}
Next, let us consider the contribution to 
$[a,\widetilde W(z)\id_{V_-^d}]$
coming from the sixth highlighted term in \eqref{20180221:eq1}).
It gives
\begin{equation}\label{20180301:eq4}
\begin{array}{l}
\displaystyle{
\vphantom{\Big(}
-
\Res_x x^{-1}
\widetilde W''(z){F'}^t[a,\textcolor{red}{E'_0}]\id_{{V'}_-^d}
\id_{V_-^u \cap {V'}_-^d}
(1+x^{-1}F)^{-1}
}\\
\displaystyle{
\vphantom{\Big(}
\times
\bigg(
-\frac1{r_1 + r'_1}[\widetilde W''(x),\id_{{F'}^t{V'}_-^d}E'_{-1}]^1
-\widetilde W''(x){F'}^t(x\id_{V'}+\textcolor{red}{E'_0}+D')\id_{{V'}_-^d}
}\\
\displaystyle{
\vphantom{\Big(}
+\Res_{\bar{y}} \bar{y}^{-1}\widetilde W''(\bar{y})
\id_{{V'}_-^u}
(1+\bar{y}^{-1}{F'})^{-1}\widetilde W''(\bar{y}){F'}^t\id_{{V'}_-^d}
\bigg)
F^t\id_{V_-^d}
}\\
\displaystyle{
\vphantom{\Big(}
%%%
=
\Res_x x^{-1}
\widetilde W''(z){F'}^tAE_1
\id_{V_-^u \cap {V'}_-^d}
(1+x^{-1}F)^{-1}
}\\
\displaystyle{
\vphantom{\Big(}
\times
\bigg(
\!\!\!\!
-\frac1{r_1 \!+\! r'_1}[\widetilde W''(x),\id_{{F'}^t{V'}_-^d}E'_{-1}]^1
-\widetilde W''(x){F'}^t(x\id_{V'}+E'_0+D')\id_{{V'}_-^d}
}\\
\displaystyle{
\vphantom{\Big(}
+\Res_{\bar{y}} \bar{y}^{-1}\widetilde W''(\bar{y})
\id_{{V'}_-^u}
(1+\bar{y}^{-1}{F'})^{-1}\widetilde W''(\bar{y}){F'}^t\id_{{V'}_-^d}
\bigg)
F^t\id_{V_-^d}
}\\
\displaystyle{
\vphantom{\Big(}
%%%
\equiv
\Res_x x^{-1}
\widetilde W''(z){F'}^tA
\Bigg[E_1,
\id_{V_-^u \cap {V'}_-^d}
(1+x^{-1}F)^{-1}
}\\
\displaystyle{
\vphantom{\Big(}
\times
\bigg(
\!\!\!\!
-\frac1{r_1 \!+\! r'_1}[\widetilde W''(x),\id_{{F'}^t{V'}_-^d}E'_{-1}]^1
-\widetilde W''(x){F'}^t(x\id_{V'}+E'_0+D')\id_{{V'}_-^d}
}\\
\displaystyle{
\vphantom{\Big(}
+\Res_{\bar{y}} \bar{y}^{-1}\widetilde W''(\bar{y})
\id_{{V'}_-^u}
(1+\bar{y}^{-1}{F'})^{-1}\widetilde W''(\bar{y}){F'}^t\id_{{V'}_-^d}
\bigg)
\Bigg]^1
F^t\id_{V_-^d}
\,.}
\end{array}
\end{equation}
For the first equality we used \eqref{20180223:eq3},
for the second equality we used the fact that $E_1\equiv F\mod\mc I$
and the identity $F\id_{V_-^u}=0$.
Note that $E_1\id{{V'}_-^d}$ is a matrix with coefficients in $\Hom(V_-^d,{V'}_-^d)\subset\mf g$,
hence it commutes, with respect to the bracket $[\cdot\,,\,\cdot]^1$
with $\widetilde W''(z)$ and with $\id_{{F'}^t{V'}_-^d}E'_{-1}$
(which has coefficients in $\Hom({F'}^t{V'}_-^d,{V'}_-^d)$).
While it has a non trivial bracket with $E_0'$
(which we replace by $E_0$ thanks to Lemma \ref{lem:restr}).
Hence, the RHS of \eqref{20180301:eq4}
can be rewritten as
\begin{equation}\label{20180301:eq5}
\begin{array}{l}
\displaystyle{
\vphantom{\Big(}
-
\Res_x x^{-1}
\widetilde W''(z){F'}^tA
\Big[E_1,
\id_{V_-^u \cap {V'}_-^d}
(1+x^{-1}F)^{-1}
\widetilde W''(x){F'}^t
E_0
\Big]^1
\id_{{V'}_-^d}
F^t\id_{V_-^d}
\,.}
\end{array}
\end{equation}
Let
\begin{equation}\label{20180301:eq6}
\begin{array}{l}
\displaystyle{
\vphantom{\Big(}
X
:=
\Res_x x^{-1}
\id_{V_-^u \cap {V'}_-^d}
(1+x^{-1}F)^{-1}
\widetilde W''(x){F'}^t
\id_{{V'}_-^d}
} \\
\displaystyle{
\vphantom{\Big(}
=
(-1)^{p_1}
\Res_x x^{-p_1+1}
\id_{V_-^u \cap {V'}_-^d}
F^{p_1-2}
\widetilde W''(x){F'}^t
\id_{{V'}_-^d}
\,.}
\end{array}
\end{equation}
For the second equality in \eqref{20180301:eq6} we used the fact that 
$\widetilde W''(z)=\id_{V_+}\widetilde W''(z)\in U(\mf g)\otimes \Hom(V''_-,V_+)$
and 
$$
\id_{{V'}_-^d}(1+x^{-1}F)^{-1}\id_{V_+}=(-1)^{p_1}x^{-p_1+2}\id_{{V'}_-^d}F^{p_1-2}\id_{V_+}
\,.
$$
By Lemma \ref{20180222:lem1},
$$
\widetilde W''(x)
=
x\id_{V_+}(1+xF^t)^{-1}\id_{V''_-}
+\sum_{i\in\mc F''}\widetilde {w}''_i(x) U^i
\,,
$$
and $\widetilde {w}''_i(x)$ is a polynomial in $x$ of degree at most $p_1-3$,
hence it does not contribute to the residue in \eqref{20180301:eq6}.
Hence, 
\begin{equation}\label{20180301:eq7}
\begin{array}{l}
\displaystyle{
\vphantom{\Big(}
X
=
(-1)^{p_1}
\Res_x x^{-p_1+2}
\id_{V_-^u \cap {V'}_-^d}
F^{p_1-2}
\id_{V_+}(1+xF^t)^{-1}\id_{V''_-}
{F'}^t
\id_{{V'}_-^d}
} \\
\displaystyle{
\vphantom{\Big(}
=
-
\id_{V_-^u \cap {V'}_-^d}
F^{p_1-2}
\id_{V_+}
(F^t)^{p_1-3}
\id_{V''_-}
{F'}^t
\id_{{V'}_-^d}
=
-
\id_{V_-^u \cap {V'}_-^d}
\,.}
\end{array}
\end{equation}
In view of \eqref{20180301:eq6}-\eqref{20180301:eq7},
we can rewrite \eqref{20180301:eq5} as
\begin{equation}\label{20180301:eq8}
\begin{array}{l}
\displaystyle{
\vphantom{\Big(}
-
\widetilde W''(z){F'}^tA
\Big[E_1,XE_0\Big]^1
\id_{{V'}_-^d}
F^t\id_{V_-^d}
} \\
\displaystyle{
\vphantom{\Big(}
=
\widetilde W''(z){F'}^tA
\Big[E_1,
\id_{V_-^u \cap {V'}_-^d}
E_0\Big]^1
\id_{{V'}_-^d}
F^t\id_{V_-^d}
} \\
\displaystyle{
\vphantom{\Big(}
=
-r'_1
\widetilde W''(z){F'}^tA
E_1
F^t\id_{V_-^d}
\equiv
-r'_1
\widetilde W''(z){F'}^tA
\id_{V_-^d}
\,.}
\end{array}
\end{equation}
For the second equality we used Lemma \ref{lem:E}(b),
while for the last equality we used the fact that $E_1\equiv F\mod\mc I$.
Next, let us consider the contribution to 
$[a,\widetilde W(z)\id_{V_-^d}]$
coming from the seventh (and last) highlighted term in \eqref{20180221:eq1}).
It gives
\begin{equation}\label{20180301:eq9}
\begin{array}{l}
\displaystyle{
\vphantom{\Big(}
-
\Res_x x^{-1}
\bigg(
\!-\frac1{r_1 \!+\! r'_1}[\widetilde W''(z),\id_{{F'}^t{V'}_-^d}E'_{-1}]^1
-\widetilde W''(z){F'}^t(z\id_{V'}+\textcolor{red}{E'_0}+D')\id_{{V'}_-^d}
}\\
\displaystyle{
\vphantom{\Big(}
+\Res_y y^{-1}\widetilde W''(z)
\id_{{V'}_-^u}
(1+y^{-1}{F'})^{-1}\widetilde W''(y){F'}^t\id_{{V'}_-^d}
\bigg)
}\\
\displaystyle{
\vphantom{\Big(}
\times
\id_{V_-^u \cap {V'}_-^d}
(1+x^{-1}F)^{-1}
\widetilde W''(x){F'}^t[a,\textcolor{red}{E'_0}]\id_{{V'}_-^d}
F^t\id_{V_-^d}
}\\
\displaystyle{
\vphantom{\Big(}
%%%
\equiv
\Res_x x^{-1}
\bigg(
\!-\frac1{r_1 \!+\! r'_1}[\widetilde W''(z),\id_{{F'}^t{V'}_-^d}E'_{-1}]^1
-\widetilde W''(z){F'}^t(z\id_{V'}+E'_0+D')\id_{{V'}_-^d}
}\\
\displaystyle{
\vphantom{\Big(}
+\Res_y y^{-1}\widetilde W''(z)
\id_{{V'}_-^u}
(1+y^{-1}{F'})^{-1}\widetilde W''(y){F'}^t\id_{{V'}_-^d}
\bigg)
}\\
\displaystyle{
\vphantom{\Big(}
\times
\id_{V_-^u \cap {V'}_-^d}
(1+x^{-1}F)^{-1}
\widetilde W''(x){F'}^tA
F\id_{{V'}_-^d}
F^t\id_{V_-^d}
}\\
\displaystyle{
\vphantom{\Big(}
%%%
=
-\frac1{r_1 + r'_1}
\Res_x x^{-1}
[\widetilde W''(z),\id_{{F'}^t{V'}_-^d}E'_{-1}]^1
\id_{V_-^u \cap {V'}_-^d}
(1+x^{-1}F)^{-1}
}\\
\displaystyle{
\vphantom{\Big(}
\times
\widetilde W''(x){F'}^tA \id_{V_-^d}
-
\Res_x x^{-1}
\widetilde W''(z){F'}^t((z-N+2r_1+r'_1)\id_{V'}+E'_0)
}\\
\displaystyle{
\vphantom{\Big(}
\times
\id_{V_-^u \cap {V'}_-^d}
(1+x^{-1}F)^{-1}
\widetilde W''(x){F'}^tA \id_{V_-^d}
+
\Res_x \Res_y
x^{-1} y^{-1}
\widetilde W''(z)
}\\
\displaystyle{
\vphantom{\Big(}
\times
\id_{{V'}_-^u}
(1+y^{-1}{F'})^{-1}\widetilde W''(y){F'}^t
\id_{V_-^u \cap {V'}_-^d}
(1+x^{-1}F)^{-1}
\widetilde W''(x){F'}^tA \id_{V_-^d}
}\\
\displaystyle{
\vphantom{\Big(}
%%%
=
\frac1{r_1 + r'_1}
[\widetilde W''(z),\id_{{F'}^t{V'}_-^d}E'_{-1}]^1
\id_{V_-^u \cap {V'}_-^d}
A \id_{V_-^d}
}\\
\displaystyle{
\vphantom{\Big(}
+
\widetilde W''(z){F'}^t((z-N+2r_1+r'_1)\id_{V'}+E'_0)
\id_{V_-^u \cap {V'}_-^d}
A \id_{V_-^d}
}\\
\displaystyle{
\vphantom{\Big(}
-
\Res_y
y^{-1}
\widetilde W''(z)
\id_{{V'}_-^u}
(1+y^{-1}{F'})^{-1}\widetilde W''(y){F'}^t
\id_{V_-^u \cap {V'}_-^d}
A \id_{V_-^d}
\,.}
\end{array}
\end{equation}
For the first equality we used \eqref{20180223:eq3},
for the second equality we used the identities
$F\id_{{V'}_-^d}F^t\id_{V_-^d}=\id_{V_-^d}$
and $D'\id_{{V'}_-^d}=-(N-2r_1-r'_1)\id_{{V'}_-^d}$
(cf. \eqref{20180219:eq2}),
and for the last equality we used equations \eqref{20180301:eq6} and \eqref{20180301:eq7}.
To conclude, we combine
equations \eqref{20180221:eq1}, \eqref{20180301:eq2}, \eqref{20180221:eq3b},
\eqref{20180301:eq3}, 
%\eqref{20180301:eq4}, \eqref{20180301:eq5}, \eqref{20180301:eq6}, 
\eqref{20180301:eq8} and \eqref{20180301:eq9}, to get
\begin{equation}\label{20180302:eq1}
\begin{array}{l}
\displaystyle{
\vphantom{\Big(}
[a,\widetilde W(z)\id_{V_-^d}]
\equiv
-
\frac1{r_1+r'_1}
\Big[
\widetilde W''(z),
\id_{{F'}^t{V'}_-^d}E_{-1}\Big]^1
\big(1-\id_{F^tV_-^d}-\id_{V_-^u \cap {V'}_-^d}\big)
A\id_{V_-^d}
}\\
\displaystyle{
\vphantom{\Big(}
-
\widetilde W''(z){F'}^t
(E_0+z-N+2r_1+r'_1)(1-\id_{F^tV_-^d}-\id_{V_-^u \cap {V'}_-^d})
A\id_{V_-^d}
}\\
\displaystyle{
\vphantom{\Big(}
+
\Res_x x^{-1}\widetilde W''(z)
\id_{{V'}_-^u}
(1+x^{-1}F)^{-1}
\widetilde W''(x)
{F'}^t
(1\!-\!\id_{F^tV_-^d}\!-\!\id_{V_-^u \cap {V'}_-^d})
A\id_{V_-^d}
=0
\,}
\end{array}
\end{equation}
since, obviously, $(1-\id_{F^tV_-^d}-\id_{V_-^u \cap {V'}_-^d})A=0$.
\end{proof}

%%%
\subsection{Recursive formula for the operator $Z(z) $}\label{sec:5.5b}

Recall the matrix $Z(z)\in U(\mf g^f)[z]\otimes\Hom(V_-,V_+)$ from equation \eqref{eq:Z}.
\begin{proposition}\label{20180405:prop1}
The matrix $Z(z) $ satisfies the following recursive formula:
\begin{equation}\label{20180312:eq2}
\begin{array}{l}
\displaystyle{
\vphantom{\Big(}
Z(z) =  Z'(z)\id_{V_-^u} -  \frac1{r_1}[Z'(z), \id_{F^tV_-^d}E_{-1}]^1 
} \\
\displaystyle{
\vphantom{\Big(}
- 
zZ'(z)F^t\id_{V_-^d}
-
z\id_{V_+}(1 + zF^t)^{-1}F^tE_0\id_{V_-^d}
} \\
\displaystyle{
\vphantom{\Big(}
+
z\id_{V_+}(1 + zF^t)^{-1}\id_{V_-^u} \Res_x x^{-1} (1 + x^{-1}F)^{-1} Z'(x)F^t\id_{V_-^d}
\,.}
\end{array}
\end{equation}
where $Z'(z) \in U(\mf {g'}^f)[z]\otimes\Hom(V'_-,V_+)$ is the analogue 
of $Z(z) $ with respect to $ \mf g' $.
\end{proposition}
\begin{proof}
First, let us consider $Z(z)\id_{V_-^u}$. By \eqref{eq:Z} we have
\begin{equation}
\begin{array}{l}
\displaystyle{
\vphantom{\Big(}
Z(z)\id_{V_-^u}
=
z\id_{V_+}(1+zF^t)^{-1}\id_{V_-^u}
+
\sum_{i\in\mc F}\phi_z(u_i) U^i\id_{V_-^u}
} \\
\displaystyle{
\vphantom{\Big(}
=
z\id_{V_+}(1+zF^t)^{-1}\id_{V_-^u}
+
\sum_{i\in\mc F'}\phi'_z(u'_i) {U'}^i\id_{V_-^u}
=
Z'(z)\id_{V_-^u}
\,.}
\end{array}
\end{equation}
Here, $\{u'_i\}_{i \in \mc F} $ is a basis of $ \Hom(V_+, V'_-) \subset \mf g$, and $ \{{U'}^i\}_{i \in \mc F} $ is the dual basis of $ \Hom(V'_-, V_+) \subset \End V$. Moreover, $\phi'_z(u'_i) $ is defined as $\phi_z(u_i) $, with $ F' $ in place of $F $.
Hence, in the second equality we have used the fact that $ V_-^u \subset V_- \cap V'_- $, and that $ \phi'_z $ and $ \phi_z $ clearly coincide on a basis of $ \Hom(V_+, V_-^u)$.
When composed to $\id_{V_-^d}$ the first term of the RHS of \eqref{20180312:eq2}
vanishes, 
the second term gives
\begin{equation}\label{20180426:eq1}
-  \frac1{r_1}[Z'(z), \id_{F^tV_-^d}E_{-1}]^1 
=
- 
\frac1{r_1}\sum_{k=0}^{p_1-2} \sum_{i \in \mc F'} (-z)^k [\phi'_k(u'_i){U'}^i,  \id_{F^tV_-^d}E_{-1}]^1 
\,,
\end{equation}
the third term gives
\begin{equation}\label{20180426:eq2}
- 
zZ'(z)F^t\id_{V_-^d}
=
z\id_{V_+}(1 + zF^t)^{-1}\id_{V_-^d}
+
\sum_{k=0}^{p_1-2} \sum_{i \in \mc F'} (-z)^{k+1} \phi'_k(u'_i){U'}^i F^t\id_{V_-^d}
\,,
\end{equation}
the fourth term gives
\begin{equation}\label{20180426:eq3}
-
z\id_{V_+}(1 + zF^t)^{-1}F^tE_0\id_{V_-^d}
=
(-z)^{p_1-1} (F^t)^{p_1-1}E_0\id_{V_-^d}
\,,
\end{equation}
and the fifth term gives
\begin{equation}\label{20180426:eq4}
\begin{array}{l}
\displaystyle{
\vphantom{\Big(}
z\id_{V_+}(1 + zF^t)^{-1}\id_{V_-^u} \Res_x x^{-1} (1 + x^{-1}F)^{-1} Z'(x)F^t\id_{V_-^d}
} \\
\displaystyle{
\vphantom{\Big(}
=
\Res_x z\id_{V_+}(1 + zF^t)^{-1}\id_{V_-^u} (1 + x^{-1}F)^{-1} 
\id_{V_+}(1+xF^t)^{-1}F^t\id_{V_-^d}
} \\
\displaystyle{
\vphantom{\Big(}
+
\sum_{i\in\mc F'}\sum_{k=0}^{p_1-2}
\Res_x x^{-1} (-x)^k \phi'_k(u_i')
z\id_{V_+}(1 + zF^t)^{-1}\id_{V_-^u}  (1 + x^{-1}F)^{-1} {U'}^i F^t\id_{V_-^d}
\,.}
\end{array}
\end{equation}
The first term in the RHS of \eqref{20180426:eq4} vanishes 
since $\id_{V_-^u}F^i\id_{V_+}(F^t)^j\id_{V_-^d}$ is always zero for all $i,j\geq0$.
For the second term in the RHS of \eqref{20180426:eq4} we compute the residue in $x$, to get
\begin{equation}\label{20180426:eq4b}
\begin{array}{l}
\displaystyle{
\vphantom{\Big(}
\sum_{i\in\mc F'}\sum_{k=0}^{p_1-2}
\phi'_k(u_i')
z\id_{V_+}(1 + zF^t)^{-1}\id_{V_-^u}  F^k {U'}^i F^t\id_{V_-^d}
} \\
\displaystyle{
\vphantom{\Big(}
=
-\sum_{i\in\mc F'}\sum_{k=0}^{p_1-2}(-z)^{k+1}
\phi'_k(u_i')
\id_{V_+}(F^t)^k\id_{V_-^u}  F^k {U'}^i F^t\id_{V_-^d}
} \\
\displaystyle{
\vphantom{\Big(}
=
-\sum_{i\in\mc F'}\sum_{k=0}^{p_1-2}(-z)^{k+1}
\phi'_k(u_i')
\id_{V_+\cap(F^t)^kV_-} {U'}^i F^t\id_{V_-^d}
\,.}
\end{array}
\end{equation}
For the last equality we expanded $\id_{V_+}=\sum_{k=0}^{p_1-1}\id_{V_+\cap(F^t)^kV_-}$,
and we used the obvious identity
$$
\id_{V_+\cap(F^t)^kV_-}
(F^t)^h\id_{V_-}  F^h
=
\id_{V_+\cap(F^t)^kV_-}\delta_{h,k}
\,.
$$
Combining \eqref{20180426:eq1}-\eqref{20180426:eq4b},
we rewrite equation \eqref{20180312:eq2} (composed to $\id_{V_-^d}$) as
\begin{equation}\label{20180312:eq4}
\begin{array}{l}
\displaystyle{
\vphantom{\Big(}
\sum_{i\in\mc F}\sum_{k =0}^{p_1-1} (-z)^k \phi_k(u_i) U^i\id_{V_-^d}
=
- \frac1{r_1}\sum_{k=0}^{p_1-2} \sum_{i \in \mc F'} (-z)^k [\phi'_k(u'_i){U'}^i,  \id_{F^tV_-^d}E_{-1}]^1 
} \\
\displaystyle{
\vphantom{\Big(}
+ \sum_{k=0}^{p_1-2} \sum_{i \in \mc F'} (-z)^{k+1} \phi'_k(u'_i){U'}^i F^t\id_{V_-^d}
+ (-z)^{p_1-1} (F^t)^{p_1-1}E_0\id_{V_-^d}
} \\
\displaystyle{
\vphantom{\Big(}
-\sum_{i\in\mc F'}\sum_{k=0}^{p_1-2}(-z)^{k+1}
\phi'_k(u_i')
\id_{V_+\cap(F^t)^kV_-} {U'}^i F^t\id_{V_-^d}
\,.}
\end{array}
\end{equation}
We shall prove equation \eqref{20180312:eq4}
by equating the coefficients of each power of $z$ in both sides.
First, let us consider the constant terms in $z$. 
(Note that $ \phi_0 $ is the identity map on $ u_i $.)
Equation \eqref{20180312:eq4} reduces to
\begin{equation}\label{20180312:eq4a}
\begin{array}{l}
\displaystyle{
\vphantom{\Big(}
\sum_{i\in\mc F}u_i U^i\id_{V_-^d}
=
- \frac1{r_1}\sum_{i \in \mc F'} [u'_i{U'}^i,  \id_{F^tV_-^d}E_{-1}]^1 
\,,}
\end{array}
\end{equation}
which we need to prove.
Pairing the LHS of \eqref{20180312:eq4a} with an arbitrary element $ a \in \mf g$, 
we have 
\begin{equation}\label{20180426:eq5}
\sum_{i\in\mc F} (a | u_i) U^i\id_{V_-^d}
= 
\id_{V_+}A\id_{V_-^d}
 \,,
\end{equation}
where $ A $ is the element of $ \End V $corresponding to $ a \in \mf g $
(cf. Section \ref{sec:3.4}).
On the other hand, pairing the RHS of \eqref{20180312:eq4a} with the same element $a \in \mf g$, 
we have
\begin{equation}\label{20180426:eq6}
\begin{array}{l}
\displaystyle{
\vphantom{\Big(}
- \frac1{r_1}\sum_{i\in\mc F'} \sum_{j \in I} (a | [u'_i , u_j ]) {U'}^i\id_{F^tV_-^d}U^j\id_{V_-^d}
=
- \frac1{r_1}\sum_{i\in\mc F'} {U'}^i\id_{F^tV_-^d}[A, U'_i]\id_{V_-^d}
} \\
\displaystyle{
\vphantom{\Big(}
=
\frac1{r_1}\sum_{i\in I} \id_{V_+}{U'}^i\id_{F^tV_-^d} U'_iA \id_{V_-^d}
=
\id_{V_+}A \id_{V_-^d}
\,,}
\end{array}
\end{equation}
where in the first equality we used the completeness identity \eqref{eq:compl}, 
in the second equality we used the fact that $ U'_i \id_{V_-^d} = 0 $, 
and in the last equality we used Lemma \ref{20180216:lem1}.
Comparing \eqref{20180426:eq5} and \eqref{20180426:eq6},
we get equation \eqref{20180312:eq4a}.
Next, the coefficients of $(-z)^{k} $, for $ 1 \leq k \leq p_1-2$,
in \eqref{20180312:eq4}
give the equation
\begin{equation}\label{20180312:eq4b}
\begin{array}{l}
\displaystyle{
\vphantom{\Big(}
\sum_{i\in\mc F}\phi_k(u_i) U^i\id_{V_-^d}
=
- \frac1{r_1}\sum_{i \in \mc F'} [\phi'_k(u'_i){U'}^i,  \id_{F^tV_-^d}E_{-1}]^1 
} \\
\displaystyle{
\vphantom{\Big(}
+(1- 
\id_{V_+\cap(F^t)^{k-1}V_-}
) 
\sum_{i \in \mc F'} \phi'_{k-1}(u'_i) {U'}^i F^t \id_{V_-^d}
\,,}
\end{array}
\end{equation}
which we need to prove.
As before, we pair the LHS of \eqref{20180312:eq4b} with $a\in\mf g$
(and denote by $A\in\End V$ the corresponding endomorphism).
As a result, we get, recalling the definition \eqref{eq:phik} of $\phi_k$,
\begin{equation}\label{20180426:eq7}
\begin{array}{l}
\displaystyle{
\vphantom{\Big(}
\sum_{i\in\mc F}(a|\phi_k(u_i)) U^i\id_{V_-^d}
=
\sum_{i\in\mc F}\sum_{h=0}^k \tr(A F^h (F^t)^k U_i (F^t)^k F^{k-h}) 
U^i\id_{V_-^d}
} \\
\displaystyle{
\vphantom{\Big(}
=
\sum_{h=0}^k 
\id_{V_+}
(F^t)^k F^{k-h} A (F^t)^{k-h} \id_{V_-^d}
\,.}
\end{array}
\end{equation}
For the last equality we used the cyclic property of the trace and the completeness 
identity \eqref{eq:compl}.
We also used the facts that 
$\sum_{i\in\mc F}u_i\otimes U^i=\sum_{i\in I}u_i\otimes \id_{V_+}U^i\id_{V_-}$,
and $F^h(F^t)^k\id_{V_-^d}=(F^t)^{k-h}\id_{V_-^d}$ for all $h\leq k$.
If we pair the first term of the RHS of \eqref{20180312:eq4b} with $a\in\mf g$,
we get, by the definition \eqref{eq:phik} of $\phi_k$,
\begin{equation}\label{20180426:eq8}
\begin{array}{l}
\displaystyle{
\vphantom{\Big(}
- \frac1{r_1}\sum_{i \in \mc F'}
\sum_{\ell\in I_{-1}}
(a|[\phi'_k(u'_i),  u_\ell])
{U'}^i\id_{F^tV_-^d}U^\ell
} \\
\displaystyle{
\vphantom{\Big(}
=
- \frac1{r_1}
\sum_{i \in \mc F'}
\sum_{h=0}^k
\sum_{\ell\in I}
\tr(A[(F^t)^{k-h}U'_i(F^t)^k{F'}^{k-h},U_\ell])
{U'}^i\id_{F^tV_-^d}U^\ell\id_{V_-^d}
} \\
\displaystyle{
\vphantom{\Big(}
=
- \frac1{r_1}
\sum_{i \in \mc F'}
\sum_{h=0}^k
\sum_{\ell\in I}
\tr(A[(F^t)^{k-h}U'_i(F^t)^k{F'}^{k-h},\id_{V_-^d}U_\ell\id_{F^tV_-^d}])
{U'}^iU^\ell
\,.}
\end{array}
\end{equation}
For both equalities we used Lemma \ref{20180426:lem}.
We also used the fact that 
$F^h(F^t)^k\id_{V'_-}=(F^t)^{k-h}\id_{V'_-}$ if $h\leq k$ and $k\leq p_1-2$.
Note that the summands in the RHS of \eqref{20180426:eq8}
with $h<k$ vanish,
since $F'\id_{V_-^d}=0$ and $\id_{F^tV_-^d}(F^t)^{k-h}\id_{V'_-}=0$ for $h<k$.
Hence, \eqref{20180426:eq8} becomes, by Lemma \ref{20180216:lem1}(b),
\begin{equation}\label{20180426:eq8b}
\begin{array}{l}
\displaystyle{
\vphantom{\Big(}
\frac1{r_1}
\sum_{i \in \mc F'}
\sum_{\ell\in I}
\tr(A\id_{V_-^d}U_\ell\id_{F^tV_-^d}U'_i(F^t)^k)
{U'}^iU^\ell
} \\
\displaystyle{
\vphantom{\Big(}
=
\frac1{r_1}
\sum_{i \in I'}
\id_{V_+}{U'}^i\id_{F^tV_-^d}U'_i(F^t)^kA\id_{V_-^d}
=
\id_{V_+}(F^t)^kA\id_{V_-^d}
\,,}
\end{array}
\end{equation}
which coincides with the summand $h=k$ in \eqref{20180426:eq7}.
If instead we pair the second term of the RHS of \eqref{20180312:eq4b} with $a\in\mf g$,
we get
\begin{equation}\label{20180426:eq9}
\begin{array}{l}
\displaystyle{
\vphantom{\Big(}
(1- \id_{V_+\cap(F^t)^{k-1}V_-} ) 
\sum_{i \in \mc F'} (a|\phi'_{k-1}(u'_i)) {U'}^i F^t \id_{V_-^d}
} \\
\displaystyle{
\vphantom{\Big(}
=
(1\!-\! \id_{V_+\cap(F^t)^{k-1}V_-} ) 
\sum_{i \in I'}
\sum_{h=0}^{k-1}
\tr(A(F^t)^{k-h-1}U'_i(F^t)^{k-1}{F}^{k-h-1})
\id_{V_+} {U'}^i F^t \id_{V_-^d}
} \\
\displaystyle{
\vphantom{\Big(}
=
(1- \id_{V_+\cap(F^t)^{k-1}V_-} ) 
\sum_{h=0}^{k-1}
\id_{V_+} (F^t)^{k-1}{F}^{k-h-1}A(F^t)^{k-h} \id_{V_-^d}
} \\
\displaystyle{
\vphantom{\Big(}
=
\sum_{h=0}^{k-1}
\id_{V_+} (F^t)^{k}{F}^{k-h}A(F^t)^{k-h} \id_{V_-^d}
\,.}
\end{array}
\end{equation}
For the first equality we used Lemma \ref{20180426:lem},
for the second equality we used the cyclic property of the trace
and the completeness identity \eqref{eq:compl},
and for the last equality we used the following identity:
$$
(1- \id_{V_+\cap(F^t)^{k-1}V_-} ) 
\id_{V_+} (F^t)^{k-1}
=
\id_{V_+} (F^t)^{k}F
\,,
$$
which holds for every $k\geq1$, and we leave as an exercise to the reader.
Note that the RHS of \eqref{20180426:eq9} 
equals the summands in \eqref{20180426:eq7} with $h<k$.
Hence, equation \eqref{20180312:eq4b} holds.
Finally, we need to prove that the coefficients of $(-z)^{p_1-1}$
in both sides of equation \eqref{20180312:eq4} coincide.
In other words, we are left to prove the following identity:
\begin{equation}\label{20180312:eq4c}
\begin{array}{l}
\displaystyle{
\vphantom{\Big(}
\sum_{i\in\mc F}\phi_{p_1-1}(u_i) U^i\id_{V_-^d}
=
(F^t)^{p_1-1}E_0\id_{V_-^d}
} \\
\displaystyle{
\vphantom{\Big(}
+
(1-\id_{V_+\cap(F^t)^{p_1-2}V_-})
\sum_{i\in\mc F'}
\phi'_{p_1-2}(u_i')
 {U'}^i F^t\id_{V_-^d}
\,.}
\end{array}
\end{equation}
Pairing the LHS of \eqref{20180312:eq4c} with $a\in\mf g$, we get,
with the same computation as in \eqref{20180426:eq7},
\begin{equation}\label{20180427:eq1}
\sum_{i\in\mc F}(a|\phi_{p_1-1}(u_i)) U^i\id_{V_-^d}
=
\sum_{h=0}^{p_1-1} 
\id_{V_+}
(F^t)^{p_1-1} F^{p_1-1-h} A (F^t)^{p_1-1-h} \id_{V_-^d}
\,.
\end{equation}
Pairing the first term in the RHS of \eqref{20180312:eq4c} with $a\in\mf g$, we get,
by the completeness identity \eqref{eq:compl}
\begin{equation}\label{20180427:eq2}
\begin{array}{l}
\displaystyle{
\vphantom{\Big(}
(F^t)^{p_1-1}(a|E_0)\id_{V_-^d}
} \\
\displaystyle{
\vphantom{\Big(}
=
(F^t)^{p_1-1}A\id_{V_-^d}
\,,}
\end{array}
\end{equation}
which coincides with the summand of \eqref{20180427:eq1} with $h=p_1-1$.
Finally, pairing the second term in the RHS of \eqref{20180312:eq4c} with $a\in\mf g$, 
we get, by \eqref{20180426:eq9} with $k=p_1-1$,
\begin{equation}\label{20180427:eq3}
\begin{array}{l}
\displaystyle{
\vphantom{\Big(}
(1-\id_{V_+\cap(F^t)^{p_1-2}V_-})
\sum_{i\in\mc F'}
(a|\phi'_{p_1-2}(u_i'))
 {U'}^i F^t\id_{V_-^d}
} \\
\displaystyle{
\vphantom{\Big(}
=
\sum_{h=0}^{p_1-2}
\id_{V_+} (F^t)^{p_1-1}{F}^{p_1-1-h}A(F^t)^{p_1-1-h} \id_{V_-^d}
\,,}
\end{array}
\end{equation}
which coincides with the summands of \eqref{20180427:eq1} with $h<p_1-1$.
Combining \eqref{20180427:eq1}, \eqref{20180427:eq2} and \eqref{20180427:eq3},
we get \eqref{20180312:eq4c},
completing the proof of the Proposition.
\end{proof}

%%%
\subsection{A choice for the subspace $U^\perp\subset\mf g$ complementary to $\mf g^f$}\label{sec:5.6a}

For every pyramid $p$,
fix a subspace $U^\perp\subset\mf g$ 
compatible with the $\Gamma$-grading of $\mf g$
and complementary to $\mf g^f$, cf. \eqref{eq:decompgl12}
(and we let $U$ be its orthocomplement with respect to the trace form).
It is defined as follows:
\begin{equation}\label{20180511:eq1}
U^\perp
=
\id_{F^tV}\End V\oplus\id_{V_-}(\End V)[>0]
\,\subset\mf g
\,.
\end{equation}
In other words, with the ``pictorial'' description presented in Section \ref{sec:3},
$U^\perp$ is spanned by all arrows between the boxes of the pyramid
which do NOT end in a leftmost box ($V_-$),
and the arrows ending in a leftmost box but strictly bending to the right
(i.e. of strictly positive degree).
\begin{lemma}\label{def:U}
The family of subspaces $U^\perp$, depending on the pyramid $p$,
defined in \eqref{20180511:eq1}
is compatible with the $\Gamma$-grading of $\mf g$,
is complementary to $\mf g^f$,
and satisfies the following three conditions:
\begin{enumerate}[(i)]
\item ${U'}^\perp \subset U^\perp$;
\item $[{U'}^\perp, \mf h] =0$, where $\mf h=\Hom(F^tV_-^d, V_-^d)\subset\mf g$;
\item $\phi'_\ell(\Hom(V_+,F^tV_-^d)) \subset U^\perp$
for every $\ell=0,\dots,p_1-2$,
where $\phi'_\ell:\,{\mf g'}_0^f\to{\mf g'}_\ell^f$ is the map defined by \eqref{eq:phik}.
\end{enumerate}
\end{lemma}
\begin{proof}
The subspace $U^\perp$ is obviously compatible with the $\Gamma$-grading.
Let us show that it is complementary to $\mf g^f$.
In order to prove that $U^\perp\cap\mf g^f=0$,
we can proceed degree by degree,
since all spaces $\id_{F^tV}\End V$, $\id_{V_-}\End V$ and $\mf g^f$
are compatible with the $\Gamma$-grading.
Since, moreover, $\mf g^f\subset\mf g[\leq0]$, and obviously $\id_{F^tV}\End V\cap\mf g^f=0$,
we immediately get that $U^\perp\cap\mf g^f=0$.
Hence, to prove that $U^\perp$ is complementary to $\mf g^f$,
we only need to perform a dimension computation.
Obviously,
\begin{equation}\label{20180511:eq2}
\dim\big(\id_{F^tV}\End V\big)
=
N(N-r)
\,.
\end{equation}
By Lemma \ref{lem:gfk}, we have (in the notation \eqref{eq:p})
\begin{equation}\label{20180511:eq3}
\begin{array}{l}
\displaystyle{
\vphantom{\Big(}
\dim(\mf g^f)
=
\sum_{h,k=0}^{p_1-1}\sum_{\ell=0}^{\min\{h,k\}}\dim\Hom(V_+\cap (F^t)^hV_-,V_-\cap F^kV_+)
} \\
\displaystyle{
\vphantom{\Big(}
=
\sum_{i,j=1}^{s}\sum_{\ell=0}^{\min\{p_i,p_j\}-1}r_ir_j
=
\sum_{i,j=1}^{s}r_ir_j\min\{p_i,p_j\}
=
\sum_{i<j}r_ir_j p_j
+
\sum_{i\geq j}r_ir_j p_i
\,.}
\end{array}
\end{equation}
In order to compute the dimension of $\id_{V_-}(\End V)[>0]$,
having in mind the ``pictorial'' description of the space $V$ presented in Section \ref{sec:3},
for each of the $r_j$ leftmost boxes of the pyramid in a rectangle $r_j\times p_j$,
we need to count all boxes strictly to its left,
which are $r_1\times(p_1-p_j)$, $r_2\times(p_2-p_j)$, $\dots$, $r_{j-1}\times(p_{j-1}-p_j)$.
Hence,
\begin{equation}\label{20180511:eq4}
\dim\big(
\id_{V_-}(\End V)[>0]
\big)
=
\sum_{j=1}^sr_j\sum_{i=1}^{j-1}r_i(p_i-p_j)
=
\sum_{i<j}r_ir_j(p_i-p_j)
\,.
\end{equation}
Combining \eqref{20180511:eq2}, \eqref{20180511:eq3} and \eqref{20180511:eq4},
we get
$$
\dim(U^\perp)+\dim(\mf g^f)
=
N(N-r)+\sum_{i,j}r_ir_jp_i
=
N^2
=
\dim(\mf g)
\,.
$$
Next, let us prove condition (i).
Obviously, $F^tV'\subset F^tV$, so that 
$$
\id_{F^tV'}\End V'\subset\id_{F^tV}\End V
\,.
$$
Moreover, ${V'}_-^u\subset V_-^u$,
and $\id_{{V'}_-^d}(\End V')[>0]=0$,
so that 
$$
\id_{V'_-}(\End V')[>0]
=
\id_{{V'}_-^u}(\End V')[>0]
\subset
\id_{V_-}(\End V)[>0]
\,.
$$
As a consequence, condition (i) holds.
Next, let us prove condition (ii).
Since $V_-^d\cap V'=0$, $F^tV_-^d\cap F^tV'=0$ and $F^tV_-^d\cap {V'}_-^u=0$,
we have
$$
[\id_{F^tV'}\End V',\Hom(F^tV_-^d,V_-^d)]=0
\,\,\text{ and }\,\,
[\id_{V'_-}(\End V')[>0],\Hom(F^tV_-^d,V_-^d)]=0
\,.
$$
As a consequence, condition (ii) holds.
Finally, we prove condition (iii).
For every $A\in\End V$ and $i<\ell$, we have
$$
F^i(F^t)^\ell A\in\id_{F^tV}\End V\subset U^\perp
\,,
$$
while, for $A\in\Hom(V,F^tV_-^d)$ we also have
$$
F^\ell(F^t)^\ell A
=
A
\in
\id_{F^tV}\End V
\subset U^\perp
\,.
$$
Recalling the definition \eqref{eq:phik} of the map $\phi_\ell$, we conclude that 
condition (iii) holds,
completing the proof of the Lemma.
\end{proof}
Recall the surjective algebra homomorphism $\eta^f:\,S(\mf g)\twoheadrightarrow S(\mf g^f)$
given by \eqref{eq:projectionpif}.
\begin{lemma}\label{20180510:lem}
\begin{enumerate}[(a)]
\item
For every $v\in S(\mf g')$ and $u\in S(\mf g)$, we have
\begin{equation}\label{20180510:eq7}
\eta^f(vu)=\eta^f(\eta^{f'}(v)u)
\,,
\end{equation}
provided that ${U'}^\perp\subset U^\perp$.
In particular, \eqref{20180510:eq7} holds for the subspaces $U^\perp$ in \eqref{20180511:eq1}.
\item
Let $\mf h\subset\mf g$ be a subspace such that $[{U'}^\perp,\mf h]=0$.
Then, for every $v\in S(\mf g')$ and $h\in\mf h$, we have
\begin{equation}\label{20180510:eq7b}
\eta^f([v,h])=\eta^f([\eta^{f'}(v),h])
\,.
\end{equation}
In particular, \eqref{20180510:eq7b} holds 
for the subspaces $U^\perp$ in \eqref{20180511:eq1}
and $\mf h=\Hom(F^tV_-^d,V_-^d)\subset\mf g$.
\end{enumerate}
\end{lemma}
\begin{proof}
Both $\eta^f$ and $\eta^f\circ\eta^{f'}$ are algebra homomorphisms,
hence it suffices to prove \eqref{20180510:eq7} for $v\in\mf g'\subset\mf g$ and $u=1$.
Let 
\begin{equation}\label{20180510:eq9}
v
=
v_1+v_2
=
{v'}_1+{v'}_2
\,,
\end{equation}
with $v_1\in\mf g^f$, $v_2\in U^\perp$,
${v'}_1\in\mf g'^{f'}$, ${v'}_2\in {U'}^\perp$.
Since, by assumption, ${U'}^\perp\subset U^\perp$, we have
$$
{v'}_1-v_1=v_2-{v'}_2\in U^\perp
\,.
$$
Hence,
$$
\eta^f(\eta^{f'}(v))
=
\eta^f({v'}_1+(f'|{v'}_2))
=
v_1+(f|v_2-{v'}_2)+(f'|{v'}_2)
=
v_1+(f|v_2)
=\eta^f(v)
\,.
$$
For the third equality we used the fact that 
the forms $(f|\,\cdot)$ and $(f'|\,\cdot)$
coincide on $\mf g'$.
This proves part (a).
Next, let us prove part (b).
Note that, since $\eta^f$ and $\eta^{f'}$ are algebra homomorphisms
and $\ad(a)=[a,\,\cdot]$ is a derivation of $S(\mf g)$,
it suffices to prove \eqref{20180510:eq7b} for $v\in\mf g'$.
Let us decompose it as in \eqref{20180510:eq9}.
We have
$$
\begin{array}{l}
\displaystyle{
\vphantom{\Big(}
\eta^f([\eta^{f'}(v),h])
=
\eta^f([{v'}_1+(f'|{v'}_2),h])
=
\eta^f([{v'}_1,h])
} \\
\displaystyle{
\vphantom{\Big(}
=
\eta^f([v,h])-\eta^f([{v'}_2,h])
=
\eta^f([v,h])
\,.}
\end{array}
$$
For the fourth equality we used the assumption that $[{U'}^\perp,\mf h]=0$.
Hence, \eqref{20180510:eq7b} holds.
The last assertion of part (b) follows by
Lemma \ref{def:U}(ii) and the definition of the trace form.
\end{proof}

%%%
\subsection{Premet's form of the operator $W(z)$}\label{sec:5.6}

\begin{proposition}\label{20180405:prop2}
For the subspace $U^\perp$ defined by \eqref{20180511:eq1}
there exists a Premet map
$w:\,\mf g^f\to W(\mf g,f)$ (cf. Definition \ref{def:premet})
such that 
\begin{equation}\label{20180507:eq10}
w(Z(z))=W(z)
\,.
\end{equation}
\end{proposition}
\begin{proof}
By the definition \eqref{eq:Z} of $Z(z)$ and Lemma \ref{lem:gfk}, we have
\begin{equation}\label{20180507:eq11}
\begin{array}{l}
\displaystyle{
\vphantom{\Big(}
Z(z)
=
z\id_{V_+}(1+zF^t)^{-1}\id_{V_-}
} \\
\displaystyle{
\vphantom{\Big(}
+
\sum_{h,k=0}^{p_1-1}
\sum_{i\in\mc F(h,k)}
\sum_{\ell=0}^{\min\{h,k\}}
(-z)^\ell
\phi_\ell(u_i) U^i
\,\in U(\mf g^f)[z]\otimes\Hom(V_-,V_+)
\,.}
\end{array}
\end{equation}
On the other hand,
by Lemma \ref{20180222:lem1}, we also have
\begin{equation}\label{20180507:eq12}
\begin{array}{l}
\displaystyle{
\vphantom{\Big(}
W(z)
=
\widetilde{W}(z)\bar 1
=
z\id_{V_+}(1+zF^t)^{-1}\id_{V_-}
} \\
\displaystyle{
\vphantom{\Big(}
+
\sum_{h,k=0}^{p_1-1}
\sum_{i\in\mc F(h,k)}
\sum_{\ell=0}^{\min\{h,k\}}
(-z)^\ell
\widetilde{w}_{i,\ell}\bar1 U^i
\,\in W(\mf g,f)[z]\otimes\Hom(V_-,V_+)
\,.}
\end{array}
\end{equation}
Since, by Lemma \ref{lem:gfk},
the collection of elements $\phi_\ell(u_i)$,
with $i\in\mc F(h,k)$, $\ell\in\{0,\dots,\min\{h,k\}\}$ and $0\leq h,k\leq p_1-1$,
is a basis of $\mf g^f$,
equation \eqref{20180507:eq10}
obviously defines uniquely a linear map $w:\,\mf g^f\to W(\mf g,f)$,
mapping
\begin{equation}\label{20180508:eq1}
\mf g^f\ni \phi_\ell(u_i)\,\mapsto\,
w(\phi_\ell(u_i)):=\widetilde{w}_{i,\ell}\bar1\,\in W(\mf g,f)
\,.
\end{equation}

Recall the definition \eqref{eq:kaz} of the Kazhdan filtration of $U(\mf g)$,
and the induced filtration \eqref{eq:kazW} of $W(\mf g,f)$.
We next prove the first Premet condition (i) in Definition \eqref{def:premet},
i.e. if $a\in\mf g^f[1-\Delta]$ then $w(a)\in F_\Delta W(\mf g,f)$.
By the definition of the $\Gamma$-grading of $\mf g$
and the definition \eqref{20180502:eq2} of $\mf g_0^f(h,k)$,
it is immediate to check that, if $i\in\mc F(h,k)$,
then $u_i$ has $\Gamma$-degree $-k$.
Moreover, for $0\leq\ell\leq\min\{h,k\}$,
by the definition \eqref{eq:phik} of the map $\phi_\ell$,
we also have $\phi_\ell(u_i)\in\mf g^f[-k+\ell]$.
Hence, by \eqref{20180508:eq1},
the first Premet condition amounts to proving that,
for every $i\in\mc F(h,k)$ and $0\leq\ell\leq\min\{h,k\}$,
we have
\begin{equation}\label{20180508:eq2}
\widetilde{w}_{i,\ell}\bar1\,\in F_{1+k-\ell}W(\mf g,f)
\,.
\end{equation}
Let $\End V=\bigoplus_{j=-p_1+1}^{p_1-1}(\End V)[j]$
be the $\Gamma$-grading of $\End V$
(which is the same as the $\Gamma$-grading of $\mf g$).
We extend the Kazhdan filtrations of $U(\mf g)$ and $W(\mf g,f)$
to filtrations of $U(\mf g)[z]\otimes\End V$ and $W(\mf g,f)[z]\otimes\End V$
by letting $z$ have Kazhdan degree equal to $1$,
and the elements of $(\End V)[j]$ have Kazhdan degree equal to $-j$.
In other words, we let
\begin{equation}\label{20180508:eq3}
F_\Delta \big(
U(\mf g)[z]\otimes\End V
\big)
=
\sum_{\ell=0}^\infty\sum_{k=-p_1+1}^{p_1-1}
z^\ell F_{\Delta+k-\ell}U(\mf g)\otimes(\End V)[k]
\,,
\end{equation}
and the same for $W(\mf g,f)$.
For example, we have, recalling \eqref{eq:Ej},
\begin{equation}\label{20180508:eq8}
\begin{array}{l}
\displaystyle{
\vphantom{\Big(}
F\in (\End V)[-1]\subset 
F_1 \big(
U(\mf g)[z]\otimes\End V
\big)
\,,} \\
\displaystyle{
\vphantom{\Big(}
F^t\in (\End V)[1]\subset 
F_{-1} \big(
U(\mf g)[z]\otimes\End V
\big)
\,,} \\
\displaystyle{
\vphantom{\Big(}
E_j\,\in \mf g[j]\otimes(\End V)[-j]\subset F_1\big(
U(\mf g)[z]\otimes\End V
\big)
\,\,,\,\text{ for every }\,\, j\in\frac12\mb Z
\,.}
\end{array}
\end{equation}
The filtration \eqref{20180508:eq3} is obviously an algebra filtration:
\begin{equation}\label{20180508:eq5}
F_{\Delta_1} \big(
U(\mf g)[z]\otimes\End V
\big)
\cdot
F_{\Delta_2} \big(
U(\mf g)[z]\otimes\End V
\big)
\subset
F_{\Delta_1+\Delta_2} \big(
U(\mf g)[z]\otimes\End V
\big)
\,,
\end{equation}
and, by the properties of the usual Kazhdan filtration of $U(\mf g)$, we also have
\begin{equation}\label{20180508:eq6}
\big[
F_{\Delta_1} \big(
U(\mf g)[z]\otimes\End V
\big)
,
F_{\Delta_2} \big(
U(\mf g)[z]\otimes\End V
\big)
\big]^1
\subset
F_{\Delta_1+\Delta_2-1} \big(
U(\mf g)[z]\otimes\End V
\big)
\,.
\end{equation}
Moreover, although the $\Gamma$-grading of $V'\subset V$ is NOT compatible
with that of $V$ (due to the shift by $\frac12$ of the $x$ coordinates of the boxes
in the corresponding pyramid),
the $\Gamma$-grading of $\End V'$ and $\End V$ 
(being defined by differences in $x$ coordinates) are compatible:
$(\End {V'})[k]\subset(\End V)[k]$.
Likewise, the $\Gamma$-gradings of $\mf g'$ and $\mf g$ are compatible,
and therefore
the Kazhdan filtration of $U(\mf g')$ is compatible with that of $U(\mf g)$.
As a consequence, we have
\begin{equation}\label{20180508:eq7}
F_{\Delta} \big(
U(\mf g')[z]\otimes\End V'
\big)
\subset
F_{\Delta} \big(
U(\mf g)[z]\otimes\End V
\big)
\,.
\end{equation}
Using Lemma \ref{20180222:lem1},
we can then translate the claim \eqref{20180508:eq2},
with respect to this extended filtration, to
\begin{equation}\label{20180508:eq4}
\widetilde{W}(z)
\,\in\,
F_1\big(
U(\mf g)[z]\otimes\End V
\big)
\,,
\end{equation}
which we need to prove.
We shall prove \eqref{20180508:eq4}
by induction on $p_1$
using the recursive definition \eqref{eq:recW} of $\widetilde{W}(z)$.
For $p_1=1$ we have $\widetilde{W}(z)=z\id_V+E$, and, since $\mf g=\mf g[0]$, 
the claim is obvious.
For $p_1>1$, by the inductive assumption and \eqref{20180508:eq7},
we have
\begin{equation}\label{20180508:eq9a}
\widetilde W'(z)
\in 
F_1\big(
U(\mf g)[z]\otimes\End V
\big)
\,,
\end{equation}
and hence, by \eqref{20180508:eq8} and \eqref{20180508:eq6}, we also have
\begin{equation}\label{20180508:eq9b}
[\widetilde W'(z),\id_{F^tV_-^d}E_{-1}]^1
\in 
F_1\big(
U(\mf g)[z]\otimes\End V
\big)
\,,
\end{equation}
and
\begin{equation}\label{20180508:eq9c}
\widetilde W'(z)F^t(z\id_V+E_0+D)
\in 
F_1\big(
U(\mf g)[z]\otimes\End V
\big)
\,.
\end{equation}
Moreover, we have, by \eqref{20180502:eq4} and \eqref{20180507:eq12}
\begin{equation}\label{20180508:eq9d}
\begin{array}{l}
\displaystyle{
\vphantom{\Big(}
\Res_x x^{-1}\widetilde W'(z)
\id_{V_-^u}
(1+x^{-1}F)^{-1}\widetilde W'(x)F^t
} \\
\displaystyle{
\vphantom{\Big(}
=
-\sum_{\ell=0}^{p_1-2}
\widetilde W'(z)
\Res_x (-x)^{-1-\ell}
\id_{V_-\cap F^\ell V_+}
F^\ell
\widetilde W'(x)F^t
} \\
\displaystyle{
\vphantom{\Big(}
=
\sum_{\ell=0}^{p_1-2}
\sum_{i\in\mc F(\ell,p_1-2)}
\widetilde W'(z)
\id_{V_-\cap F^\ell V_+}
F^\ell
\widetilde{w}'_{i,\ell}
{U'}^i
F^t
\,\in 
F_1\big(
U(\mf g)[z]\otimes\End V
\big)
\,.}
\end{array}
\end{equation}
Condition \eqref{20180508:eq4} then follows by the recursive equation \eqref{eq:recW}
and \eqref{20180508:eq9a}-\eqref{20180508:eq9d}.

Finally, we are left to prove Premet's condition (ii) in Definition \eqref{def:premet},
i.e. $\eta^f(\gr_\Delta w(a))=a$ for every $a\in\mf g^f[1-\Delta]$.
As above, we can rewrite this condition in terms of the matrices $Z(z)$ and $\widetilde{W}(z)$
using the extended filtration \eqref{20180508:eq3}:
\begin{equation}\label{20180510:eq1}
\eta^f(\gr_1 \widetilde{W}(z)\bar1) = Z(z)
\,.
\end{equation}
Recall the surjective algebra homomorphism $\eta^f:\,S(\mf g)\twoheadrightarrow S(\mf g^f)$
given by \eqref{eq:projectionpif}.
It clearly maps $S(\mf g)\Span\{b-(f|b)\}_{b\in\mf g_{\geq1}}$ to $0$,
hence it induces the corresponding homomorphism, which we still denote $\eta^f$,
$$
\eta^f\,:\,\,W^{cl}(\mf g,f)\simeq 
\big(S(\mf g)/S(\mf g)\Span\{b-(f|b)\}_{b\in\mf g_{\geq1}}\big)^{\ad \mf g_{\geq1}} 
\,\longrightarrow\, S(\mf g^f)
\,.
$$
Note that $\eta^f$ is a projection map: $(\eta^f)^2=\eta^f$ (since $(f|\mf g^f)=0$).
For $v\in S(\mf g)$, we have 
\begin{equation}\label{20180510:eq2}
\eta^f(v)=\eta^f(v\bar 1^{cl})
\,,
\end{equation}
where $v\mapsto v\bar 1^{cl}$ is the natural quotient map 
$S(\mf g)\twoheadrightarrow S(\mf g)/S(\mf g)\Span\{b-(f|b)\}_{b\in\mf g_{\geq1}}$.
On the other hand, taking the associated graded with respect to the Kazhdan filtration,
we have $\gr U(\mf g)=S(\mf g)$.
Hence, if we denote, as usual,
$v\mapsto v\bar 1$ the natural quotient map
$U(\mf g)\twoheadrightarrow 
U(\mf g)/U(\mf g)\Span\{b-(f|b)\}_{b\in\mf g_{\geq1}}$,
then
\begin{equation}\label{20180510:eq3}
\gr_\Delta(v\bar1)=\gr_\Delta(v)\bar1^{cl}
\,,
\end{equation}
for every $v\in F_\Delta U(\mf g)$.
In view of \eqref{20180510:eq2} and \eqref{20180510:eq3},
we rewrite \eqref{20180510:eq1} as the following equation
\begin{equation}\label{20180510:eq4}
\eta^f(\gr_1 \widetilde{W}(z)) = Z(z)
\,,
\end{equation}
which we need to prove.
We shall prove equation \eqref{20180510:eq4} by induction on $p_1$.
For $p_1=1$, we have $f=0$, $\mf g^f=\mf g=\mf g[0]$
and $\widetilde W(z)=Z(z)=z\id+E$,
so the claim is obvious.
Next, let $p_1>1$.
Recall the recursive definition \eqref{eq:recW} of $\widetilde{W}(z)$.
Recalling \eqref{20180508:eq8}, \eqref{20180508:eq5}, \eqref{20180508:eq6} 
and \eqref{20180508:eq7},
and since $D\in F_0(U(\mf g)[z]\otimes\End V)$, we have
\begin{equation}\label{20180510:eq5}
\begin{array}{l}
\displaystyle{
\vphantom{\Big(}
\gr_1(\widetilde W(z))
=
\gr_1(\widetilde W'(z))\id_{V_-^u}
} \\
\displaystyle{
\vphantom{\Big(}
-
\frac1{r_1}[\gr_1(\widetilde W'(z)),\id_{F^tV_-^d}E_{-1}]^1
-
\gr_1(\widetilde W'(z))F^t(z\id_V+E_0)\id_{V_-^d}
} \\
\displaystyle{
\vphantom{\Big(}
+
\Res_x x^{-1}
\gr_1(\widetilde W'(z))
\id_{V_-^u}
(1+x^{-1}F)^{-1}
\gr_1(\widetilde W'(x))
F^t
\id_{V_-^d}
\,.}
\end{array}
\end{equation}
For the last term, we used the following observation:
taking $\Res_x x^{-1} F(x)$ 
amounts to taking the coefficient of $x^0$ in $F(x)$,
and hence we can assign to $x$ an arbitrary Kazhdan degree, 
without changing the value of $\Res_x x^{-1} \gr_\Delta F(x)$
On the other hand, if we assign to $x$ Kazhdan degree $1$, we obviously have
$$
\begin{array}{l}
\displaystyle{
\vphantom{\Big(}
\gr_1\big(
\widetilde W'(z)
\id_{V_-^u}
(1+x^{-1}F)^{-1}
\widetilde W'(x)
F^t
\id_{V_-^d}
\big)
} \\
\displaystyle{
\vphantom{\Big(}
=
\gr_1(\widetilde W'(z))
\id_{V_-^u}
(1+x^{-1}F)^{-1}
\gr_1(\widetilde W'(x))
F^t
\id_{V_-^d}
\,.}
\end{array}
$$
Next, we apply $\eta^f$ to both sides of \eqref{20180510:eq5}.
As a result, we get
\begin{equation}\label{20180510:eq6}
\begin{array}{l}
\displaystyle{
\vphantom{\Big(}
\eta^f(\gr_1\widetilde W(z))
=
\eta^f\big(
\gr_1(\widetilde W'(z))\id_{V_-^u}
\big)
-
\frac1{r_1}
\eta^f\big(
[\gr_1(\widetilde W'(z)),\id_{F^tV_-^d}E_{-1}]^1
\big)
} \\
\displaystyle{
\vphantom{\Big(}
-
\eta^f\big(
z\gr_1(\widetilde W'(z))F^t\id_{V_-^d}
\big)
-
\eta^f\big(
\gr_1(\widetilde W'(z))F^tE_0\id_{V_-^d}
\big)
} \\
\displaystyle{
\vphantom{\Big(}
+
\eta^f\big(
\Res_x x^{-1}
\gr_1(\widetilde W'(z))
\id_{V_-^u}
(1+x^{-1}F)^{-1}
\gr_1(\widetilde W'(x))
F^t
\id_{V_-^d}
\big)
\,.}
\end{array}
\end{equation}
We discuss each term in the RHS of \eqref{20180510:eq6} separately.
The first term, thanks to Lemma \ref{20180510:lem}(a), is,
by the inductive assumption,
\begin{equation}\label{20180510:eq8}
\eta^f\big(
\gr_1(\widetilde W'(z))\id_{V_-^u}
\big)
=
\eta^f\big(
\eta^{f'}(\gr_1 \widetilde W'(z)))\id_{V_-^u}
\big)
=
\eta^f\big(
Z'(z)\id_{V_-^u}
\big)
\,.
\end{equation}
The same argument works for the third term in the RHS of \eqref{20180510:eq6},
leading to
\begin{equation}\label{20180510:eq11}
\eta^f\big(
z\gr_1(\widetilde W'(z))F^t\id_{V_-^d}
\big)
=
\eta^f\big(
zZ'(z)F^t\id_{V_-^d}
\big)
\,.
\end{equation}
Next, let us consider the second term in the RHS of \eqref{20180510:eq6}.
Note that, by Lemma \ref{20180426:lem}, we have
$$
\id_{F^tV_-^d}E_{-1}
=
\id_{F^tV_-^d}E\id_{V_-^d}
\,\in\mf h\otimes\End V
\,,
$$
where $\mf h$ is as in Lemma \ref{def:U}(ii).
By Lemma \ref{20180510:lem}(b) and the inductive assumption,
we thus have
\begin{equation}\label{20180510:eq10}
\begin{array}{l}
\displaystyle{
\vphantom{\Big(}
\eta^f\big(
[\gr_1(\widetilde W'(z)),\id_{F^tV_-^d}E_{-1}]^1
\big)
=
\eta^f\big(
[\eta^{f'}(\gr_1\widetilde W'(z)),\id_{F^tV_-^d}E_{-1}]^1
\big)
} \\
\displaystyle{
\vphantom{\Big(}
=
\eta^f\big(
[Z'(z),\id_{F^tV_-^d}E_{-1}]^1
\big)
\,.}
\end{array}
\end{equation}
For the fourth term in the RHS of \eqref{20180510:eq6}
we shall need condition (iii) of Lemma \ref{def:U}.
By Lemma \ref{20180510:lem}(a), the inductive assumption, 
and equation \eqref{eq:Z},
we have
\begin{equation}\label{20180510:eq12}
\begin{array}{l}
\displaystyle{
\vphantom{\Big(}
\eta^f\big(
\gr_1(\widetilde W'(z))F^tE_0\id_{V_-^d}
\big)
=
\eta^f\big(
\eta^{f'}(\gr_1\widetilde W'(z))F^tE_0\id_{V_-^d}
\big)
} \\
\displaystyle{
\vphantom{\Big(}
=
\eta^f\big(
Z'(z)F^tE_0\id_{V_-^d}
\big)
=
\eta^f\big(
z\id_{V_+}(1+zF^t)^{-1}F^tE_0\id_{V_-^d}
\big)
} \\
\displaystyle{
\vphantom{\Big(}
+
\sum_{i\in{\mc F}'}
\sum_{\ell\in\mb Z_+}
(-z)^\ell
\eta^f\big(
\phi'_\ell(u'_i) {U'}^i
F^tE_0\id_{V_-^d}
\big)
\,.}
\end{array}
\end{equation}
By Lemma \ref{20180426:lem}, 
$$
\sum_{i\in\mc F'}U'_i \otimes {U'}^i\id_{F^tV_-^d} \,\in\,
\Hom(V_+,F^tV_-^d)\otimes\End V\,.
$$
Hence, by Lemma \ref{def:U}(iii), we have
\begin{equation}\label{20180510:eq13}
\sum_{i\in{\mc F}'}
\eta^f\big(
\phi'_\ell(u'_i)
\big)
 {U'}^i\id_{F^tV_-^d}
=
\sum_{i\in{\mc F}'}
(f|\phi'_\ell(u'_i))
{U'}^i
\id_{F^tV_-^d}
=0
\,,
\end{equation}
since $\phi'_\ell(u'_i)\in{\mf g'}^{f'}\subset\mf g_{\leq0}\subset f^\perp$.
By \eqref{20180510:eq13} the second term in the RHS of \eqref{20180510:eq12}
vanishes, so that
\begin{equation}\label{20180510:eq14}
\eta^f\big(
\gr_1(\widetilde W'(z))F^tE_0\id_{V_-^d}
\big)
=
\eta^f\big(
z\id_{V_+}(1+zF^t)^{-1}F^tE_0\id_{V_-^d}
\big)
\,.
\end{equation}
Finally, we want to prove that the last term in the RHS of \eqref{20180510:eq6}
vanishes.
First, by the same arguments as above, based on Lemma \ref{20180510:lem}(a), 
we have
\begin{equation}\label{20180510:eq15}
\begin{array}{l}
\displaystyle{
\vphantom{\Big(}
\eta^f\big(
\Res_x x^{-1}
\gr_1(\widetilde W'(z))
\id_{V_-^u}
(1+x^{-1}F)^{-1}
\gr_1(\widetilde W'(x))
F^t
\id_{V_-^d}
\big)
} \\
\displaystyle{
\vphantom{\Big(}
=
\eta^f\big(
\Res_x x^{-1}
Z'(z)
\id_{V_-^u}
(1+x^{-1}F)^{-1}
Z'(x)
F^t
\id_{V_-^d}
\big)
\,.}
\end{array}
\end{equation}
Since $\eta^f$ is an algebra homomorphism, 
in order to prove that \eqref{20180510:eq15} is $0$,
it is enough to prove that 
\begin{equation}\label{20180510:eq16}
\eta^f\big(
\id_{V_-^u}
(1+x^{-1}F)^{-1}
Z'(x)
F^t
\id_{V_-^d}
\big)
\end{equation}
vanishes.
Since $\id_{V_-^u}F^i\id_{V_+}(F^t)^j\id_{V_-^d}$
is zero for all $i,j\geq0$,
we can replace $Z'(x)$ in \eqref{20180510:eq16}
by its second contribution coming from the expansion \eqref{eq:Z},
thus getting
%\begin{equation}\label{20180510:eq17}
$$
\sum_{i\in{\mc F}'}
\sum_{\ell\in\mb Z_+}
(-z)^\ell
\id_{V_-^u}
(1+x^{-1}F)^{-1}
\eta^f\big(\phi'_\ell(u'_i)\big)
 {U'}^i
F^t
\id_{V_-^d}
\,,
$$
%\end{equation}
which vanishes by \eqref{20180510:eq13}.
Note that, the exact same argument proving that \eqref{20180510:eq15} vanishes,
also shows that applying $\eta^f$ to the last term in the recursive equation 
\eqref{20180312:eq2} of $Z(z)$ we get zero:
\begin{equation}\label{20180510:eq16b}
\eta^f\big(
\Res_x x^{-1}
z\id_{V_+}(1+zF^t)^{-1}
\id_{V_-^u}
(1+x^{-1}F)^{-1}
Z'(x)
F^t
\id_{V_-^d}
\big)
=0
\,.
\end{equation}
Combining \eqref{20180510:eq6}, \eqref{20180510:eq8}, \eqref{20180510:eq11},
\eqref{20180510:eq10}, \eqref{20180510:eq14}, \eqref{20180510:eq15}, we get
\begin{equation}\label{20180510:eq17}
\begin{array}{l}
\displaystyle{
\vphantom{\Big(}
\eta^f(\gr_1\widetilde W(z))
=
\eta^f\Big(
Z'(z)\id_{V_-^u}
-
\frac1{r_1}
[Z'(z),\id_{F^tV_-^d}E_{-1}]^1
-
zZ'(z)F^t\id_{V_-^d}
} \\
\displaystyle{
\vphantom{\Big(}
-
z\id_{V_+}(1+zF^t)^{-1}F^tE_0\id_{V_-^d}
\Big)
\,.
}
\end{array}
\end{equation}
By equations \eqref{20180312:eq2} and \eqref{20180510:eq16b},
the RHS of \eqref{20180510:eq17} equals $\eta^f(Z(z))$,
which is the same as $Z(z)$ since $\eta^f|_{\mf g^f}=\id$.
This proves \eqref{20180510:eq4}
and concludes the proof of the Proposition.
\end{proof}
\begin{remark}\label{20180511:rem}
The choice of the subspaces $U^\perp\subset\mf g$ complementary to $\mf g^f$
is obviously not unique.
A different choice is, for example,
\begin{equation}\label{20180511:eq1b}
U^\perp
=
\End V\id_{FV}\oplus\Big(\bigoplus_{k=0}^{p_1-1}(F^t)^{k+1}\End V\id_{V_+\cap (F^t)^kV_-}\Big)
\,\subset\mf g\,.
\end{equation}
Note that for this choice of subspace Lemma \ref{def:U}
fails, since conditions (ii) and (iii) do not hold
(but some similar properties do).
On the other hand, it is possible to check that, also for the choice \eqref{20180511:eq1b} 
of $U^\perp$, Proposition \ref{20180405:prop2} holds
(with essentially the same proof).
\end{remark}

%%%%%%%%%%%%%%%%%%%%%%%%%%%%%%%%%%%%%%%
\section{Recursion for the quasideterminant of \texorpdfstring{$W(z)$}{W(z)}
and proof of Theorem \ref{thm:main} for an aligned pyramid}\label{sec:5d}

%%%
\subsection{Recursive formula for the quasideterminant of $W(z)$}\label{sec:5.7}

Consider the operator $\widetilde W(z) \in U(\mf g)[z]\otimes \Hom(V_-,V_+) $ as in \eqref{eq:recW}, and the generalized quasideterminant $ | \widetilde W(z) |_{V_-^d,V_+^d} $ according to the decomposition \eqref{eq:V_-ud} of $ V_\pm $.
\begin{proposition}\label{20180408:prop2}
For the operator $| \widetilde W(z) |_{V_-^d,V_+^d} $ the following recursive formula holds:
\begin{equation}\label{20180406:eq1}
\begin{array}{l}
\displaystyle{
\vphantom{\Big(}
| \widetilde W(z) |_{V_-^d,V_+^d} 
=
-\frac1{r_1}[| \widetilde W'(z) |_{F^tV_-^d,V_+^d} ,\id_{F^tV_-^d}E_{-1}]^1
} \\
\displaystyle{
\vphantom{\Big(}
-| \widetilde W'(z) |_{F^tV_-^d,V_+^d} F^t(z\id_V+E_0+D)\id_{V_-^d}
\,,}
\end{array}
\end{equation}
where the bracket $[\cdot\,,\,\cdot]^1$ is on the first factor of $U(\mf g)\otimes\End V$
(cf. Section \ref{sec:3.4b}), and for the quasideterminant $| \widetilde W'(z) |_{F^tV_-^d,V_+^d}  $ we are considering the decompositions  \eqref{eq:V''_-ud} and  \eqref{eq:V_-ud} of $ V'_- $ and $ V_+ $ respectively.
\end{proposition}
\begin{remark}\label{20180511:rem2}
Note that equation \eqref{20180406:eq1} is indeed a recursive formula
since, by the hereditary property \eqref{eq:hered} of the quasideterminant,
we have
$$
| \widetilde W'(z) |_{F^tV_-^d,V_+^d}
=
\big|
| \widetilde W'(z) |_{{V'}_-^d,{V'}_+^d}
\big|_{F^tV_-^d,V_+^d}
\,.
$$
\end{remark}
\begin{proof}
First, we prove that the quasideterminants 
$ | \widetilde W(z) |_{V_-^d,V_+^d} $ and $ | \widetilde W'(z) |_{F^tV_-^d,V_+^d} $ exist. 
For the first quasideterminant,
by Proposition \ref{prop:quasidet} it suffices to show that 
$ \id_{V_+^u}\widetilde W(z)\id_{V_-^u} $ is invertible 
in $ U(\mf g)((z^{-1})) \otimes \Hom(V_-^u,V_+^u) $. 
By Lemma \ref{20180222:lem1} we have
\begin{equation}\label{20180511:eq5}
\id_{V_+^u}\widetilde W(z)\id_{V_-^u} 
= z \id_{V_+^u}(1 + zF^t)^{-1} \id_{V_-^u} + \sum_{i \in \mc F} \widetilde w_i(z) \id_{V_+^u}U^i\id_{V_-^u}
\,.
\end{equation}
The first term in the RHS of \eqref{20180511:eq5}
is invertible, with inverse
$$
z^{-1}\id_{V_-^u}(1 + z^{-1}F)^{-1}\id_{V_+^u}
\,\in \mb F[z^{-1}]\otimes\Hom(V_+^u,V_-^u)
\,.
$$
Hence, we can rewrite the RHS of \eqref{20180511:eq5} as
\begin{equation}\label{20180511:eq6}
z \id_{V_+^u}(1 + zF^t)^{-1} \id_{V_-^u}  \big( \id_{V_-^u} 
+\sum_{i \in \mc F} 
z^{-1}\id_{V_-^u}(1 + z^{-1}F)^{-1}\id_{V_+^u}
\widetilde w_i(z) U^i\id_{V_-^u}\big)
 \,.
\end{equation}
It is immediate to check, using Lemma \ref{20180222:lem1},
that 
$$
\id_{V_-^u}(1 + z^{-1}F)^{-1}\id_{V_+^u}
\widetilde w_i(z) U^i\id_{V_-^u}
\,\in U(\mf g)[z^{-1}]\otimes\End(V_-^u)
\,.
$$
Hence, 
$$
\id_{V_-^u} 
+\sum_{i \in \mc F} 
z^{-1}\id_{V_-^u}(1 + z^{-1}F)^{-1}\id_{V_+^u}
\widetilde w_i(z) U^i\id_{V_-^u}
$$
has inverse in $U(\mf g)[[z^{-1}]]\otimes\End(V_-^u)$,
which can be computed by geometric expansion.
As a consequence, \eqref{20180511:eq6} has inverse in 
$ U(\mf g)((z^{-1})) \otimes \Hom(V_-^u,V_+^u) $,
and the quasideterminant $ | \widetilde W(z) |_{V_-^d,V_+^d} $ exists.
With exactly the same argument,
one can show that $\widetilde{W}(z)$ is invertible
(so that $|\widetilde{W}(z)|_{V^d_-,V^d_+}$ exists)
and that $| \widetilde W'(z) |_{F^tV_-^d,V_+^d} $ exists as well.

Next, we prove equation \eqref{20180406:eq1}.
By the definition \eqref{eq:quasidet2} of generalized quasideterminant we have
\begin{equation}\label{20180406:eq2}
| \widetilde W(z) |_{V_-^d,V_+^d} 
=
\id_{V_+^d}
\Big(
\widetilde W(z)-
\widetilde W(z)\id_{V_-^u}
\big(\id_{V_+^u}\widetilde W(z)\id_{V_-^u}\big)^{-1}
\id_{V_+^u}\widetilde W(z)
\Big)
\id_{V_-^d}
\,.
\end{equation}
By the recursive formula \eqref{eq:recW} for $\widetilde{W}(z)$, 
the RHS of equation \eqref{20180406:eq2} becomes
\begin{equation}\label{20180406:eq3}
\begin{array}{l}
\displaystyle{
\vphantom{\Big(}
-
\frac1{r_1}
\id_{V_+^d}
[\widetilde W'(z),\id_{F^tV_-^d}E_{-1}]^1
\id_{V_-^d}
-
\id_{V_+^d}
\widetilde W'(z)F^t(z\id_V+E_0+D)
\id_{V_-^d}
} \\
\displaystyle{
\vphantom{\Big(}
+\Res_x x^{-1}
\id_{V_+^d}
\widetilde W'(z)
\id_{V_-^u}
(1+x^{-1}F)^{-1}\widetilde W'(x)F^t
\id_{V_-^d}
} \\
\displaystyle{
\vphantom{\Big(}
+
\frac1{r_1}\id_{V_+^d}
\widetilde W'(z)\id_{V_-^u}
\big(\id_{V_+^u}\widetilde W'(z)\id_{V_-^u}\big)^{-1}
\id_{V_+^u}
[\widetilde W'(z),\id_{F^tV_-^d}E_{-1}]^1
\id_{V_-^d}
} \\
\displaystyle{
\vphantom{\Big(}
+
\id_{V_+^d}
\widetilde W'(z)\id_{V_-^u}
\big(\id_{V_+^u}\widetilde W'(z)\id_{V_-^u}\big)^{-1}
\id_{V_+^u}
\widetilde W'(z)F^t(z\id_V+E_0+D)
\id_{V_-^d}
} \\
\displaystyle{
\vphantom{\Big(}
-
\!\Res_x\! x^{-1}\!
\id_{V_+^d}
\widetilde W'\!(z)\id_{V_-^u}
\big(\!\id_{V_+^u}\widetilde W'\!(z)\id_{V_-^u}\!\big)^{-1}\!
\id_{V_+^u}\!
\widetilde W'\!(z)
\id_{V_-^u}\!
(1\!+\!x^{-1}\!F)^{-1}\widetilde W'\!(x)F^t
\id_{V_-^d}
.}
\end{array}
\end{equation}
Note that the two residue terms (the third and the last one) 
in \eqref{20180406:eq3} vanish, since
$$ 
\big(\id_{V_+^u}\widetilde W'(z)\id_{V_-^u}\big)^{-1}\id_{V_+^u}\widetilde W'(z)\id_{V_-^u} 
= \id_{V_-^u}
\,.
$$
Combining the second and fifth terms in \eqref{20180406:eq3},
we get
\begin{equation}\label{20180406:eq4}
\begin{array}{l}
\displaystyle{
\vphantom{\Big(}
-\id_{V_+^d}
\widetilde W'(z)F^t(z\id_V+E_0+D)
\id_{V_-^d}
} \\
\displaystyle{
\vphantom{\Big(}
+
\id_{V_+^d}
\widetilde W'(z)\id_{V_-^u}
\big(\id_{V_+^u}\widetilde W'(z)\id_{V_-^u}\big)^{-1}
\id_{V_+^u}
\widetilde W'(z)F^t(z\id_V+E_0+D)\id_{V_-^d}
} \\
\displaystyle{
\vphantom{\Big(}
=
-\id_{V_+^d}
\Big(
\widetilde W'(z)
-
\widetilde W'(z)\id_{V_-^u}
\big(\id_{V_+^u}\widetilde W'(z)\id_{V_-^u}\big)^{-1}
\id_{V_+^u}
\widetilde W'(z)
\Big)
\id_{F^tV_-^d}
\times
} \\
\displaystyle{
\vphantom{\Big(}
\times
F^t(z\id_V+E_0+D)
\id_{V_-^d}
=
-|\widetilde W'(z)|_{F^tV_-^d,V_+^d}
F^t(z\id_V+E_0+D)
\id_{V_-^d}
\,.}
\end{array}
\end{equation}
For the first equality in \eqref{20180406:eq4} we used the fact that 
$ F^t(z\id_V + E_0 + D)\id_{V_-^d} \subset F^tV_-^d $, 
while for the second equality we used the definition \eqref{eq:quasidet2} 
of the quasideterminant $| \widetilde W'(z) |_{F^tV_-^d,V_+^d}  $, 
with respect to the decompositions  \eqref{eq:V''_-ud} and  \eqref{eq:V_-ud} 
of $ V'_- $ and $ V_+ $ respectively.
On the other hand, combining the first and fourth terms in \eqref{20180406:eq3}, we get
\begin{equation}\label{20180406:eq5}
\begin{array}{l}
\displaystyle{
\vphantom{\Big(}
-
\frac1{r_1}
\id_{V_+^d}
[\widetilde W'(z),\id_{F^tV_-^d}E_{-1}]^1
\id_{V_-^d}
} \\
\displaystyle{
\vphantom{\Big(}
+
\frac1{r_1}
\id_{V_+^d}
\widetilde W'(z)\id_{V_-^u}
\big(\id_{V_+^u}\widetilde W'(z)\id_{V_-^u}\big)^{-1}
\id_{V_+^u}
[\widetilde W'(z),\id_{F^tV_-^d}E_{-1}]^1
\id_{V_-^d}
} \\
\displaystyle{
\vphantom{\Big(}
=
-
\frac1{r_1}[\id_{V_+^d}\Big(\widetilde W'(z)-\widetilde W'(z)\id_{V_-^u}
\big(\id_{V_+^u}\widetilde W'(z)\id_{V_-^u}\big)^{-1}
\id_{V_+^u}
\widetilde W'(z)
\Big)%\id_{F^tV_-^d}
,\id_{F^tV_-^d}E_{-1}]^1
} \\
\displaystyle{
\vphantom{\Big(}
-
\frac1{r_1}\id_{V_+^d}\sum_{j \in I_{-1}}
[\widetilde W'(z)\id_{V_-^u}
\big(\id_{V_+^u}\widetilde W'(z)\id_{V_-^u}\big)^{-1}, u_j]
\id_{V_+^u}\widetilde W'(z)
\id_{F^tV_-^d}U^j
} \\
\displaystyle{
\vphantom{\Big(}
=
-
\frac1{r_1}[  |\widetilde W'(z)|_{F^tV_-^d, V_+^d},\id_{F^tV_-^d}E_{-1}]^1
} \\
\displaystyle{
\vphantom{\Big(}
-
\frac1{r_1}\id_{V_+^d}\sum_{j \in I_{-1}}
[\widetilde W'(z)\id_{V_-^u}
\big(\id_{V_+^u}\widetilde W'(z)\id_{V_-^u}\big)^{-1}, u_j]
\id_{V_+^u}\widetilde W'(z)
\id_{F^tV_-^d}U^j
\,.}
\end{array}
\end{equation}
For the first equality we used the Leibniz rule for the bracket $[\cdot\,,\,\cdot]^1$
and the definition \eqref{eq:Ej} of $E_{-1}$.
Combining \eqref{20180406:eq2}, \eqref{20180406:eq3}, \eqref{20180406:eq4}
and \eqref{20180406:eq5},
we get our claim \eqref{20180406:eq1}
once we prove that the last term in the RHS of \eqref{20180406:eq5}
vanishes.
By Lemma \eqref{20180426:lem}, we have
$$
\sum_{j \in I_{-1}}
u_j\otimes \id_{F^tV_-^d}U^j
=
\sum_{j\in J} u_j\otimes U^j
\,,
$$
where $\{u_j\}_{j\in J}$ is a basis of 
$$
\mf h=\Hom(F^tV_-^d, V_-^d)\subset \mf g
\,.
$$
Hence, in order to prove the claim
it suffices to show that
\begin{equation}\label{20180522:eq1}
[ \widetilde W'(z)\id_{V_-^u}, u] = 0
\,\,\text{ for every }\,\,
u\in \mf h
\,.
\end{equation}
Note that we can decompose $V_-^u = {V'}_-^u \oplus (V_-^u \cap {V'}_-^d)$. 
Hence, the LHS of \eqref{20180522:eq1} decomposes as
\begin{equation}\label{20180522:eq2}
[ \widetilde W'(z)\id_{{V'}_-^u}, u] 
+
[ \widetilde W'(z)\id_{V_-^u\cap{V'}_-^d}, u]
\,.
\end{equation}
By the recursive formula \eqref{eq:recW} for $ \widetilde W(z) $, we have
$$ \widetilde W'(z)\id_{{V'}_-^u} =  \widetilde W''(z)\id_{{V'}_-^u} 
\in U(\mf g'') \otimes \Hom({V'}_-^u,V_+)\,,
$$
and the Lie bracket with $u \in \mf h$ vanishes since, obviously, 
$ [\mf g'',\mf h] = 0 $. 
Hence, the first term in \eqref{20180522:eq2} vanishes.
On the other hand, by the recursive formula \eqref{eq:recW}
applied to $ \widetilde W'(z) $, we have
\begin{equation}\label{20180516:eq1}
\begin{array}{l}
\displaystyle{
\vphantom{\Big(}
 \widetilde W'(z)\id_{V_-^u \cap {V'}_-^d} = -\frac1{r_1+r'_1}[\widetilde W''(z), \id_{F^t{V'}_-^d}E'_{-1}]^1\id_{V_-^u \cap {V'}_-^d}
} \\
\displaystyle{
\vphantom{\Big(}
- \widetilde W''(z)F^t(z\id_{V'} + E'_0 + D')\id_{V_-^u \cap {V'}_-^d}
} \\
\displaystyle{
\vphantom{\Big(}
+ \Res_x x^{-1} \widetilde W''(z) \id_{{V'}_-^u}(1 + x^{-1}F)^{-1}
\widetilde W''(x)F^t\id_{V_-^u \cap {V'}_-^d}
\,.}
\end{array}
\end{equation}
As before, the bracket of $u \in \mf h$ 
with $ \widetilde W''(z) \in U(\mf g'')[z]\otimes \End V'' $ vanishes,
since $[\mf h,\mf g'']=0$.
Hence, the second term in \eqref{20180522:eq2} is
\begin{equation}\label{20180522:eq3}
-\frac1{r_1+r'_1}[\widetilde W''(z), \id_{F^t{V'}_-^d}[E'_{-1},u]]^1\id_{V_-^u \cap {V'}_-^d}
- \widetilde W''(z)F^t[E'_0,u]\id_{V_-^u \cap {V'}_-^d}
\,.
\end{equation}
By Lemma \ref{20180426:lem}, we have 
$$
\id_{F^t{V'}_-^d}E'_{-1} \id_{V_-^u \cap {V'}_-^d}
=
\sum_{j\in\tilde J_1}u_j U^j
\,,
$$
where $\{u_j\}_{j\in\tilde J_1}$ is a basis of 
$$
\tilde{\mf h}_1=\Hom(F^t{V'}_-^d,V_-^u \cap {V'}_-^d)\subset\mf g \otimes \End V' 
\,.
$$
On the other hand, we clearly have
$[\tilde{\mf h}_1,\mf h]=0$,
since
${V'} \cap V_-^d = 0 $ and $ (V_-^u\cap {V'}_-^d) \cap F^tV_-^d = 0 $.
Hence, the first summand of \eqref{20180522:eq3} vanishes.
Similarly, by Lemma \ref{20180426:lem}, we have 
$$
E'_0\id_{V_-^u \cap {V'}_-^d} 
=
\sum_{j\in\tilde J_2}u_j U^j
\,,
$$
where $\{u_j\}_{j\in\tilde J_2}$ is a basis of 
$$
\tilde{\mf h}_2=\Hom({V'}_-^d, V_-^u\cap {V'}_-^d)\subset\mf g \otimes \End V' 
\,,
$$
and again we have
$[\tilde{\mf h}_2,\mf h]=0$
for exactly the same reason as before.
Hence, the second summand of \eqref{20180522:eq3} vanishes.
This completes the proof.
\end{proof}

%%%
\subsection{Proof of Theorem \ref{thm:main} in the case of an aligned pyramid}\label{sec:5.8}

First, we observe that the operator 
$ W(z)=\widetilde{W}(z)\bar 1 \in W(\mf g,f)[z] \otimes \Hom(V_-,V_+) $
defined recursively by \eqref{eq:recW}
is the same as the operator $W(z)$ 
which is associated via \eqref{eq:W}
to the Premet map $w:\,\mf g^f \longrightarrow W(\mf g,f)$ 
given by Proposition \ref{20180405:prop2}.

Since the algebra structure of $ W(\mf g,f) $ is induced by that of $ U(\mf g)$ we have that
$$
| \widetilde W(z) |_{V_-^d,V_+^d} \bar 1 = | W(z) |_{V_-^d,V_+^d}
\,.
$$
Hence, the claim \eqref{eq:main} can be restated as the following equation:
\begin{equation}\label{20180516:eq2}
\widetilde L(z) 
= 
| \widetilde W(z) |_{V_-^d,V_+^d}
\,,
\end{equation}
which we need to prove.
This equation holds since,
by Proposition \ref{prop:recursionL},
$\tilde L(z)$ satisfies the recursive equation \eqref{eq:recL},
and, by Proposition \ref{20180408:prop2},
$| \widetilde W(z) |_{V_-^d,V_+^d}$ satisfies exactly the same recursive 
equation \eqref{20180406:eq1} (cf. Remark \ref{20180511:rem2}).
The fact that for $p_1=1$ \eqref{20180516:eq2}
is clear since, in this case,
$ \widetilde L(z) = z\id_V + E $ (cf. Example \ref{ex:f=0}),
while 
$ |\widetilde W(z) |_{V_-^d,V_+^d} = \widetilde W(z)  = z\id_V + E $ 
by definition of $W(z) $.

%%%%%%%%%%%%%%%%%%%%%%%%%%%%%%%%%%%%%%%
\section{The general case: arbitrary pyramids}
\label{sec:6}

In the present section we prove Theorem \ref{thm:main} in the general case, 
when the good grading $\Gamma$ corresponds to a pyramid $p$ which is not
necessarily aligned to the right.

Throughout the Section we will consider several gradings of the Lie algebra $\mf g$,
good or, in general, \emph{semisimple} (cf. Definition \ref{def:grad} below)
and we will make the following fundamental assumption.
We fix, once and forall, a nilpotent element $f$, i.e. a partition 
$p_1^{r_1}\dots p_s^{r_s}$ of $N$, and we fix a basis of the vector space $V$
of the form $f^k(e_i)$, $i=1,\dots,r=\sum_ir_i$, $k=0,\dots,p_i-1$.
Hence, the basis of the vector space $V$ can be depicted, as we did in 
the previous sections, as a collection of boxes arranged in a pyramid
with rectangles of sizes $p_i\times r_i$, $i=1,\dots,s$
(and the exact shape of the pyramid is associated to a good grading of $\mf g$).

As before, we shall assume that all the good gradings $\Gamma$ that we encounter
are compatible with the given basis.
This amounts to require that the semisimple element 
$h_\Gamma$ associated to the grading (see Section \ref{sec:3.1})
is diagonal in the fixed basis.
Similarly, for all semisimple gradings $G$ that we shall encounter,
we require that the semisimple element $h$ attached to it is diagonal
in the given basis.

%%%%%%%%%%%%%%%%%%%%%%%%%%%%%%%%%%%%%%%
\subsection{Adjacent pyramids}
\label{sec:6.0}

Recall from Section \ref{sec:3.1} that good gradings $\Gamma$
for the nilpotent element $f\in\mf{gl}_N$
are in bijective correspondence with pyramids of $N$ boxes.
Following Brundan and Goodwin \cite{BG05}
we introduce the notion of adjacent pyramids, which will be essential
in the present section.

Given a pyramid $p$ for the partition $p_1^{r_1}p_2^{r_2}\dots p_s^{r_s}$ of $N$,
we denote by $x^{(p)}(i)$ the $x$ coordinate ot the 
box $i$.
We say that a pyramid $\tilde p$, corresponding to the same partition of $N$, 
is \emph{adjacent to the right} of $p$ 
if $x^{\tilde p}(i)=x^{(p)}(i)$ or $x^{(p)}(i)+\frac12$ for every box $i$.
In this case we say that $p$ is adjacent to the left of $\tilde p$
and that $p$ and $\tilde p$ are \emph{adjacent}.

For example, in the following Figure \ref{fig:adj}, $\tilde p$ is adjacent to the right of $p$,
while $\tilde{\tilde p}$ is adjacent to the right of $\tilde p$ but not of $p$.
\begin{figure}[H]
\setlength{\unitlength}{0.15in}
% selecting unit length
\centering
\begin{picture}(14,9)

\put(-5,7){\framebox(2,2){}}

\put(-6,5){\framebox(2,2){}}
\put(-4,5){\framebox(2,2){}}

\put(-7,3){\framebox(2,2){}}
\put(-5,3){\framebox(2,2){}}
\put(-3,3){\framebox(2,2){}}
\put(-1,3){\framebox(2,2){}}

\put(-8,2){\vector(1,0){9}}
\put(0,1){$x$}

\put(-6,1.8){\line(0,1){0.4}}
\put(-5,1.6){\line(0,1){0.8}}
\put(-4,1.8){\line(0,1){0.4}}
\put(-3,1.6){\line(0,1){0.8}}
\put(-2,1.8){\line(0,1){0.4}}
\put(-1,1.6){\line(0,1){0.8}}
\put(0,1.8){\line(0,1){0.4}}

\put(-3.3,0.6){0}
\put(-5.5,0.6){-1}
\put(-1.2,0.6){1}

\put(-3.3,-0.6){$p$}

\put(6,7){\framebox(2,2){}}

\put(5,5){\framebox(2,2){}}
\put(7,5){\framebox(2,2){}}

\put(3,3){\framebox(2,2){}}
\put(5,3){\framebox(2,2){}}
\put(7,3){\framebox(2,2){}}
\put(9,3){\framebox(2,2){}}

\put(2,2){\vector(1,0){9}}
\put(10,1){$x$}

\put(4,1.8){\line(0,1){0.4}}
\put(5,1.6){\line(0,1){0.8}}
\put(6,1.8){\line(0,1){0.4}}
\put(7,1.6){\line(0,1){0.8}}
\put(8,1.8){\line(0,1){0.4}}
\put(9,1.6){\line(0,1){0.8}}
\put(10,1.8){\line(0,1){0.4}}

\put(6.7,0.6){0}
\put(4.5,0.6){-1}
\put(8.8,0.6){1}

\put(6.7,-0.6){$\tilde p$}

\put(17,7){\framebox(2,2){}}

\put(15,5){\framebox(2,2){}}
\put(17,5){\framebox(2,2){}}

\put(13,3){\framebox(2,2){}}
\put(15,3){\framebox(2,2){}}
\put(17,3){\framebox(2,2){}}
\put(19,3){\framebox(2,2){}}

\put(12,2){\vector(1,0){9}}
\put(20,1){$x$}

\put(14,1.8){\line(0,1){0.4}}
\put(15,1.6){\line(0,1){0.8}}
\put(16,1.8){\line(0,1){0.4}}
\put(17,1.6){\line(0,1){0.8}}
\put(18,1.8){\line(0,1){0.4}}
\put(19,1.6){\line(0,1){0.8}}
\put(20,1.8){\line(0,1){0.4}}

\put(16.7,0.6){0}
\put(14.5,0.6){-1}
\put(18.8,0.6){1}

\put(16.7,-0.6){$\tilde{\tilde p}$}

\end{picture}
\caption{}\label{fig:adj}
\end{figure}

It is clear by the definition,
that given an arbitrary pyramid $p$,
we can construct a sequence of pyramids associated to the same partition of $N$,
and hence of good gradings
\begin{equation}\label{20180523:eq4}
\Gamma_0=\Gamma\,,\,\, \Gamma_1\,,\,\,\dots\,,\,\, \Gamma_K=\Gamma_R\,,
\end{equation}
such that the pyramid associated to $\Gamma_i$ is adjacent to the right 
to the pyramid corresponding to $\Gamma_{i-1}$
and $\Gamma_R$ is the (unique) grading corresponding to the
pyramid aligned to the right for that partition.
Hence, any two pyramids, are linked by a chain of adjacent pyramids.

Let $p$ and $\tilde p$ be two adjacent pyramids,
and let $\Gamma: \mf g = \bigoplus_{i \in \frac12 \mb Z} \mf g^{\Gamma}[i] $
and $\tilde \Gamma: \mf g = \bigoplus_{j \in \frac12 \mb Z} \mf g^{\tilde\Gamma}[j] $
be the corresponding gradings.
Note that given an elementary matrix $e_{ij}\in\mf g$, 
depicted as an arrow from box $j$ to box $i$ in the pyramid,
which has degree $k=x^{(p)}(i)-x^{(p)}(j)$ with respect to the grading $\Gamma$,
it must have degree $k$, $k+\frac12$ or $k-\frac12$ 
with respect to the grading $\tilde\Gamma$.
In other words,
\begin{equation}\label{20180523:eq1}
\mf g^{\Gamma}[k]
=
(\mf g^{\Gamma}[k]\cap\mf g^{\tilde\Gamma}[k])
\oplus
(\mf g^{\Gamma}[k]\cap\mf g^{\tilde\Gamma}[k+\frac12])
\oplus
(\mf g^{\Gamma}[k]\cap\mf g^{\tilde\Gamma}[k-\frac12])
\,.
\end{equation}
\begin{lemma}\cite[Lemma 26]{BG05}\label{20180419:lem1}
Let $\Gamma$ and $\tilde\Gamma$ be two adjacent good gradings for $f$.
There exist Lagrangian subspaces $ \mf l \subseteq \mf g^{\Gamma}[\frac12] $
and $ \tilde{\mf l} \subseteq \mf g^{\tilde\Gamma}[\frac12] $
(with respect to the form $ \omega $ defined in Section \ref{sec:2.1})
such that
\begin{equation}\label{20180419:eq4}
\mf l \oplus \mf g^{\Gamma}[\geq1] = \tilde{\mf l} \oplus \mf g^{\tilde\Gamma}[\geq1]
\,.
\end{equation}
\end{lemma}
Before proving the claim, let us view how to implement it
for the pyramids $p$ and $\tilde p$ of Figure \ref{fig:adj}:
\begin{figure}[H]
\setlength{\unitlength}{0.15in}
\setlength\fboxsep{0pt}
% selecting unit length
\centering
\begin{picture}(14,9)

\put(-3,7){\framebox(2,2){}}

\put(-4,5){\framebox(2,2){}}
\put(-2,5){\framebox(2,2){}}

\put(-5,3){\framebox(2,2){}}
\put(-3,3){\framebox(2,2){}}
\put(-1,3){\framebox(2,2){}}
\put(1,3){\framebox(2,2){}}

\put(-6,2){\vector(1,0){9}}
\put(2,1){$x$}

\put(-4,1.8){\line(0,1){0.4}}
\put(-3,1.6){\line(0,1){0.8}}
\put(-2,1.8){\line(0,1){0.4}}
\put(-1,1.6){\line(0,1){0.8}}
\put(0,1.8){\line(0,1){0.4}}
\put(1,1.6){\line(0,1){0.8}}
\put(2,1.8){\line(0,1){0.4}}

\put(-1.3,0.6){0}
\put(-3.5,0.6){-1}
\put(0.8,0.6){1}

\put(-1.3,-0.6){$p$}

\put(-4,4){\color{green}{\vector(1,2){1}}}
\put(-2,4){\color{green}{\vector(1,2){1}}}
\put(-3,6){\color{red}{\vector(1,-2){1}}}
\put(-1,6){\color{red}{\vector(1,-2){1}}}
\put(-3,6){\color{gray}{\vector(1,2){1}}}
\put(-2,8){\color{gray}{\vector(1,-2){1}}}

\put(11,7){\framebox(2,2){}}

\put(10,5){\framebox(2,2){}}
\put(12,5){\framebox(2,2){}}

\put(8,3){\framebox(2,2){}}
\put(10,3){\framebox(2,2){}}
\put(12,3){\framebox(2,2){}}
\put(14,3){\framebox(2,2){}}

\put(7,2){\vector(1,0){9}}
\put(15,1){$x$}

\put(9,1.8){\line(0,1){0.4}}
\put(10,1.6){\line(0,1){0.8}}
\put(11,1.8){\line(0,1){0.4}}
\put(12,1.6){\line(0,1){0.8}}
\put(13,1.8){\line(0,1){0.4}}
\put(14,1.6){\line(0,1){0.8}}
\put(15,1.8){\line(0,1){0.4}}

\put(11.7,0.6){0}
\put(9.5,0.6){-1}
\put(13.8,0.6){1}

\put(11.7,-0.6){$\tilde p$}

\put(11,4){\color{red}{\vector(1,4){1}}}
\put(12,8){\color{green}{\vector(1,-4){1}}}
\put(11,6){\color{gray}{\vector(1,2){1}}}
\put(12,8){\color{gray}{\vector(1,-2){1}}}

\end{picture}
\caption{}\label{fig:adj2}
\end{figure}
The spaces $\mf g^{\Gamma}[\frac12]$ and $\mf g^{\tilde\Gamma}[\frac12]$
are spanned by the depicted arrows.
Consider for example the space $\mf g^{\Gamma}[\frac12]$.
The green arrows are the ones which have degree $1$ with respect 
to the grading $\tilde\Gamma$,
and therefore we need to include them in $\mf l$ 
in order for equation \eqref{20180419:eq4} to hold.
The red arrows are the ones which have degree $0$ with respect 
to the grading $\tilde\Gamma$,
and therefore we do not include them in $\mf l$.
The gray arrows are those that have degree $\frac12$ with respect to both
gradings $\Gamma$ and $\tilde\Gamma$
(and therefore are depicted in both pyramids),
and we choose half of them (for example those going upwards)
in order to form a Lagrangian subspace.
\begin{proof}[Proof of Lemma \ref{20180419:lem1}]
By \eqref{20180523:eq1} with $k=\frac12$, we have
$$
\mf g^{\Gamma}[\frac12] 
= 
\big(\mf g^{\Gamma}[\frac12] \cap \mf g^{\tilde\Gamma}[0]\big) 
\oplus 
\big(\mf g^{\Gamma}[\frac12] \cap \mf g^{\tilde\Gamma}[1]\big)
\oplus
\big(\mf g^{\Gamma}[\frac12] \cap \mf g^{\tilde\Gamma}[\frac12]\big)
\,.
$$
Note that
$$
\dim\big(\mf g^{\Gamma}[\frac12] \cap \mf g^{\tilde\Gamma}[0]\big) 
=
\dim\big(\mf g^{\Gamma}[\frac12] \cap \mf g^{\tilde\Gamma}[1]\big)
\,.
$$
(The LHS corresponds to the red arrow in Figure \ref{fig:adj2},
while the RHS corresponds to the green arrows.)
Moreover, $\mf g^{\tilde\Gamma}[\frac12]$
is clearly orthogonal 
to $\mf g^{\tilde\Gamma}[\neq\frac12]$ with respect to $\omega=(f|[\cdot\,,\,\cdot])$.
Hence, $\omega$ restricts to a non-degenerate skewsymmetric form
on $\mf g^{\Gamma}[\frac12]\cap\mf g^{\tilde\Gamma}[\frac12]$.
Let $\bar{\mf l}$ be its Lagrangian subspace.
Then, the claim is obtained for the following choice:
$$
\mf l=\bar{\mf l}\oplus\big(\mf g^{\Gamma}[\frac12] \cap \mf g^{\tilde\Gamma}[1]\big)
\,,
$$
and
$$
\tilde{\mf l}=\bar{\mf l}\oplus\big(\mf g^{\tilde\Gamma}[\frac12] \cap \mf g^{\Gamma}[1]\big)
\,.
$$
Indeed, 
$\mf l\subset\mf g^{\Gamma}[\frac12]$ and $\tilde{\mf l}\subset\mf g^{\tilde\Gamma}[\frac12]$
are obviously isotropic,
and they have half the dimension of $\mf g^{\Gamma}[\frac12]$ 
and $\mf g^{\tilde\Gamma}[\frac12]$
respectively.
Hence, they are Lagrangian.
Moreover, \eqref{20180419:eq4} holds by construction.
\end{proof}

As a consequence of Lemma \ref{20180419:lem1}
and of Definition \ref{def:Walg},
we get that the $W$-algebra
$W(\mf g, f, \Gamma, \mf l)$ associated to the grading $\Gamma$ of $\mf g$
and the Lagrangian subspace $\mf l\subset\mf g^\Gamma[\frac12]$, 
and the $W$-algebra $W(\mf g, f, \tilde\Gamma, \tilde{\mf l})$
associated to the grading $\tilde\Gamma$ of $\mf g$
and the Lagrangian subspace $\tilde{\mf l}\subset\mf g^{\tilde\Gamma}[\frac12]$,
are the same:
\begin{equation}\label{20180523:eq2}
W(\mf g, f, \Gamma, \mf l)
=
W(\mf g, f, \tilde\Gamma, \tilde{\mf l})
= \Big( U(\mf g)/U(\mf g)\Span \{b - (f \vert b)\}_{b \in \mf n} \Big)^{\ad \mf n}
\,,
\end{equation}
where $\mf n\subset\mf g$ is the nilpotent subalgebra \eqref{20180419:eq4}
(cf. \eqref{eq:mn}).

Note that \eqref{20180523:eq2}
is an equality (hence an isomorphism) of algebras,
but NOT of filtered algebras 
(with respect to the Kazhdan filtrations $F^\Gamma$ and $F^{\tilde\Gamma}$
defined in \eqref{eq:kazW}).
Indeed, the gradings $\Gamma$ and $\tilde\Gamma$ are different,
and therefore the corresponding Kazhdan filtrations do not coincide.

%%%%%%%%%%%%%%%%%%%%%%%%%%%%%%%%%%%%%%%
\subsection{
The Gan-Ginzburg isomorphism
}
\label{sec:GG}

Recall the following result of Gan and Ginzburg:
\begin{theorem}\cite[Theorem 4.1]{GG02}\label{thm:GG}
Let $\Gamma$ be a good grading for $f\in\mf g$,
and let $\mf{l}_1\subset\mf{l}_2 \subseteq \mf{g}^\Gamma[\frac12] $ be isotropic subspaces
with respect to the bilinear form $ \omega $ in Section \ref{sec:2.1}.
Let also
$$
\mf m_1 = \mf l_1 \oplus \mf g^\Gamma[\geq 1]
\,\subset\,
\mf m_2 = \mf l_2 \oplus \mf g^\Gamma[\geq 1]
\,\,\text{ and }\,\,
\mf n_1 = \mf l_1^\perp \oplus \mf g^\Gamma[\geq 1]
\,\supset\,
\mf n_2 = \mf l_2^\perp \oplus \mf g^\Gamma[\geq 1]
\,,
$$
so that 
$W(\mf g,f,\Gamma,\mf l_i)
=\big(U(\mf g)/ U(\mf g) \Span \{ b - ( f \vert  b) \}_{b \in \mf m_i}\big)^{\ad\mf n_i}$, $i=1,2$.
Then, we have an isomorphism of the corresponding $W$-algebras,
preserving the Kazhdan filtrations,
\begin{equation}\label{20180419:eq2}
\rho^\Gamma_{\mf{l}_1,\mf l_2}\,:\,\, 
W(\mf g,f,\Gamma,\mf l_1) \overset{\sim}{\longrightarrow} W(\mf g,f,\Gamma,\mf l_2)
\,,
\end{equation}
obtained by the restriction of the natural quotient map
\begin{equation}\label{20180419:eq1}
U(\mf g)/ U(\mf g) \Span \{ b - ( f \vert  b) \}_{b \in \mf m_1} 
\twoheadrightarrow U(\mf g) / U(\mf g) \Span \{b - (f \vert  b) \}_{b \in \mf m_2}
\,.
\end{equation}
Consequently, given arbitrary isotropic subspaces 
$ \mf l_1,\mf l_2\subset\mf g^\Gamma[\frac12] $,
the corresponding $W$-algebras are isomorphic filtered algebras 
(with respect to the Kazhdan filtration):
\begin{equation}\label{20180419:eq2b}
\big(W(\mf g,f,\Gamma,\mf l_1),F^\Gamma\big) 
\,\simeq\,
\big(W(\mf g,f,\Gamma,\mf l_2),F^\Gamma\big)
\,.
\end{equation}
\end{theorem}

%%%%%%%%%%%%%%%%%%%%%%%%%%%%%%%%%%%%%%%
\subsection{
Isomorphism of $W$-algebras for different good gradings and isotropic subspaces
}
\label{sec:6.1}

As a corollary of Lemma \ref{20180419:lem1}, Theorem \ref{thm:GG},
and the observation behind \eqref{20180523:eq4}, we get
\begin{theorem}[{\cite[Thm.1]{BG05}}]
Given two arbitrary good gradings $\Gamma$ and $\tilde\Gamma$ for $f\in\mf g$,
and two isotropic subspaces $\mf l\in\mf g^\Gamma[\frac12]$
and $\tilde{\mf l}\in\mf g^{\tilde\Gamma}[\frac12]$,
there exists an algebra isomorphism
\begin{equation}\label{20180523:eq3}
\phi_{\Gamma,\mf l;\tilde\Gamma,\tilde{\mf l}}\,:\,\,
W(\mf g,f,\Gamma,\mf l) \,\stackrel{\sim}{\longrightarrow}\, W(\mf g,f,\tilde\Gamma,\tilde{\mf l})
\,.
\end{equation}
\end{theorem}
The isomorphism $\phi_{\Gamma,\mf l;\tilde\Gamma,\tilde{\mf l}}$ is constructed as follows.
Let, according to \eqref{20180523:eq4},
$$
\Gamma=\Gamma_0\,,\,\,
\Gamma_1\,,\dots\,,
\Gamma_K=\tilde{\Gamma}
\,,
$$
be a chain of good gradings joining $\Gamma$ and $\tilde\Gamma$
such that for each $i=0,\dots,K-1$, the gradings $\Gamma_i$ and $\Gamma_{i+1}$
are adjacent.
By Lemma \ref{20180419:lem1}, for each $i$
there exist Lagrangian subspaces 
$\mf l_i\subset\mf g^{\Gamma_i}[\frac12]$ 
and
$\tilde{\mf l}_i\subset\mf g^{\Gamma_{i+1}}[\frac12]$ 
such that
$\mf l_i\oplus\mf g^{\Gamma_i}[\geq1]=\tilde{\mf l}_i\oplus\mf g^{\Gamma_{i+1}}[\geq1]$.
In particular,
the corresponding $W$-algebras coincide:
\begin{equation}\label{20180525:eq1}
W(\mf g,f,\Gamma_i,\mf l_i) \,=\, W(\mf g,f,\Gamma_{i+1},\tilde{\mf l}_i)
\,.
\end{equation}
Moreover, by Theorem \ref{thm:GG},
we have isomorphisms of filtered algebras
$$
\big(W(\mf g,f,\Gamma_{i+1},\tilde{\mf l}_i),F^{\Gamma_{i+1}}\big)
\,\stackrel{\sim}{\longrightarrow}\,
\big(W(\mf g,f,\Gamma_{i+1},\mf l_{i+1}),F^{\Gamma_{i+1}}\big)
\,.
$$
Hence, we have a sequence of isomorphisms of algebras
\begin{equation}\label{20180523:eq5}
\begin{split}
& W(\mf g,f,\Gamma,\mf l)
\stackrel{\text{GG}}{\longrightarrow}
W(\mf g,f,\Gamma_0,\mf l_0)
\,=\,
W(\mf g,f,\Gamma_1,\tilde{\mf l}_0)
\stackrel{\text{GG}}{\longrightarrow}
W(\mf g,f,\Gamma_1,\mf l_1)
=\dots
\\
& \dots
=
W(\mf g,f,\Gamma_K,\tilde{\mf l}_{K-1})
\stackrel{\text{GG}}{\longrightarrow}
W(\mf g,f,\tilde\Gamma,\tilde{\mf l})
\,.
\end{split}
\end{equation}
This composition is the desired isomorphism $\phi_{\Gamma,\mf l;\tilde\Gamma,\tilde{\mf l}}$.

Note that, by the observation at the end of Section \ref{sec:6.0},
$\phi_{\Gamma,\mf l;\tilde\Gamma,\tilde{\mf l}}$ is NOT an isomorphism 
between the filtered algebras
$(W(\mf g,f,\Gamma,\mf l),F^\Gamma)$
and
$(W(\mf g,f,\tilde\Gamma,\tilde{\mf l}),F^{\tilde\Gamma})$.
In the next section we will develop a machinery to reconstruct how 
$\phi_{\Gamma,\mf l;\tilde\Gamma,\tilde{\mf l}}$
acts on the filtrations.

%%%%%%%%%%%%%%%%%%%%%%%%%%%%%%%%%%%%%%%
\subsection{
Semisimple gradings of $\mf g$.
}
\label{sec:6.KK}

\begin{definition}\label{def:grad}
Let $h$ be a semisimple element of $\End V$
(that, according to the basic assumption described at the beginning of Section \ref{sec:6},
we assume to be diagonal in the fixed basis of $V$).
The corresponding \emph{semisimple grading} $G_h$ of $\mf g$
is given by $\ad h$-eigenvalues:
$$
\mf g
=
\bigoplus_{k\in\mb F}
\mf g^{G_h}[k]
\,\,\text{ where }\,\,
\mf g^{G_h}[k]
=
\big\{
a\in\mf g\,\big|\,[h,a]=ka
\big\}
\,.
$$
Moreover, we say that $G$ is \emph{neutral} for the nilpotent element $f$,
if the diagonal matrix $h$ has constant eigenvalues on the rectangles of size $p_i\times r_i$
of any pyramid attached to $f$.
\end{definition}
Note that semisimple gradings of $\mf g$ form a vector space,
(isomorphic to the space of diagonal matrices in $\End V$),
and neutral gradings form a vector subspace.
Explicitly,
given two semisimple gradings $G_1$ and $G_2$
attached to the diagonal matrices $h_1$ and $h_2$,
their linear combination is ($c_1,c_2\in\mb F$)
\begin{equation}\label{20180524:eq5}
\mf g^{c_1 G_1+c_2 G_2}[k]
=
\bigoplus_{c_1k_1+c_2k_2=k}
\big(
\mf g^{G_1}[k_1]
\cap
\mf g^{G_2}[k_2]
\big)
\,.
\end{equation}
Obviously, this is the semisimple grading attached to $c_1h_1+c_2h_2$.
\begin{example}\label{ex:grad}
If $\Gamma$ is a good grading of $\mf g$
and we take $h=\frac12 h_\Gamma$ defined in Section \ref{sec:3.1}, 
the corresponding semisimple grading $G_h$ coincides with $\Gamma$.
Hence, semisimple gradings constitute a great generalization of good gradings.
Recall that the various good gradings for $f$
are obtained by shifting the rectangles of sizes $p_i\times r_i$ 
of a pyramid. 
It follows that,
given two good gradings $\Gamma_1$ and $\Gamma_2$,
their difference $\Gamma_1-\Gamma_2$ is neutral for $f$.
\end{example}

Clearly, a semisimple grading $G$ is a Lie algebra grading
(in fact it is an associative algebra grading, if we identify $\mf g\simeq\End V$).
It follows that it induces an associative algebra grading of the universal enveloping 
algebra $U(\mf g)$:
\begin{equation}\label{20180524:eq2}
U(\mf g)
=
\bigoplus_{k}
U(\mf g)^{G}[k]
\,,
\end{equation}
where 
$$
U(\mf g)^{G}[k]
=
\sum_{s=0}^\infty
\sum_{k_1+\dots+k_s=k}
\mf g^{G}[k_1]
\dots
\mf g^{G}[k_s]
\,.
$$
It is natural to ask whether \eqref{20180524:eq2} induces a grading of the $W$-algebras.
This is true for neutral gradings, 
and it will be proved in Theorem \ref{20180524:thm3} below.

The Kazhdan filtration $F^\Gamma$ and the $G$-grading
of $U(\mf g)$ are related by the following:
\begin{lemma}\label{20180524:lem}
Let $\Gamma_1,\Gamma_2$ be good gradings for $f\in\mf g$,
and let $k,\Delta\in\frac12\mb Z$.
We have
$$
F^{\Gamma_1}_\Delta U(\mf g)^{\Gamma_1-\Gamma_2}[k]
=
F^{\Gamma_2}_{\Delta+k} U(\mf g)^{\Gamma_1-\Gamma_2}[k]
\,.
$$
\end{lemma}
\begin{proof}
Since $F^{\Gamma_1}$ and $F^{\Gamma_2}$ are algebra filtrations of $U(\mf g)$,
and $\Gamma_1-\Gamma_2$ defines an algebra grading of $U(\mf g)$,
it is enough to prove the claim on the generating space $\mf g$.
On the other hand, we have
\begin{align*}
& F^{\Gamma_1}_\Delta\mf g^{\Gamma_1-\Gamma_2}[k]
=
\bigoplus_{1-k_1\leq\Delta,\,k_1-k_2=k}
\big(\mf g^{\Gamma_1}[k_1]\cap\mf g^{\Gamma_2}[k_2]\big)
\\
& =
\bigoplus_{1-k_2\leq\Delta+k,\,k_1-k_2=k}
\big(\mf g^{\Gamma_1}[k_1]\cap\mf g^{\Gamma_2}[k_2]\big)
=
F^{\Gamma_2}_{\Delta+k} \mf g^{\Gamma_1-\Gamma_2}[k]
\,.
\end{align*}
\end{proof}

%%%%%%%%%%%%%%%%%%%%%%%%%%%%%%%%%%%%%%%
\subsection{
Semisimple grading of the $W$-algebras
}
\label{sec:6.KKb}

Recall the definition \eqref{eq:kaz} of the Kazhdan filtration $(U(\mf g),F^\Gamma)$,
and the induced filtration \eqref{eq:kazW} of the $W$-algebra 
$(W(\mf g,f,\Gamma,\mf l),F^\Gamma)$.
As in \eqref{20180508:eq3}
we extend it to a filtration $F^\Gamma$ of 
$W(\mf g,f,\Gamma,\mf l)[z]\otimes\End V$
by letting
\begin{equation}\label{20180523:eq8}
\begin{split}
& F^\Gamma_\Delta \big(
W(\mf g,f,\Gamma,\mf l)[z]\otimes\End V
\big)
\\
& =
\sum_{\ell\in\mb Z_+}
\sum_{\substack{k\in\frac12\mb Z \\ |k|\leq p_1-1}}
z^\ell 
F^\Gamma_{\Delta+k-\ell}W(\mf g,f,\Gamma,\mf l)
\otimes(\End V)^\Gamma[k]
\,.
\end{split}
\end{equation}
In other words, we assign to $z$ Kazhdan degree equal to $1$,
and we let $U\in (\End V)^\Gamma[j]$ have Kazhdan degree $-j$.
We shall also consider at times 
the analogue extension of the Kazhdan filtration of $U(\mf g)$.

As for the filtration $F^\Gamma$ in \eqref{20180523:eq8}, 
we extend a semisimple grading $G=G_h$ 
\eqref{20180524:eq2} to an associative algebra grading 
of $U(\mf g)[z]\otimes\End V$,
defined by letting $z$ have $G$-degree $0$,
and 
$$
(\End V)^G[k]
=
\big\{
U\in\End V\,\big|\, HU-UH=kU
\big\}
\,.
$$
(Here, with the usual convention, 
$H$ is the same as $h$, viewed as an element of $\End V$.)
Hence,
\begin{equation}\label{20180524:eq3}
\big(
U(\mf g)[z]\otimes\End V
\big)^{G}[k]
=
\bigoplus_{i+j=k}
(U(\mf g)^G[i])[z]\otimes(\End V)^G[j]
\,.
\end{equation}

The analogue of Lemma \ref{20180524:lem} holds for the extended $\Gamma$-filtration
and $G$-grading as well:
\begin{lemma}\label{20180524:lemb}
Let $\Gamma_1,\Gamma_2$ be good gradings for $f\in\mf g$,
and let $k,\Delta\in\frac12\mb Z$.
We have
$$
F^{\Gamma_1}_\Delta 
\big(
U(\mf g)[z]\otimes\End V
\big)^{\Gamma_1-\Gamma_2}[k]
=
F^{\Gamma_2}_{\Delta+k} 
\big(
U(\mf g)[z]\otimes\End V
\big)^{\Gamma_1-\Gamma_2}[k]
\,.
$$
\end{lemma}
\begin{proof}
By \eqref{20180523:eq8} and \eqref{20180524:eq3} we have
\begin{align*}
& F^{\Gamma_1}_\Delta
\big(
U(\mf g)[z]\otimes\End V
\big)^{\Gamma_1-\Gamma_2}[k]
\\
& =
\sum_{\ell\in\mb Z_+}
\sum_{\substack{j\in\frac12\mb Z}}
\sum_{k_1+k_2=k}
z^\ell 
F^{\Gamma_1}_{\Delta+j-\ell}
U(\mf g)^{\Gamma_1-\Gamma_2}[k_1]
\otimes
\big(
(\End V)^{\Gamma_1}[j]
\cap
(\End V)^{\Gamma_1-\Gamma_2}[k_2]
\big)
\\
& =
\sum_{\ell\in\mb Z_+}
\sum_{\substack{j\in\frac12\mb Z}}
\sum_{k_1+k_2=k}
\!\!\!\!\!\!
z^\ell 
F^{\Gamma_2}_{\Delta\!+\!j\!-\!\ell\!+\!k_1}
U(\mf g)^{\Gamma_1\!-\!\Gamma_2}[k_1]
\!\otimes\!
\big(
(\End V)^{\Gamma_2}[j\!-\!k_2]
\cap
(\End V)^{\Gamma_1\!-\!\Gamma_2}[k_2]
\big)
\\
& =
F^{\Gamma_2}_{\Delta+k}
\big(
U(\mf g)[z]\otimes\End V
\big)^{\Gamma_1-\Gamma_2}[k]
\,.
\end{align*}
\end{proof}

\begin{lemma}\label{20180524:lem0}
Let $G$ be a semisimple grading of $\mf g$,
extended to a grading of $U(\mf g)[z]\otimes\End V$
as in \eqref{20180524:eq3}.
Then:
\begin{enumerate}[(a)]
\item
the identity map $\id_U\in\End V$ is homogeneous of $G$-degree $0$
for every subspace $U\subset V$ spanned by some basis elements of $V$;
\item
for every $j\in\frac12\mb Z$,
the operator $E_j\in U(\mf g)\otimes\End V$ 
defined in \eqref{eq:Ej} is homogeneous of $G$-degree $0$;
\item
the matrix $D\in\End V$ defined in \eqref{eq:D}
is homogeneous of $G$-degree $0$;
\item
assuming, moreover, that the semisimple grading $G$ is neutral for $f$,
$F$ and $F^t$ are both homogeneous of $G$-degree $0$
\end{enumerate}
\end{lemma}
\begin{proof}
Claim (a) is obvious.
For claim (b),
recall that $E_j=\sum_{i\in I_j}u_iU^i$,
and we can depict $u_i$ and $U^i$ as arrows of the pyramid
going in opposite directions.
In particular $u_i$ and $U^i$ have opposite $G$-degree,
and their product has $G$-degree $0$.
The same argument proves claim (c).
Finally, note that both $F$ and $F^t$ can be depicted
as a collection of arrows moving horizontally
within the rectangles of the pyramid.
Hence, claim (d) follows by the definition of neutral grading.
\end{proof}

\begin{proposition}\label{20180524:thm1}
For every semisimple neutral grading $G$ of $\mf g$
and every nilpotent element $f\in\mf g$,
the operator $\widetilde{W}(z)\in U(\mf g)[z]\otimes\End V$
defined by \eqref{eq:recW}
(with respect to the right aligned pyramid associated to $f$)
is homogeneous of (extended) $G$-degree $0$.
\end{proposition}
\begin{proof}
Note that, if $G$ is a semisimple neutral grading of $\mf g$,
then its restriction to $\mf g'$ is again semisimple and neutral.
Then, the claim follows, inductively,
by the recursive definition \eqref{eq:recW} of the operator $\widetilde{W}(z)$,
using Lemma \ref{20180524:lem0}.
\end{proof}

\begin{proposition}\label{20180524:thm2b}
Let $\Gamma$ be a good grading for $f\in\mf g$,
$\mf l$ be an isotropic subspace of $\mf g^\Gamma[\frac12]$
spanned by elementary matrices (with respect to the fixed basis of $V$),
and $G$ be a neutral semisimple grading of $\mf g$.
Then
the $G$-grading of $U(\mf g)$
induces a $G$-grading of the $W$-algebra:
$$
W(\mf g,f,\Gamma,\mf l)
=
\bigoplus_{k\in\frac12\mb Z}
W(\mf g,f,\Gamma,\mf l)^G[k]
\,.
$$
\end{proposition}
\begin{proof}
Recall the definition \eqref{eq:Walg} 
of the $W$-algebra:
$$
W(\mf g,f,\Gamma,\mf l)
=
\widetilde{W}/\mc I
\,,
$$
where
$$
\widetilde{W}
=
\Big\{
u\in U(\mf g)\,\big|\,(\ad a)(u)\in\mc I
\,\text{ for all }\, 
a\in\mf l^\perp\oplus\mf g^\Gamma[\geq1]
\Big\}
\,,
$$
is a subalgebra of $U(\mf g)$,
and
$$
\mc I
\,=\,
U(\mf g)\Span\big\{b-(f|b)\,|\,b\in\mf l\oplus\mf g^\Gamma[\geq1]\big\}
\,,
$$
is its algebra ideal.
First, observe that the space $\mc I$
is compatible with any semisimple neutral grading $G$ of $U(\mf g)$.
Indeed, if we choose a basis $\{u_i\}$ of $\mf l\oplus\mf g^\Gamma[\geq1]$
consisting of elementary matrices,
then $u_i-(f|u_i)$ is homogeneous with respect to the neutral grading $G$.

Note that, since the $G$-grading is defined by $\ad h$-eigenvalues
for a semisimple element $h$, 
$(\mf g^G[i]|\mf g^G[j])=0$ unless $i+j=0$.
Since, by assumption, $\mf l$ is compatible with the $G$-grading,
and since, by the neutrality assumption on $G$, $f$ is homogeneous of $G$-degree $0$,
it follows that $\mf l^\perp=\{a\in\mf g^\Gamma[\frac12]\,|\,(a|[f,\mf l])=0\}$
is compatible with the $G$-grading as well.
It follows, that $\widetilde{W}\subset U(\mf g)$ 
decomposes as direct sum of $G$-graded subspaces.
The claim follows.
\end{proof}

\begin{proposition}\label{20180524:thm2}
Let $G$ be a semisimple grading of $\mf g$ neutral for $f$.
Let $\Gamma$, $\tilde \Gamma$ be good gradings for $f$,
and $\mf l\subset\mf g^{\Gamma}[\frac12]$, $\tilde{\mf l}\subset\mf g^{\tilde\Gamma}[\frac12]$ 
be isotropic subspaces.
The corresponding isomorphism 
$\phi_{\Gamma,\mf l;\tilde\Gamma,\tilde{\mf l}}$
defined in \eqref{20180523:eq3} preserves the $G$-grading:
$$
\phi_{\Gamma,\mf l;\tilde\Gamma,\tilde{\mf l}}
\,:\,\,
W(\mf g,f,\Gamma,\mf l)^G[k]
\,\stackrel{\sim}{\longrightarrow}\,
W(\mf g,f,\tilde\Gamma,\tilde{\mf l})^G[k]
\,,
$$
for every degree $k$.
\end{proposition}
\begin{proof}
First, consider the case $\Gamma=\tilde\Gamma$ and $\mf l\subset\tilde{\mf l}$.
In this case the isomorphism $\phi_{\Gamma,\mf l;\tilde\Gamma,\tilde{\mf l}}$
coincides with the Gan-Ginzburg isomorphism $\rho^\Gamma_{\mf l,\tilde{\mf l}}$
defined in \eqref{20180419:eq2}.
On the other hand, the Gan-Ginzburg isomorphism
is induced by the quotient map
$U(\mf g)/ U(\mf g) \Span \{ b - ( f \vert  b) \}_{b \in \mf m} 
\twoheadrightarrow U(\mf g) / U(\mf g) \Span \{b - (f \vert  b) \}_{b \in \tilde{\mf m}}$,
which is induced by the identity map on $U(\mf g)$.
Since the identity on $U(\mf g)$ preserves the $G$-grading
and, by Proposition \ref{20180524:thm2b}, 
the $G$-gradings of $W(\mf g,f,\Gamma,\mf l)$
and $W(\mf g,f,\Gamma,\tilde{\mf l})$ are induced by that of $U(\mf g)$,
it follows that $\rho^\Gamma_{\mf l,\tilde{\mf l}}$ preserves the $G$-grading.

The claim in the general case follows by the above special case
and the fact that, by construction \eqref{20180523:eq5},
the isomorphism $\phi_{\Gamma,\mf l;\tilde\Gamma,\tilde{\mf l}}$
is a composition of identity maps and Gan-Ginzburg isomorphisms.
\end{proof}
\begin{corollary}\label{20180529:cor}
Let $\Gamma$, $\tilde \Gamma$ be good gradings for $f$,
and $\mf l\subset\mf g^{\Gamma}[\frac12]$, $\tilde{\mf l}\subset\mf g^{\tilde\Gamma}[\frac12]$ 
be isotropic subspaces.
Consider the isomorphism 
$\phi_{\Gamma,\mf l;\tilde\Gamma,\tilde{\mf l}}$
defined in \eqref{20180523:eq3},
extended to an isomorphism
$$
W(\mf g,f,\Gamma,\mf l)[z]\otimes\End V
\stackrel{\sim}{\longrightarrow}
W(\mf g,f,\tilde\Gamma,\tilde{\mf l})[z]\otimes\End V
\,,
$$
commuting with $z$ and acting only on the first factor of the tensor product.
Let $A(z)\in F^\Gamma_\Delta\big(W(\mf g,f,\Gamma,\mf l)[z]\otimes\End V\big)$
be of $G$-degree $0$ for every semisimple neutral grading $G$.
Then
$$
\phi_{\Gamma,\mf l;\tilde\Gamma,\tilde{\mf l}}(A(z))
\,\in\,
F^{\tilde{\Gamma}}_\Delta
\big(
W(\mf g,f,\tilde\Gamma,\tilde{\mf l})
[z]\otimes\End V
\big)
\,.
$$
\end{corollary}
\begin{proof}
Recall that the isomorphism $\phi_{\Gamma,\mf l;\tilde\Gamma,\tilde{\mf l}}$
is defined by a chain of identity maps
and Gan-Ginzburg maps as in \eqref{20180523:eq5}.
By \eqref{20180525:eq1}
we have
$W(\mf g,f,\Gamma,\mf l_0)=W(\mf g,f,\Gamma_1,\tilde{\mf l}_0)$
for some Lagrangian subspaces $\mf l_0\subset\mf g^{\Gamma}[\frac12]$
and $\tilde{\mf l}_0\subset\mf g^{\Gamma_1}[\frac12]$.
By Theorem \ref{thm:GG} and Proposition \ref{20180524:thm2},
and the assumption on $A(z)$, we have
\begin{align*}
\rho^\Gamma_{\mf l,\mf l_0}(A(z))
& \,\in\,
F^{\Gamma}_\Delta
\big(
W(\mf g,f,\Gamma,\mf l_0)[z]\otimes\End V
\big)^G[0]
\\
& =
F^{\Gamma}_\Delta
\big(
W(\mf g,f,\Gamma_1,\tilde{\mf l}_0)[z]\otimes\End V
\big)^G[0]
\,,
\end{align*}
for every semisimple neutral $G$.
Recall by Example \ref{ex:grad} that the semisimple grading $G=\Gamma-\Gamma_1$ is 
neutral for $f$.
Hence, by Lemma \ref{20180524:lemb}, we also have
$$
\rho^\Gamma_{\mf l,\mf l_0}(A(z))
\,\in\,
F^{\Gamma_1}_\Delta
\big(
W(\mf g,f,\Gamma_1,\tilde{\mf l}_0)[z]\otimes\End V
\big)^G[0]
\,,
$$
again for every semisimple neutral $G$.
Repeating the same argument,
by \eqref{20180525:eq1}
there exist Lagrangian subspaces $\mf l_1\subset\mf g^{\Gamma_1}[\frac12]$
and $\tilde{\mf l}_1\subset\mf g^{\Gamma_2}[\frac12]$
such that
$W(\mf g,f,\Gamma_1,\mf l_1)=W(\mf g,f,\Gamma_2,\tilde{\mf l}_1)$,
and, by Theorem \ref{thm:GG}, Proposition \ref{20180524:thm2},
and Lemma \ref{20180524:lemb},
we have
\begin{align*}
\rho^{\Gamma_1}_{\tilde{\mf l}_0,\mf l_1}
\rho^\Gamma_{\mf l,\mf l_0}
(A(z))
& \,\in\,
F^{\Gamma_1}_\Delta
\big(
W(\mf g,f,\Gamma_1,\mf l_1)[z]\otimes\End V
\big)^G[0]
\\
& =
F^{\Gamma_1}_\Delta
\big(
W(\mf g,f,\Gamma_2,\tilde{\mf l}_1)[z]\otimes\End V
\big)^G[0]
\\
& =
F^{\Gamma_2}_\Delta
\big(
W(\mf g,f,\Gamma_2,\tilde{\mf l}_1)[z]\otimes\End V
\big)^G[0]
\,,
\end{align*}
for every semisimple neutral $G$.
Proceeding by induction, we get the claim.
\end{proof}
\begin{theorem}\label{20180524:thm3}
Given $f$,
let $\Gamma_R$ be the good grading associated to the right aligned pyramid,
$\Gamma$ be an arbitrary good grading for $f$,
and $\mf l$ be an isotropic subspace of $\mf g^\Gamma[\frac12]$,
and consider the algebra isomorphism $\phi_{\Gamma_R,0;\Gamma,{\mf l}}$
defined in \eqref{20180523:eq3}.
We have
\begin{equation}\label{20180524:eq4}
W^{\Gamma,\mf l}(z)
:=
\phi_{\Gamma_R,0;\Gamma,{\mf l}}
\big(W(z))\big)
\,\in\,
F^\Gamma_1\big(
W(\mf g,f,\Gamma,\mf l)[z]\otimes\Hom(V_-,V_+)
\big)^G[0]
\,,
\end{equation}
for every semisimple grading $G$ neutral for $f$.
In \eqref{20180524:eq4} the map $\phi_{\Gamma_R,0;\Gamma,{\mf l}}$
is extended to $W(\mf g,f,\Gamma_R,0)[z]\otimes\Hom(V_-,V_+)$
by acting only on the first factor of the tensor product, and commuting with $z$.
\end{theorem}
\begin{proof}
By Theorem \ref{thm:20180208}, condition \eqref{20180508:eq4}
and Propositions \ref{20180524:thm1} and \ref{20180524:thm2b}, 
we have
$$
W(z)\,\in\,F^{\Gamma_R}_1
\big(
W(\mf g,f,\Gamma_R,0)[z]\otimes\Hom(V_-,V_+)
\big)^G[0]
\,,
$$
for every semisimple grading $G$ neutral for $f$.
Hence, the statement is a special case of Corollary \ref{20180529:cor}.
\end{proof}

%%%%%%%%%%%%%%%%%%%%%%%%%%%%%%%%%%%%%%%
\subsection{Definition of the Premet map}
\label{sec:6.2}

Recall the definition of the subspace $U^\perp$ in \eqref{20180511:eq1}:
\begin{equation}\label{20180511:eq1c}
U^\perp
=
\id_{F^tV}\End V\oplus\id_{V_-}(\End V)^{\Gamma_R}[>0]
\,\subset\mf g
\,,
\end{equation}
where $\Gamma_R$ is the good grading of $f$ associated to the right aligned pyramid.
\begin{lemma}\label{20180529:lem}
For every good grading $\Gamma$, we have
$\mf g^\Gamma[\geq\frac12]\subset U^\perp$.
\end{lemma}
\begin{proof}
When going from a pyramid to the corresponding right aligned pyramid,
an arrow going upright goes even more to the right
(i.e. the corresponding element of $\mf g$ increases the degree),
while an arrow going downright has certainly image in $F^tV$.
\end{proof}

Recall the operators $Z(z)\in U(\mf g^f)[z]\otimes\Hom(V_-,V_+)$ in \eqref{eq:Z}
and $W^{\Gamma,\mf l}(z)\in W(\mf g,f,\Gamma,\mf l)[z]\otimes\Hom(V_-,V_+)$ 
defined in Theorem \ref{20180524:thm3}.
\begin{proposition}\label{thm:premetmap}
Let $\Gamma$ be a good grading for $f\in\mf g$
and $\mf l\subset\mf g^\Gamma[\frac12]$ be an isotropic subspace.
The formula
\begin{equation}\label{20180523:eq9}
w^{\Gamma,\mf l}(Z(z))
=
W^{\Gamma,\mf l}(z)
\,,
\end{equation}
defines uniquely a linear map
$w^{\Gamma,\mf l}:\,\mf g^f\to W(\mf g,f,\Gamma,\mf l)$.
Then, the map $w^{\Gamma,\mf l}$
satisfies Premet's conditions (i) and (ii) of Definition \ref{def:premet}
with respect to the subspace $U^\perp$ in \eqref{20180511:eq1c}.
\end{proposition}
\begin{proof}
The same argument as in the proof of Proposition \ref{20180405:prop2}
shows that equation \eqref{20180523:eq9}
defines uniquely a linear map $w^{\Gamma,\mf l}:\,\mf g^f\to W(\mf g,f,\Gamma,\mf l)$.
Indeed, $Z(z)$ can be expanded as \eqref{20180507:eq11} and,
by Theorem \ref{20180524:thm3}, we have
$$
\begin{array}{l}
\displaystyle{
\vphantom{\Big(}
W^{\Gamma,\mf l}(z)
=
z\id_{V_+}(1+zF^t)^{-1}\id_{V_-}
} \\
\displaystyle{
\vphantom{\Big(}
+\!\
\sum_{h,k=0}^{p_1-1}
\sum_{i\in\mc F(h,k)}
\!\!\!
\sum_{\ell=0}^{\min\{h,k\}}
(-z)^\ell
\phi_{\Gamma_R,0;\Gamma,{\mf l}}(\widetilde{w}_{i,\ell}\bar 1) U^i
\,\in W(\mf g,f,\Gamma,\mf l)[z]\otimes\Hom(V_-,V_+)
\,.}
\end{array}
$$
The first Premet condition can be restated, as we did in \eqref{20180508:eq4},
as
$$
W^{\Gamma,\mf l}(z)
\,\in\, 
F^{\Gamma}_1
\big(
W(\mf g,f,\Gamma,\mf l)[z]\otimes\Hom(V_-,V_+)
\big)
\,,
$$
which holds by Theorem \ref{20180524:thm3}.
We are left to prove the second Premet condition,
which can be stated, as we did in \eqref{20180510:eq1}, 
as
\begin{equation}\label{20180510:eq1b}
\eta^f(\gr^{\Gamma}_1 W^{\Gamma,\mf l}(z)) = Z(z)
\,.
\end{equation}
By Theorem \ref{20180524:thm3} all the operators $W^{\Gamma,\mf l}(z)$
lie in $G$-degree $0$ for every $(\Gamma,\mf l)$ and every semisimple neutral grading $G$.
On the other hand, by Corollary \ref{20180529:cor}
the map $\phi_{\Gamma_R,0;\Gamma,{\mf l}}$ preserves 
the Kazhdan filtrations
when restricted to elements of $G$-degree $0$ for every semisimple neutral grading $G$.
Hence, it induces a map on the associated graded spaces
preserving the Kazhdan degree.
The LHS of \eqref{20180510:eq1b} is then equal to
\begin{equation}\label{20180510:eq1c}
\eta^f(\gr^{\Gamma}_1 W^{\Gamma,\mf l}(z)) 
=
\eta^f(\gr^{\Gamma}_1 \phi_{\Gamma_R,0;\Gamma,{\mf l}} W(z)) 
=
\eta^f(\phi_{\Gamma_R,0;\Gamma,{\mf l}} \gr^{\Gamma_R}_1 W(z)) 
\,.
\end{equation}
Recall that the isomorphism $\phi_{\Gamma_R,0;\Gamma,{\mf l}}$
is defined by a chain of identity maps
and Gan-Ginzburg maps as in \eqref{20180523:eq5}:
\begin{align*}
& \phi_{\Gamma_R,0;\Gamma,{\mf l}}
=
\rho^{\Gamma_K}_{\tilde{\mf l}_{K-1},\mf l_K}
\circ
\dots
\circ
\rho^{\Gamma_2}_{\tilde{\mf l}_1,\mf l_2}
\circ
\rho^{\Gamma_1}_{\tilde{\mf l}_0,\mf l_1}
\\
& =
\rho^{\Gamma_K}_{0,\mf l_K}
\circ
(\rho^{\Gamma_K}_{0,\tilde{\mf l}_{K-1}})^{-1}
\circ
\dots
\circ
\rho^{\Gamma_2}_{0,\mf l_2}
\circ
(\rho^{\Gamma_2}_{0,\tilde{\mf l}_1})^{-1}
\circ
\rho^{\Gamma_1}_{0,\mf l_1}
\circ
(\rho^{\Gamma_1}_{0,\tilde{\mf l}_0})^{-1}
\,.
\end{align*}
Recall also that each Gan-Ginzburg map $\rho^{\tilde\Gamma}_{0,\tilde{\mf l}}$
is induced by the quotient map with kernel generated by the elements
$b-(f|b)$, with $b\in\tilde{\mf l}+\mf g^{\tilde\Gamma}[\geq1]$.
By definition, the map $\eta^f$ in \eqref{diag:commut-pi_gf}
is induced by the quotient map with kernel generated by the elements
$b-(f|b)$, with $b\in U^\perp$.
Hence, by Lemma \ref{20180529:lem}, we have 
$\eta^f\circ \rho^{\tilde{\Gamma}}_{0,\tilde{\mf l}}=\eta^f$
for all good gradings $\tilde\Gamma$ and all isotropic subspaces $\tilde{\mf l}$,
and equation \eqref{20180510:eq1c} becomes
\begin{equation}\label{20180510:eq1d}
\eta^f(\gr^{\Gamma}_1 W^{\Gamma,\mf l}(z)) 
=
\eta^f(\gr^{\Gamma_R}_1 W(z)) 
\,.
\end{equation}
The claim then follows by \eqref{20180510:eq1}.
\end{proof}

%%%%%%%%%%%%%%%%%%%%%%%%%%%%%%%%%%%%%%%
\subsection{Proof of Equation \eqref{eq:main} (and of Theorem \ref{thm:main})}
\label{sec:6.4}

Let $\Gamma$ be a good grading for $f\in\mf g$ 
and $\mf l$ be an isotropic subspace of $\mf g^\Gamma[\frac12]$.
Recall the definition \eqref{eq:L} of the operator $L(z)$,
which we now denote with superscript $(\Gamma,\mf l)$
to remind its dependence on the grading $\Gamma$ and the isotropic subspace $\mf l$:
\begin{equation}\label{20180529:eq1}
L^{\Gamma,\mf l}(z)
=
|z\id_V+F+E_{\mf p}+D_{\mf m}|_{V_-^d,V_+^d}\bar 1_{\mf m}
\,\in W(\mf g,f,\Gamma,\mf l)((z^{-1}))\otimes\Hom(V_-^d,V_+^d)
\,,
\end{equation}
where $\mf m=\mf l\oplus\mf g^\Gamma[\geq1]$,
$\mf p\subset\mf g^\Gamma[\leq\frac12]$ is complementary to $\mf m$,
$E_{\mf p}$ is as in \eqref{eq:E}, $D_{\mf m}$ is as in \eqref{eq:D},
and $\bar 1_{\mf m}$ is the image of $1$ in $U(\mf g)/U(\mf g)\langle b-(f|b)\rangle_{b\in\mf m}$. 
Recall also from \cite{DSKV16,DSKV17} (see also \cite{Fed18}) that
\begin{equation}\label{20180529:eq2}
L^{\Gamma,\mf l}(z)
=
z^{\frac{p_1+1}{2}}|\id_V+z^{\Gamma-1}E|_{V_-^d,V_+^d}\bar 1_{\mf m}
\,,
\end{equation}
where $z^{\Gamma-1}$ is the automorphism of $U(\mf g)[z^{\pm\frac12}]$
mapping $a\mapsto z^{i-1}a$ for $a\in\mf g^\Gamma[i]$.
Equation \eqref{eq:main}, which we plan to prove, reads
\begin{equation}\label{20180529:eq3}
|W^{\Gamma,\mf l}(z)|_{V_-^d,V_+^d}
=
L^{\Gamma,\mf l}(z)
\,,
\end{equation}
where $W^{\Gamma,\mf l}(z)$ is defined by \eqref{20180524:eq4}.
Since 
$\phi_{\Gamma_R,0;\Gamma,{\mf l}}
:\, W(\mf g,f,\Gamma_R,0)\to W(\mf g,f,\Gamma,\mf l)$
is an isomorphism of algebras,
it commutes with taking quasideterminants.
Hence,
$$
|W^{\Gamma,\mf l}(z)|_{V_-^d,V_+^d}
=
\phi_{\Gamma_R,0;\Gamma,{\mf l}}
|W^{\Gamma_R,0}(z)|_{V_-^d,V_+^d}
=
\phi_{\Gamma_R,0;\Gamma,{\mf l}}
\big(L^{\Gamma_R,0}(z)\big)
\,,
$$
by \eqref{20180516:eq2}.
Hence, equation \eqref{20180529:eq3}
can be rewritten as
\begin{equation}\label{20180529:eq4}
\phi_{\Gamma_R,0;\Gamma,{\mf l}}
\big(L^{\Gamma_R,0}(z)\big)
=
L^{\Gamma,\mf l}(z)
\,,
\end{equation}
which we are left to prove.

Recall that the isomorphism $\phi_{\Gamma_R,0;\Gamma,{\mf l}}$
is defined by a chain of identity maps
and Gan-Ginzburg isomorphisms as in \eqref{20180523:eq5}:
%\begin{equation}\label{20180523:eq5}
\begin{align*}
& W(\mf g,f,\Gamma_R,0)
\,=\,
W(\mf g,f,\Gamma_1,\tilde{\mf l}_0)
\stackrel{\rho^{\Gamma_1}_{\tilde{\mf l}_0,\mf l_1}}{\longrightarrow}
W(\mf g,f,\Gamma_1,\mf l_1)
=\dots
\\
& \dots
=
W(\mf g,f,\Gamma_K,\tilde{\mf l}_{K-1})
\stackrel{\rho^{\Gamma_K}_{\tilde{\mf l}_{K-1},\mf l_K}}{\longrightarrow}
W(\mf g,f,\Gamma_K,\mf l_K)
\,,
\end{align*}
%\end{equation}
where $\Gamma_K=\Gamma$ and $\mf l_K=\mf l$.
First, note that if $(\Gamma,\mf l)$ and $(\tilde\Gamma,\tilde{\mf l})$
satisfy \eqref{20180419:eq4} (i.e. $\mf m=\tilde{\mf m}$, and hence $\mf p=\tilde{\mf p}$),
we obviously have 
$$
L^{\Gamma,\mf l}(z)=L^{\tilde\Gamma,\tilde{\mf l}}(z)
$$
by the very definition \eqref{20180529:eq1}.
On the other hand, given two isotropic subspaces $\mf l_1,\mf l_2$ of $\mf g^\Gamma[\frac12]$,
we also have, by \eqref{20180529:eq2}
$$
\rho^\Gamma_{\mf l_1,\mf l_2}\big(L^{\Gamma,\mf l_1}(z)\big)
=
L^{\Gamma,\mf l_2}(z)
\,,
$$
since the Gan-Ginzburg isomorphism is induced by the map 
$\bar 1_{\mf m_1}\mapsto\bar 1_{\mf m_2}$
(cf. \eqref{20180419:eq2}-\eqref{20180419:eq1}).
The claim \eqref{20180529:eq4} follows.

%%%%%%%%%%%%%%%%%%%%%%%%%%%%%%%%%%%%%%%
\section{An example}\label{sec:8}

In this section we show how to compute the PBW generating set constructed in Proposition \ref{20180405:prop2} in the case of
the partition $\lambda=2^{p}1^{q}$, $p\geq1$ and $q\geq0$.

Recall that, in this case, the $\Gamma$-grading of $\mf g$ is as in Figure \ref{fig11}, $\dim(V_-^d)=p$
and $\dim(V_+)=p+q=r$. Moreover, the $\Gamma$-grading of $\mf g$ is $\mf g=\mf g[-1]\oplus\mf g[0]\oplus\mf g[1]$,
where the homogeneous subspaces are described in \eqref{20180219:eq1}.

It follows by a straightforward computation,
using the recursion \eqref{eq:recW}, equation \eqref{20180601:eq1} and the fact that $D|_{V_-^d}=-r$
(see the proof of Theorem \ref{thm:20180208}), that
\begin{equation}
\begin{split}\label{Wtilde}
\widetilde W(z)=&(z+E_0)\id_{V_-^u}+E_{-1}-(z+E_0)F^t(z-r+E_0)\id_{V_-^d}
\\
&+(z+E_0)\id_{V_-^u}E_0F^t\id_{V_-^d}
\,.
\end{split}
\end{equation}
Recall that $W(z)=\widetilde W(z)\bar1$.
Let $A\in\mf g_0^f=\Hom (V_+,V_-)$.
Then, by equation \eqref{eq:W}, we have
\begin{equation}\label{20180601:eq2}
(\widetilde W(z)|A)\bar1=(W(z)|A)=z\tr((1+zF^t)^{-1}A)+w(\phi_z(a))
\,.
\end{equation}
We use equations \eqref{Wtilde} and \eqref{20180601:eq2} to compute $w(\phi_z(a))$ for every $a\in\mf g_0^f$.
First, note that, in the notation of \eqref{20180502:eq2} 
and Lemma \ref{lem:gfk}c), we have
$$
\mf g^f=\bigoplus_{h,k=0}^1\bigoplus_{\ell=0}^{\min\{h,k\}}\phi_\ell(\mf g^f_0(h,k))
\,,
$$
where (see equation \eqref{eq:phik}), for $A\in\mf g_0^f$,
\begin{equation}\label{20180601:eq3}
\phi_0(A)=A
\qquad
\text{and}
\qquad
\phi_1(A)=F^tA\id_{F^tV_-^d}+\id_{V_-^d}AF^t
\,.
\end{equation}
By equation \eqref{eq:V_-ud} it is clear that
\begin{equation}\label{20180601:eq4}
\mf g_0^f=\Hom(V_+,V_-^d)\oplus\Hom(V_+,V_-^u)
\,.
\end{equation}
Let $A\in\Hom(V_+,V_-^u)$. By equation \eqref{Wtilde} we have
$$
(\widetilde W(z)|A)=(z+E_0|A)=z\tr(A)+a=z\tr(A)+\phi_0(a)
\,.
$$
In the second identity we used the completeness relation \eqref{20180215:eq1} (with $a$ and $A$ exchanged), and in the
last identity we used the definition of $\phi_0$ given in \eqref{20180601:eq3}.
Hence, by equation \eqref{20180601:eq2} we have ($A\in\Hom(V_+,V_-^u)$)
\begin{equation}\label{speriamoultimo1}
w(\phi_z(a))=w(\phi_0(a))=w(a)=\phi_0(a)\bar1=a\bar1
\,.
\end{equation}
Next, let $A\in\Hom(V_+,V_-^d)$.
By equation \eqref{Wtilde} we have
\begin{equation}\label{20180601:eq5}
(\widetilde W(z)|A)
=\big(E_{-1}-(z+E_0)F^t(z-r+E_0)+(z+E_0)\id_{V_-^u}E_0F^t
\big|A\big)
\,.
\end{equation}
The RHS of \eqref{20180601:eq5} is a polynomial of order $2$ in $z$ with coefficients in $U(\mf g[\leq0])$. Let us compute
each contribution from the powers of $z$ separately.
It is immediate to check that the coefficient of $z^2$ in the RHS of \eqref{20180601:eq5} is
\begin{equation}\label{20180601:eq6}
-(F^t|A)=-\tr(F^tA)
\,.
\end{equation}
On the other hand, the coefficient of $z$ in the RHS of \eqref{20180601:eq5} is
\begin{equation}
\begin{split}
\label{20180601:eq7}&\big(\id_{V_-^u}E_0F^t-F^t(E_0-r)-E_0F^t\big|A\big)
\\
&=
\big(E_0\big| F^tA(\id_{V_-^u}-\id_{V_+})-A F^t\big)+r\tr(F^tA)
\\
&=-(E_0\big|\phi_1(A)\big)+r\tr(F^tA)
=-\phi_1(a)+r\tr(F^tA)
\,.
\end{split}
\end{equation}
In the first equality we used the cyclic invariance of the trace form, in the second equality we used the decomposition of $V_+$ provided
by Figure \ref{fig11} and the definition of $\phi_1$ given by \eqref{20180601:eq3}, and in the last equality we used
the completeness relation \eqref{20180215:eq1} (with $a$ and $A$ exchanged).
Finally, let us compute the constant term in $z$ in the RHS of \eqref{Wtilde}.
It is
\begin{equation}\label{20180601:eq8}
\big(E_{-1}\big|A\big)-\big(E_0F^t(E_0-r)\id_{V_-^d}\big|A\big)+\big(E_0\id_{V_-^u}E_0F^t\id_{V_-^d}\big|A\big)
\,.
\end{equation}
Let us compute the three summands in \eqref{20180601:eq7} separately.
Using the completeness relation \eqref{20180215:eq1} (with $a$ and $A$ exchanged) and the definition of $\phi_0$ given in
\eqref{20180601:eq3} we have
\begin{equation}\label{20180601:eq8a}
\big(E_{-1}\big|A\big)=a=\phi_0(a)
\,.
\end{equation}
Using the invariance of the trace form and equation \eqref{eq:Ej} the second and third summands
in \eqref{20180601:eq7} can be rewritten as
\begin{equation}\label{20180601:eq8b}
r\big(E_0\big|F^tA\big)
-\big(E_0\big|F^tE_0A\big)
=r\varphi(F^tA)
-\sum_{i\in I_0}u_i\varphi(AU^iF^t)
\,.
\end{equation}
and
\begin{equation}\label{20180601:eq8c}
\big(E_0\big|\id_{V_-^u}E_0F^tA\big)
=\sum_{i\in I_0}u_i\varphi(F^tAU^i\id_{V_-^u})
\,.
\end{equation}
In equations \eqref{20180601:eq8b} and \eqref{20180601:eq8c} we use the following notation: given $A,B\in\End (V)$,
we denote by $\varphi(AB)\in\mf g$ the element corresponding to the endomorphism $AB$.
Combining equations \eqref{20180601:eq6}, \eqref{20180601:eq7}, \eqref{20180601:eq8a}, \eqref{20180601:eq8b}
and \eqref{20180601:eq8c} we get ($A\in\Hom(V_+,V_-^{d})$)
\begin{align*}
&(\widetilde W(z)|A)=-\tr(F^tA)z^2-(\phi_1(a)-r\tr(F^tA))z
\\
&+\phi_0(a)
+r\varphi(F^tA)
-\sum_{i\in I_0}u_i\varphi(AU^iF^t)
+\sum_{i\in I_0}u_i\varphi(F^tAU^i\id_{V_-^u})
\,.
\end{align*}
Hence, using equation \eqref{20180601:eq2} we get ($A\in\Hom(V_+,V_-^d)$)
\begin{equation}
\begin{split}\label{speriamoultimo2}
&w(\phi_z(a))=-z(\phi_1(a)-r\tr(F^tA))\bar1
+\phi_0(a)\bar1+r\varphi(F^tA)\bar1
\\
&
-\sum_{i\in I_0}u_i\varphi(AU^iF^t)\bar1
+\sum_{i\in I_0}u_i\varphi(F^tAU^i\id_{V_-^u})u_i\bar1
\,.
\end{split}
\end{equation}

We can give an even more explicit form of the generators $w(\phi_z(a))$ by fixing a basis of $\mf{gl}_N$.
Let us number the boxes of the pyramid from $1$ to $N=2p+q$ going first from bottom to top and then from right to left and depict
an elementary matrix $e_{ij}\in\mf{gl}_N$ as an arrow going from box $j$ to box $i$. Then, the entries $\widetilde{W}_{ij}(z)$,
$i,j=1,\dots, r$, of the matrix $\widetilde W(z)$ given by equation \eqref{Wtilde} are the following (recall also \eqref{eq:intro2}):
$$
\widetilde{W}_{ij}(z)=
\left\{
\begin{array}{ll}
\delta_{ij}z+e_{ji}\,,& j>p\,,
\\
e_{j+r,i}-\sum_{h=1}^p(\delta_{ih}z+e_{hi})(\delta_{hj}(z-r)+e_{j+r,h+r})
&
\\
+\sum_{h=p+1}^r(\delta_{ih}z+e_{hi})e_{jh}\,,
&
j\leq p
\,.
\end{array}\right.
$$
Hence, we get
$$
w(\phi_z(e_{ji}))=e_{ji}\bar1\,,
\qquad
i=1,\dots,r\,,j=p+1,\dots,r
\,,
$$
and
\begin{align*}
&w(\phi_z(e_{j+r,i}))=-z\delta_{i\leq p}(e_{ji}+e_{j+r,i+r}-\delta_{ij}r)\bar1+e_{j+r,i}\bar1+re_{ji}\bar1
\\
&-\sum_{h=1}^pe_{hi}e_{j+r,h+r}\bar1+\sum_{h=p+1}^re_{hi}e_{jh}\bar1\,,
\qquad
i=1,\dots,r\,,j=1,\dots,p
\,.
\end{align*}
The above formulas provides a componentwise descriptions of equations \eqref{speriamoultimo1} and \eqref{speriamoultimo2}.

%%%%%%%%%%%%%%%%%%%%%%%%%%%%%%%%%%%%%%%%%%%%%%%%%%%%%%%%%%%%%%%%%%%%%%%%%%%%%%%%%%%%%%%%%%%%%%%%%%%%%%%%%%%%%%
%%%%%%%%%%%%%%% Bibliography %%%%%%%%%%%%%%%%%%%%%%%%%%%%%%%%%%%%%%%%%%%%%%%%%%%%%%%%%%%%%%%%%%%%%%%%%%%%%%%%%
%%%%%%%%%%%%%%%%%%%%%%%%%%%%%%%%%%%%%%%%%%%%%%%%%%%%%%%%%%%%%%%%%%%%%%%%%%%%%%%%%%%%%%%%%%%%%%%%%%%%%%%%%%%%%%

% Non-BibTeX users please use

\end{document}